\documentclass[a4paper,12pt]{amsart}

\usepackage[T1]{fontenc}
\usepackage[utf8]{inputenc}
\usepackage{palatino}
\usepackage{amsmath, amssymb, mathtools}
\usepackage[usenames,dvipsnames]{xcolor}
\usepackage{graphicx}
\usepackage{subfigure}
\usepackage{minitoc}
\usepackage{tikz}
\usepackage{enumitem}
\usepackage{dsfont}

\usepackage[margin=0.96 in]{geometry}
\usepackage{bbm}
\usepackage[numbers,sort&compress]{natbib}
\usepackage[colorlinks=true]{hyperref}
\hypersetup{urlcolor=blue, citecolor=red}

\DeclarePairedDelimiter{\abs}{\lvert}{\rvert}
\DeclarePairedDelimiter{\norm}{\lVert}{\rVert}
\DeclarePairedDelimiter{\set}{\{}{\}}
\DeclareMathAlphabet{\mathup}{OT1}{\familydefault}{m}{n}
\newcommand{\dx}[1]{\mathop{}\!\mathup{d} #1}
\DeclarePairedDelimiter{\prt}{(}{)}
\DeclarePairedDelimiter{\brk}{[}{]}
\newcommand{\N}{{\mathbb N}}
\newcommand{\R}{{\mathbb R}}
\newcommand{\Rd}{{\mathbb R^d}}

\def\remspace{\kern-0.5em}
\def\alnspace{0.5em}

\theoremstyle{plain}
\newtheorem{theorem}{Theorem}[section]
\newtheorem{lemma}[theorem]{Lemma}
\newtheorem{proposition}[theorem]{Proposition}
\newtheorem{corollary}[theorem]{Corollary}

\theoremstyle{remark}
\newtheorem{assumption}{Assumption}[section]
\newtheorem{remark}[theorem]{\bf Remark}
\newtheorem{definition}[theorem]{\bf Definition}

\renewcommand{\i}{^{(i)}}
\newcommand{\1}{^{(1)}}
\newcommand{\2}{^{(2)}}

\newcommand{\sign}{\mathrm{sign}}
\newcommand{\ds}{\displaystyle}
\newcommand{\ddt}{\frac{\dx{}}{\dx{t}}}
\newcommand{\partialt}[1]{\frac{\partial #1}{\partial t}}

\newcommand{\fpartial}[1]{\frac{\partial}{\partial #1}}

\providecommand{\mathfs}[1]{\mathsf{#1}}
\newcommand{\intO}[1]{\int_{\R^d} #1 \dx{x}}
\newcommand{\iintO}[1]{\int_{\R^{2d}} #1 \dx{x} \dx{y}}
\newcommand{\eps}{\varepsilon}

\newcommand{\calK}{\mathcal{K}}

\title[Incompressible limit]{Incompressible limit for a two-species Brinkman model with drift}

\author{Tomasz D\k{e}biec$^{1}$}
\author{Valentin Vincent Neumann$^{2}$}
\author{Markus Schmidtchen$^{2}$}

\address{$^{1}$ Institute of Applied Mathematics and Mechanics, University of Warsaw, Banacha 2, 02-097 Warsaw, Poland. Email: t.debiec@mimuw.edu.pl.}
\address{$^{2}$ Institute of Scientific Computing, Faculty of Mathematics, TU Dresden, Zellescher Weg 12-14, 01069 Dresden, Germany. Emails: 
	valentin\_vincent.neumann@tu-dresden.de, markus.schmidtchen@tu-dresden.de.}

\begin{document}

\begin{abstract}
    We study a two-species model for tissue growth in which both populations are transported by an external drift and by a velocity potential determined through Brinkman's law. The pressure is generated by a stiff constitutive relation depending on the total density. Our main result establishes the incompressible limit as the stiffness exponent tends to infinity, in arbitrary space dimension and for merely integrable initial data. The limit system consists of the two balance laws coupled to Brinkman's equation, the hard-congestion constraint $0\leq n_\infty\leq 1$, the graph relation $p_\infty(1-n_\infty)=0$, and the corresponding complementarity relation. A key point of the analysis is a new $L^2$-based compactness theory that avoids both uniform $L^\infty$-bounds on the pressure and the kinetic reformulation used in earlier approaches. We first construct weak solutions for bounded data and then remove the boundedness assumption by means of a weighted compactness argument inspired by Bresch--Jabin. We also prove an Aubin--Lions--Simon type lemma based on oscillation control, yielding time continuity of the constructed solutions. Finally, a refined dissipation estimate for the Bresch--Jabin compactness functional gives strong compactness of the pressure in the stiff limit and implies a regularising effect: the limiting pressure is bounded even when the approximating pressures are only integrable.
\end{abstract}

\maketitle

\section{Introduction}
In this paper, we derive novel regularity results for systems of balance laws coupled through Brinkman's law, an elliptic equation, under rather minimal assumptions on growth and drift terms. 
The novel methodology allows us to revisit singular limits in which the pressure degenerates into a stiff graph relation 
$$
    0 \leq n_\infty \leq 1,\quad p_\infty \geq 0, \quad p_\infty(1 - n_\infty) =0,
$$
also referred to as stiff limit, incompressible limit or hard congestion limit, while removing technically challenging proofs of past approaches. 
In its seminal form, this problem was addressed in the paper \cite{PV2015},
in which the authors propose the viscoelastic tissue growth model
\begin{subequations}
\label{Brinkman-sys-seminal}
\begin{align}
    \label{eq:balance-law}
    \partialt n = \nabla \cdot (n \nabla W) + n G(p),
\end{align}
where $n=n(x,t)$ is the spatial concentration of cells at locations $x \in \Rd$ at time $t>0$. The velocity potential, $W$, satisfies the so-called Brinkman law
\begin{align}
    - \nu \Delta W + W = p,
\end{align}
\end{subequations}
and the pressure $p$ is coupled to the density through the constitutive law $p = p(n) = n^k$. In their work, they address the incompressible limit, $k \to \infty$, which leads to a free-boundary type problem, as we shall discuss below in more detail. \\

\textbf{Historical background.}
It is worth noting that the interest in the stiff limit $k \to \infty$ arose much earlier and dates back to the study of the porous medium equation $\partial_t n = \Delta n^k$, see \cite{Sab61, OKJ58, Vaz07}. The approaches proposed over the years include operator semigroup theory \cite{BC81}, formal asymptotics \cite{EHKO86}, obstacle problems \cite{CF87}, and more recently, variational (Wasserstein and Kantorovich-Fisher-Rao) gradient-flow approaches \cite{alexander2014quasi, CD2020}, and viscosity solutions \cite{kim2018porous, KPW2019}. In the case of the porous medium equation, letting the exponent $k\to \infty$ leads to infinitely fast diffusion in zones of densities larger than $1$, referred to as collapse, and zones of densities less than $1$, referred to as mushy regions. The contact set is characterised by a stationary obstacle problem (on $\Rd$ or bounded domains with zero Dirichlet data) and by a Hele-Shaw-type problem (bounded domains with non-zero Dirichlet data) \cite{GQ2001, GQ2003}. The interest in this singular limit experienced a renaissance in 2014, when  \cite{PQV14} proposed a model incorporating zeroth-order growth terms in the context of tissue growth, $\partial_t n = \frac{k}{k+1}\Delta n^{k+1} + G$ posed on the whole space $\Rd$. The introduction of growth dynamics acting as a source term opened up new challenges as they can `inject' new mass into the system, leading to a different type of Hele-Shaw system in the limit.
The problem then sparked the emergence of new approaches based on (refined) versions of the celebrated Aronson-B\'enilan estimate \cite{AB1979, bevilacqua2022aronson, BPPS2020} on the one hand and methods based on energy-dissipation principles that establish strong compactness of (approximations of the identity of) the pressure gradient, \cite{LX2021, Dav2023, Jac2021}. Moreover, in the parabolic case  convergence rates for the stiff limit have been obtained in various settings, cf. \cite{alexander2014quasi, david2022convergence, david2024improved}.\\

\textbf{Back to the Brinkman system.} 
The picture is much less complete for the Brinkman system \eqref{Brinkman-sys-seminal}. As $\nu \to 0$, the velocity potential becomes $W = p$ and the velocity field in Eq. \eqref{eq:balance-law} becomes the well-known Darcy law $v = - \nabla p$, cf. \cite{DDMS2024, ES2025, dkebiec2025finite, DJK2025}. From a modelling perspective, the inclusion of the viscous term models viscous dissipation of the velocity due to friction between cells and was proposed in \cite{PV2015, DT14}. From an analytical perspective, the picture changes dramatically: while the regularity of the velocity field, $\nabla W$, enjoys better regularity compared to the Darcy setting, the density and the pressure are even less regular. As a consequence, the uniform-in-$k$ estimates from the parabolic setting fail, calling for new techniques. Moreover, now the main challenge arises from establishing strong compactness of the pressure rather than the pressure gradient to pass to the limit in the nonlinear growth terms. 

In \cite{PV2015}, the seminal paper on the stiff limit for Eq. \eqref{Brinkman-sys-seminal}, the authors show that the pressure can only oscillate between two well-defined values. This information is then used in a delicate argument based on representing nonlinear weak limits \`a la \cite{LPT1994}. We shall point out that this strategy requires certain assumptions avoiding mushy regions, i.e., zones where $0 < n < 1$,  in order to identify the kinetic defect measure. Under a similar set of assumptions, using the theory of generalised flows and of viscosity solutions, the authors of \cite{KT18} establish the local uniform convergence of the pressure. Several multispecies extensions have been obtained recently, which are either at a formal level \cite{degond2022multi} or following the technique of \cite{PV2015} in conjunction with the compactness method of \cite{BreschJabin, BreschJabin2, BelgacemJabin}, cf. \cite{DS2020, DPSV2021}.

In this paper, we revisit the stiff limit for a viscoelastic tissue-growth model, System \eqref{Brinkman-sys-seminal}, for two species, $n^{(i)} = n^{(i)}(x,t)$ satisfying
\begin{subequations}
\label{eq:main-system}
\begin{align}
    \label{eq:main-equation}
    \partialt n^{(i)} = \nabla \cdot (n^{(i)} \nabla W) + \nabla \cdot (n^{(i)} \nabla V) + n^{(i)}G^{(i)}(p),
\end{align}
on $\Rd \times (0,T)$, for $i=1,2$. Here, $p$ denotes the pressure and is related to the total population, $n=n^{(1)} + n^{(2)}$, by $p =n^k$. While $V$ is a given external potential, $W$ is the velocity potential coming from Brinkman's law
\begin{align}
    \label{eq:brinkman}
    - \nu \Delta W + W = n^k,
\end{align}
where $k \geq 1$ is the stiffness of the pressure law. The system is complemented by nonnegative initial data $n^{(i)} \geq 0 $. 
\end{subequations}
Throughout, we will work at the level of weak solutions which will be introduced below, and we should highlight that only recently, due to the contribution of \cite{KT2025}, the discussion of strong solutions has been initiated.

\subsection{Our contributions}
\mbox{}\\
The goals of this paper are threefold: First, we present a complete existence result for integrable and bounded solutions. Subsequently,  we establish an existence result for solutions that are merely integrable, and, third, we establish the stiff limit $k\to \infty$ in precisely the setting in which solutions are only integrable.

\medskip
Moreover, we prove a general Aubin-Lions-Simon-type embedding into continuous-in-time functions. The central observation is that, despite lacking the commonly used $H^s$-regularity, solutions inherit enough uniform spatial regularity to control oscillations. Since this spatial information is uniform in time, it can be combined with weak time regularity to infer time continuity. As a result, all the solutions we construct (bounded, integrable, and in the stiff limit) are in $C([0,T];L^p)$, which to the best of our knowledge is a new result. \\ Furthermore, we prove a smoothing property of the stiff limit: even if the initial pressure is only integrable, the limit pressure in the incompressible limit is  bounded in $L^\infty$.

\medskip
Below we dedicate a short paragraph to each of the results as well as a discussion of the assumptions that allow for it.

\subsubsection{Existence result}
\mbox{}\\
The first step consists of an existence result, presented in Section \ref{sec:existence}. To establish the existence of solutions, we require the following set of assumptions.
\begin{assumption}
    \label{assu:existence}
    The velocity potential $V$ is such that
    \begin{itemize}[itemsep=1em]
        \item $\nabla V \in L^\infty(0,T;L^\infty(\Rd))$, and $\Delta V \in L^\infty(0,T;L^\infty(\Rd))$,
    \end{itemize}
    whose norms are bounded by some $C_V>0$. Concerning the growth rate of the two species, $G^{(i)}$, $i=1,2$, we assume
    \begin{itemize}[itemsep=1em]
        \item $G^{(i)}\in C^1([0,\infty))$  and $\partial_p G^{(i)}<-\alpha<0$, for some $\alpha >0$.
    \end{itemize}
    Since both growth rates are decreasing functions there holds
    \begin{align}
        \label{eq:def-pL}
    	\exists p_L> 0:\  \forall p>p_L: \;\; \norm{\Delta V}_{L^\infty(0,T;L^\infty(\Rd))} + G^{(i)}(p) < 0,
    \end{align}
    for both $i=1,2$. It is convenient to define $n_L:= p_L^{1/k}$, and we introduce 
    \begin{align}
        \label{eq:bound-G}
        C_G = \max(G^{(1)}(0), G^{(2)}(0), -G^{(1)}(p_L), -G^{(2)}(p_L)).
    \end{align}
\end{assumption}

Under these assumptions, we can prove the following theorem.
\begin{theorem}\label{thm:Main1}
    Let $n^{(1),in}, n^{(2),in} \in L^1(\Rd)\cap L^\infty(\Rd)$ be nonnegative. 
    Then, there exist functions
    \begin{align*}
        n\1, n\2 \in L^\infty(0,T;L^1(\Rd)\cap L^\infty(\Rd)) \cap C([0,T];L^q(\Rd)) \cap H^1(0,T;H^{-1}(\Rd)),
    \end{align*} 
    for any $1\leq q < \infty$, such that $(n\1, n\2)$ is a weak solution of the system
    \begin{align*}
        \partialt {n\i}  &= \nabla \cdot (n\i\nabla W) + \nabla \cdot(n\i\nabla V) + n\i G\i(p),\\
        -\nu \Delta W + W &= p\\
        n\i(\cdot, 0) &= n^{(i),in},
    \end{align*}
    in the sense of Definition~\ref{def:weak-sol-Brinkman}.
\end{theorem}

\subsubsection{Removal of $L^\infty$-assumption on initial data}
\mbox{}\\
We relax the class of initial data to integrable but not necessarily bounded functions, cf. Section \ref{sec:UnboundedData}.
\begin{assumption}
    To this end, in addition to Assumption \ref{assu:existence}, we impose
    \label{assu:integrable-data}
    \begin{itemize}[itemsep=1em]
        \item $\nabla V \in L^\infty(0,T;H^1(\Rd))$, which is required to make sense of the weight equation, Proposition \ref{prop:w}, in the DiPerna-Lions setting,
        \item $\exists \Gamma >0,\  s\geq 1: \ |G^{(i)}(p)| \leq \Gamma(1 + p^s)$, for $i=1,2$, as the growth term is no longer trivially in $L^\infty$ in the case of unbounded solutions.
    \end{itemize}
\end{assumption}
Under these assumptions we can prove the following theorem.
\begin{theorem}
    Let $0 \leq n^{(i), \mathrm{in}} \in L^1(\Rd) \cap L^2(\Rd)$ be given initial data for $i=1,2$ such that $(n^{(1), \mathrm{in}}+n^{(2), \mathrm{in}})^k \in L^1(\Rd) \cap L^{2s}(\Rd)$.
    Then, there exist functions 
    \begin{align}
        n^{(1)}, n^{(2)} \in L^\infty(0,T;L^1(\Rd)\cap L^2(\Rd)) \cap C([0,T];L^q(\Rd)),
    \end{align}
    for any $1\leq q < 2$, and
    \begin{align}
        \partialt {n^{(1)}}, \partialt {n^{(2)}} \in L^2(0,T; H^{-\frak s}(\R^d)),
    \end{align}
    for any $\frak{s}>d/2 + 1$ which are a weak solution of the system
    \begin{align*}
    \partialt {n\i} &= \nabla \cdot(n\i\nabla W) + \nabla \cdot(n\i\nabla V) + n\i G\i(p),\\
    -\nu \Delta W + W &= p\\
    n\i(0,x) &= n^{(i),in}
    \end{align*}
    in the sense of Definition~\ref{def:weak-sol-integrable}.
\end{theorem}

\subsubsection{Passage to the stiff limit}
\mbox{}\\
Ultimately, we address the stiff limit, $k\to \infty$. 

\begin{assumption}
\label{assu:stiff}
    In addition to Assumptions \ref{assu:existence}, 
    \ref{assu:integrable-data}, we require
    \begin{itemize}[itemsep=1em]
        \item $\partial_t {\Delta V} \in L^1(0,T;L^{\frak p'}(\Rd))$, where $\frak p'$ is the H\"older conjugate of $\frak p$, which is required to control $\partial_t p$ in a suitable space.
        \item $\nabla \Delta V \in L^1(0,T;L^\infty(\Rd))$ is required in Lemma \ref{lemma:pQinL1}. 
    \end{itemize}
\end{assumption} 

Before we can state the theorem on the stiff limit, let us introduce the following integrability exponents
\begin{itemize}
    \item $\frak{p} = \max(2,(d/2)^+)$, for the densities,
    \item  $\frak{q} > 2s\frak{p}'$, for the pressures,
\end{itemize}
where $\frak{p}'$ is the H\"older conjugate of $\frak p$ (i.e., $\frak p' = \min(2,(\frac{d}{d-2})^{-})$ for $d\geq 2$ and $\frak p' = 2$ in one dimension), and $s\geq 1$ is give in Assumption~\ref{assu:integrable-data}. Here and henceforth, we use the following notation for integrability exponents:\ $f\in L^{\frak r\pm }$ means that $f\in L^{\frak{r} \pm \delta}$ for some $\delta >0$.  
\\

Under these assumptions, we prove the following result.

\begin{theorem}
    For $i=1,2$, let $(n_k^{(i), \mathrm{in}})_k \subset L^1(\Rd)\cap L^{\frak{p}}(\Rd)$ be sequences of initial densities, uniformly bounded in $k$, satisfying $\abs{\cup_{k\in \N} \, \mathrm{supp} (n^{\mathrm{in}}_k)} < \infty$, 
    as well as 
    $$
        n_k^{(i), \mathrm{in}} \to n_\infty^{(i), \mathrm{in}}, \qquad \text{in } L^1(\Rd).
    $$
    Further, let us assume that the sequence of initial pressures, $(p_k^{\mathrm{in}})_k \subset L^1(\Rd)\cap L^\frak{q}(\Rd)$, is uniformly bounded in $k$, as well. Denote by $n_k\i$, $i=1,2$, the solution constructed in Theorem~\ref{thm:existence-integrable-data}. Then, there exist functions 
   \begin{align*}
        n\i_\infty &\in L^\infty(0,T;L^\frak{p}(\Rd)) \cap C([0,T];L^q(\Rd)), \quad q\in[1,\frak p),\\
        p_\infty &\in L^\infty(0,T;L^1(\Rd)\cap L^\frak{q}(\Rd)),
   \end{align*}
    such that, up to a subsequence,
    \begin{alignat*}{2}
        n\i_k &\to n\i_\infty,\quad &&\text{in } L^1(0,T;L^1(\Rd)),\\
        p_k &\to p_\infty,\quad &&\text{in } L^1(0,T;L^1(\Rd)).
    \end{alignat*}
    Moreover, 
    \begin{equation*}
        W_k\to W_\infty:=K_\nu\star p_\infty,\quad \text{in } L^2(0,T;L^2(\Rd)),
    \end{equation*}
    and $(n\1_\infty, n\2_\infty, W_\infty, p_\infty)$ satisfies (in the weak sense) the system:
    \begin{align*}
        \partialt {n\i_\infty} &= \nabla \cdot(n\i_\infty \nabla W_\infty) + \nabla \cdot(n\i_\infty\nabla V) + n\i_\infty G\i(p_\infty),\\
        -\nu\Delta W_\infty + W_\infty &= p_\infty,\\
        n\i_\infty(x, 0) &= n_\infty^{(i),\mathrm{in}},
    \end{align*}
    with the pointwise relation $p_\infty(1-n_\infty)=0$.
    Finally, the following complementarity relation holds a.e.\ in $\R^d\times(0,T)$:   \begin{equation*}
        p_\infty(W_\infty - p_\infty + \nu \Delta V + \nu n\1_\infty  G\1(p_\infty) + \nu n_\infty\2  G\2(p_\infty)) = 0.
    \end{equation*}
\end{theorem}

Finally, we prove that the density and the pressure become essentially bounded in the stiff limit.
\begin{theorem}
    The limits $n_\infty$ and $p_\infty$ satisfy
    $$
        n_\infty, p_\infty \in L^\infty(0,T;L^\infty(\Rd)),
    $$ 
    with $0\leq p_\infty \leq p_L$ and $0 \leq n_\infty \leq 1$.
\end{theorem}

\section{Removal of $L^\infty$-assumption on initial data}
\label{sec:UnboundedData}

In this section we shall extend the existence result to the case of initial data that enjoy certain integrability, but which are not necessarily bounded.
Since we cannot expect the pressure to be bounded in such a situation, we need to make an additional structural assumption on the growth rates $G\i$, namely we assume that they satisfy Assumption~\ref{assu:integrable-data}, i.e.,
\begin{equation*}
    |G\i(p)| \leq \Gamma(1+p^s),\quad\text{for some $s\geq 1$}.
\end{equation*}

Before constructing the approximation and deriving the necessary estimates, let us state the definition of weak solutions we will be working with and formulate the main result of this section.
\begin{definition}[Integrable weak solutions]
    \label{def:weak-sol-integrable}
    Suppose Assumptions \ref{assu:existence} and \ref{assu:integrable-data} are met and we are given some nonnegative initial data
    $$
        (n^{(1), \mathrm{in}}, n^{(2), \mathrm{in}}) \in (L^1(\Rd)\cap L^2(\Rd))^2,
    $$ 
    such that
    $$
        p^{\mathrm{in}}:=(n^{(1), \mathrm{in}}+n^{(2), \mathrm{in}})^k \in L^1(\Rd) \cap L^{2s}(\Rd).
    $$
    Then, a pair of nonnegative functions $(n^{(1)}, n^{(2)})$ with 
    \begin{align}
        n^{(1)}, n^{(2)} \in L^\infty(0,T;L^1(\Rd)\cap L^2(\Rd)) \cap C([0,T];L^q(\Rd)),
    \end{align}
    for any $1\leq q < 2$, and
    \begin{align}
        \partialt {n^{(1)}}, \partialt {n^{(2)}} \in L^2(0,T; H^{-\frak s}(\R^d)),
    \end{align}
    for any $\frak{s}>d/2 + 1$, is called
    a weak solution to System \eqref{eq:main-system} if there holds
    \begin{align}
        \int_0^T\int_\Rd \varphi \frac{\partial n^{(i)}}{\partial t}  + n^{(i)} \nabla  \varphi \cdot (\nabla V + \nabla W)  \dx x \dx t = \int_0^T \int_\Rd  \varphi n^{(i)} G^{(i)}(p) \dx x \dx t,
    \end{align}
    for any test function $\varphi \in L^2(0,T ; H^{\frak s}(\Rd))$, and $n^{(i)}(0) = n^{(i),\mathrm{in}}$, where $p = (n^{(1)} + n^{(2)})^k \in L^\infty(0,T;L^{2s}(\Rd))$.
\end{definition}
We can now present the main result of this section, the existence of integrable weak solutions.

\begin{theorem}[Existence of solutions]
    \label{thm:existence-integrable-data}
    For any nonnegative initial data $n^{(i), \mathrm{in}} \in L^1(\Rd)\cap L^2(\Rd)$ and $p^{\mathrm{in}} \in L^1(\Rd)\cap L^{2s}(\Rd)$ there exists a solution in the sense of Definition~\ref{def:weak-sol-integrable}. 
\end{theorem}

\subsection{Approximation of the initial data}
\label{subsec:DataApproximation}
In this section, we approximate the general integrable initial data $n^{(i), \mathrm{in}}$ by bounded data $n_\eta^{(i), \mathrm{in}}\in L^1\cap L^\infty$ and derive estimates that are uniform in $\eta>0$. Crucially, these bounds are independent of the essential supremum of the approximating data.

Given some initial data $ 0 \leq n^{(i), \mathrm{in}} \in L^1\cap L^2(\Rd)$, we assume the approximations $n_{\eta}^{(i),\mathrm{in}}$, to be chosen such that
\begin{equation*}
\begin{aligned}
    0 & \leq n_{\eta}^{(i),\mathrm{in}} \in L^1\cap L^\infty(\Rd),\\[\alnspace]
    n_{\eta}^{(i),\mathrm{in}}&\to n^{(i),\mathrm{in}}, \quad \text{in $L^1(\Rd)$, as $\eta\to 0$},\\[\alnspace]
    \norm{n_{\eta}^{(i),\mathrm{in}}}_{L^{\frak{p}}(\Rd)} &\leq C \norm{n^{(i),\mathrm{in}}}_{L^{\frak{p}}(\Rd)},\quad \forall\frak{p}\in [1,2].
\end{aligned}
\end{equation*}
Furthermore, we assume that the sequence of initial pressures $p_{\eta}^{\mathrm{in}} := (n_{\eta}^{(1),\mathrm{in}} + n_{\eta}^{(2),\mathrm{in}})^k$, is bounded in $L^1\cap L^{\frak q}(\Rd)$ uniformly in $\eta$ and $k$ for some $\frak q > 2s\geq 2$.

According to the existence result for bounded solutions in Theorem~\ref{thm:Main1}, for each $\eta>0$ there exists a global weak solution $(n\1_\eta, n\2_\eta)$ of System~\eqref{eq:main-system} with data $(n^{(1),\mathrm{in}}_\eta, n^{(2),\mathrm{in}}_\eta)$. 

We next use the compactness criterion from Appendix~\ref{app:compcrit} to pass to the limit $\eta\to0$ and construct global weak solutions with the original initial data $n^{(i),\mathrm{in}}$. The argument relies only on the uniform-in-$k$ a priori estimates obtained in the previous subsection; these estimates will later provide the compactness of the densities with respect to $k$ needed for the incompressible limit.
There is, however, one additional difficulty: Lemma~\ref{lem:CompactnessCriterion} cannot be applied directly, and we therefore introduce an auxiliary weight.

\subsection{A priori estimates}
We first establish the a priori estimates needed to remove the essential boundedness assumption on the initial data. These estimates are uniform in the approximation parameter $\eta$ and, importantly, also in the stiffness exponent $k$; the latter uniformity will be essential for the incompressible limit. To lighten the notation, we keep the subscript $\eta$ only in the statements of the subsequent estimates.

\begin{lemma}[Uniform $L^1$-control of $n_\eta$]
    There holds
    \begin{align*}
       \norm{n_\eta}_{L^\infty(0,T; L^1(\Rd))} \leq C, 
    \end{align*}
    for some constant $C>0$ independent of $\eta$ and $k$.
\end{lemma}
\begin{proof}
    Consider
    \begin{align*}
        \ddt \int_{\Rd} n^{(i)} \dx x 
        &=  \int_{\Rd}  \nabla \cdot(n^{(i)} \nabla (V+W)) \dx x + \int_{\Rd} n^{(i)} G^{(i)}(p) \dx x 
        \leq C_G \int_{\Rd} n^{(i)} \dx x.
    \end{align*}
    Applying Gronwall's lemma yields
    \begin{align*}
        \norm{n\i}_{L^\infty(0,T;L^1(\Rd))} \leq e^{C_GT}\norm{n^{(i),\mathrm{in}}_\eta}_{L^1(\Rd)} \leq C\norm{n^{(i),\mathrm{in}}}_{L^1(\Rd)},
    \end{align*}
    and, summing up the two inequalities ($i=1,2$), completes the proof.
\end{proof}

\begin{lemma}[Uniform $L^\infty(L^q)$-control of $n_\eta$]
    \label{lem:propagation-of-LP-for-n}
    There is a constant $C = C(\alpha, C_V, C_G)>0$ such that
    \begin{align*}
        \norm{n_\eta}_{L^\infty(0,T; L^q(\Rd))}^q  \leq \max\left\{\norm{n_\eta^\mathrm{in}}_{L^q(\Rd)}^q, 
        \norm{n_\eta^\mathrm{in}}_{L^1(\Rd)} C^\frac{q-1}{k} \right\},
    \end{align*}
    for any $q>1$.
\end{lemma}
\begin{proof}
    Let us begin by multiplying the equation for the total population,
    $$
        \partialt n = \nabla \cdot (n \nabla (W + V)) + n^{(1)} G^{(1)}(p) + n^{(2)} G^{(2)}(p),
    $$
    by $q n^{q-1}$ and integrating to get
    \begin{align*}
        \ddt &\int_{\Rd} n^q \dx x \\
        &=  - (q-1) \int \nabla n^q \cdot \nabla (V+ W) \dx  x + q \int_{\Rd} n^{q-1} [n^{(1)} G^{(1)}(p) + n^{(2)}G^{(2)}(p)] \dx x \\
        &\leq  (q-1) \int_{\Rd} n^q \Delta W \dx x + q \left[\norm{\Delta V}_{L^\infty(0,T; L^\infty(\Rd))} + C_G\right]\int_{\Rd} n^q \dx x - q \alpha \int_{\Rd}  n^{q+k} \dx x,
    \end{align*}
    where we used the fact that 
    $$
        G^{(i)}(p) \leq C_G - \alpha p,
    $$
    by Assumption~\ref{assu:existence}. Concerning the first integral, we observe that
    \begin{align*}
        \int_{\Rd} n^q \Delta W \dx x 
        &= \frac1\nu\int_{\Rd} n^q (W - n^k) \dx x \\
        &=  - \frac1\nu \norm{n}_{L^{k+q}(\Rd)}^{k+q} + \frac1\nu \int_{\Rd} W n^q \dx x\\
        &\leq  - \frac1\nu\norm{n}_{L^{k+q}(\Rd)}^{k+q} + \frac1\nu \norm{n^q}_{L^{1 + k/q}(\Rd)} \norm{W}_{L^{1 + q/k}(\Rd)},
    \end{align*}
    by H\"{o}lder's inequality. Moreover, by Young's convolution inequality, we have
    \begin{align*}
        \norm{W}_{L^{1 + q/k}(\Rd)} = \norm{K \star p}_{L^{1 + q/k}(\Rd)}\leq \norm{p}_{L^{1 + q/k}(\Rd)} = \norm{n^k}_{L^{1 + q/k}(\Rd)} = \norm{n}_{L^{k+q}(\Rd)}^k,
    \end{align*}
    and, concurrently, we note that
    \begin{align*}
        \norm{n^q}_{L^{1 + k/q}(\Rd)}  = \norm{n}_{L^{k+q}(\Rd)}^q.
    \end{align*}
    Upon combining both, we find
    \begin{align*}
        \int_\Rd n^q \Delta W \dx x \leq 0.
    \end{align*}
    Hence, the estimate reduces to
    \begin{equation}\label{eq:LqEstimate}
    \begin{aligned}
        \ddt \int_\Rd n^q \dx x
        \leq  q \left(\gamma \norm{n}_{L^q(\Rd)}^q - \alpha \norm{n}_{L^{q+k}(\Rd)}^{q+k}\right),
    \end{aligned}
    \end{equation}
    where $\gamma  := \norm{\Delta V}_{L^\infty(0,T;L^\infty(\Rd))} + C_G$.
    Next, we use $L^p$-interpolation to estimate
    \begin{align}
        \label{eq:interpolation-1}
        \norm{n}_{L^q(\Rd)} \leq \norm{n}_{L^1(\Rd)}^{1-\theta} \norm{n}_{L^{q+k}(\Rd)}^\theta,
    \end{align}
    where $\theta$ satisfies
    \begin{align*}
        \theta = \frac{q \big(1 - (q+k)\big) + k}{q \big(1 - (q+k)\big)} \in (0,1),
    \end{align*}
    as $q,k>1$. Rearranging Eq. \eqref{eq:interpolation-1}, we find
    \begin{align*}
        \norm{n}_{L^{q+k}(\Rd)}^{q+k} \geq \norm{n}_{L^{q}(\Rd)}^{\frac{q+k}{\theta}} \norm{n}_{L^{1}(\Rd)}^{\frac{\theta - 1}{\theta} (q+k)}.
    \end{align*}
    Setting $\xi(t) := \norm{n}_{L^q(\Rd)}^q$ and revisiting Inequality~\eqref{eq:LqEstimate}, we obtain
    \begin{align*}
        \dot \xi(t) 
        &\leq q \left(\gamma \xi - \alpha \norm{n}_{L^{q}(\Rd)}^{\frac{q+k}{\theta}} \norm{n}_{L^{1}(\Rd)}^{\frac{\theta - 1}{\theta} (q+k)} \right)\\[6pt]
        &\leq q \left(\gamma \xi - \alpha \xi^{\frac{q+k}{q\theta}} C_0^{\frac{\theta - 1}{\theta} (q+k)} \right)\\[6pt]
        &= q \alpha C_0^{\frac{\theta - 1}{\theta} (q+k)} \xi \left(\frac{\gamma C_0^{\frac{1 - \theta}{\theta} (q+k)}}{\alpha} -  \xi^{\frac{q + k - q\theta}{q\theta}} \right),
    \end{align*}
    where $C_0 := \norm{n}_{L^\infty(0,T;L^1(\Rd))}$.
    To simplify the exponents in the inequality we note
    \begin{align*}
        \frac{1-\theta}{\theta} (q+k) = \frac{k}{q-1}\quad\text{and}\quad \frac{q + k - q\theta}{q\theta} = \frac{k}{q-1},
    \end{align*}
    such that the differential inequality becomes
        \begin{align*}
       \dot \xi(t) 
        &\leq q \alpha C_0^{\frac{k}{1-q}} \xi \left(\frac\gamma\alpha C_0^{\frac{k}{q-1}}  - \xi^{\frac{k}{q-1}}\right).  
    \end{align*}
    From classical ODE theory we infer that
    \begin{align*}
        0 \leq \xi(t) \leq \max\left\{\xi(0),  C_0 \left(\frac{\gamma}{\alpha}\right)^{\frac{q-1}{k}} \right\},
    \end{align*}
    i.e.,
    $$
        \norm{n}_{L^\infty(0,T;L^q(\Rd))}^q \leq \max\{\norm{n^\mathrm{in}}_{L^q(\Rd)}^q, \norm{n^\mathrm{in}}_{L^1(\Rd)} C(\alpha, C_V, C_G)^\frac{q-1}{k}\},
    $$
    where we have used $\norm{n}_{L^\infty(0,T;L^1(\Rd))}\leq C\norm{n^{\mathrm in}}_{L^1(\Rd)}$ and 
    $C(\alpha, C_V, C_G)>0$ is independent of $\eta, q, k$.
\end{proof}

\begin{corollary}[$L^\ell$-control of $p_\eta$]\label{cor:LebesguePressure}
    For any $\ell \geq 1$ there holds
    \begin{align*}
        \norm{p_\eta}_{L^\infty(0,T; L^\ell(\Rd))}^\ell \leq \max\{\norm{p_\eta^\mathrm{in}}_{L^\ell(\Rd)}^\ell, \norm{n_\eta^\mathrm{in}}_{L^1(\Rd)} C(\alpha, C_V, C_G)^{\ell}\},
    \end{align*}
    where $C$ is independent of $\eta$ and $k$.
\end{corollary}
\begin{proof}
    For any $\ell \geq 1$ we set $q = \ell k$ and observe
    \begin{align*}
        \norm{p}_{L^\infty(0,T; L^\ell(\Rd))}^\ell 
        &= \norm{n}_{L^\infty(0,T; L^q(\Rd))}^q \\
        &\leq \max\{\norm{n^\mathrm{in}}_{L^q(\Rd)}^q, \norm{n^\mathrm{in}}_{L^1(\Rd)} C(\alpha, C_V, C_G)^{\frac{q-1}{k}}\}\\
        &\leq \max\{\norm{p^\mathrm{in}}_{L^\ell(\Rd)}^\ell, \norm{n^\mathrm{in}}_{L^1(\Rd)} C(\alpha, C_V, C_G)^{\ell}\}.\qedhere
    \end{align*}
\end{proof}

\begin{proposition}[Time derivatives]\label{prop:TimeDerivatives1}
We have the following time regularity of the densities and the pressure:
    \begin{itemize}
        \item $\partial_tn_\eta\i$ is bounded in $L^\infty(0,T;W^{-1,\frak{r}'}(\Rd))$ for some $\frak{r} > d$, uniformly in $\eta$ and $k$;
        \item  $\partial_t p \in L^\infty(0,T;H^{-\frak{s}}(\Rd))$ for some $\frak{s}>d/2+1$, uniformly in $\eta$.
    \end{itemize}
\end{proposition}
\begin{proof}
Let us begin by observing that all the integrability of the initial data is propagated uniformly in time thanks to the previous a priori estimates. 
For every $\varphi\in W^{1,\frak{r}}(\Rd)$,
    \begin{align*}
        &\abs*{\int_{\Rd}\nabla\varphi\cdot n\i\nabla(W+V)\dx x} \\
        &\leq \norm{\nabla \varphi}_{L^{\frak r}(\Rd)}\norm{\nabla V}_{L^\infty(0,T;L^\infty(\Rd))}\norm{n\i}_{L^\infty(0,T;L^{\frak r'}(\Rd))} + \norm{\nabla\varphi}_{L^{\frak r}(\Rd)}\norm{n\i \nabla W}_{L^\infty(0,T;L^{\frak{r}'}(\Rd))}.
    \end{align*}
Now, notice that for any $1<\alpha< 2$
\begin{align*}
    \norm{n\i\nabla W}_{L^\alpha(\Rd)} \leq \norm{n\i}_{L^2(\Rd)}\norm{\nabla K_\nu}_{L^1(\Rd)}\norm{p}_{L^{\frac{2\alpha}{2-\alpha}}(\Rd)}.
\end{align*}
Since $\frak{q}>2$, it follows that $n\i\nabla W$ is uniformly bounded in $L^\infty(0,T;L^{1+}(\Rd))$. Therefore, we can choose $\frak r \gg 1$ such that both $n\i\nabla W$ and $n\i$ are simultaneously uniformly bounded in $L^{\frak r'}(\Rd)$.
Similarly, for the growth term, we observe
    \begin{align*}
        &\abs*{\int_\Rd n\i G\i(p) \varphi\dx x}\\
        &\quad\leq \Gamma(\norm{n\i}_{L^\infty(0,T; L^1(\Rd))} + \norm{n\i}_{L^\infty(0,T; L^{\frak{p}}(\Rd))}\norm{p}_{L^\infty(0,T; L^{s\frak{p}'}(\Rd))})\norm{\varphi}_{L^\infty(\Rd)}.
    \end{align*}
    It follows that $\partial_t n \in L^\infty(0,T;W^{-1,\frak{r}'}(\Rd))$  uniformly in both $\eta$ and $k$. We argue similarly for $\partial_t p$. Adding the two equations for the densities and using $p=n^k$, we obtain, in the sense of distributions,
    \begin{equation*}
        \partialt p = \nabla p\cdot\nabla(W+V) + kp(\Delta W + \Delta V) + k n^{k-1} n^{(1)}G\1(p) + k n^{k-1} n^{(2)}G\2(p).
    \end{equation*}
    Setting $\mathcal{V}:=W+V$, we have for any $\varphi\in H^{\frak s}(\Rd)$ with $\frak s>d/2+1$,
    \begin{align*}
        \abs*{\int_{\Rd} (\nabla p\cdot \nabla \mathcal{V} + kp\Delta \mathcal{V})\varphi\dx x}
        \leq \abs*{\int_\Rd p\nabla \mathcal{V}\cdot \nabla\varphi\dx x} + (k-1)\abs*{\int_\Rd p\Delta \mathcal{V}\,\varphi\dx x}.
    \end{align*}
    Since $H^{\frak s}(\Rd)\hookrightarrow W^{1,\infty}(\Rd)$, the first term is bounded by
    \begin{align*}
        \abs*{\int_\Rd p\nabla \mathcal{V}\cdot \nabla\varphi\dx x}
        &\leq \norm{\nabla\varphi}_{L^\infty(\Rd)}\left(\norm{p}_{L^1(\Rd)}\norm{\nabla V}_{L^\infty(\Rd)} + \norm{p}_{L^2(\Rd)}\norm{\nabla W}_{L^2(\Rd)}\right)\\
        &\leq C\norm{\varphi}_{H^{\frak s}(\Rd)}.
    \end{align*}
    For the second term, we use $\Delta W=(W-p)/\nu$ from Brinkman's law and obtain
    \begin{align*}
        \abs*{\int_\Rd p\Delta \mathcal{V}\,\varphi\dx x}
        &\leq \norm{\varphi}_{L^\infty(\Rd)}\left(\norm{p}_{L^1(\Rd)}\norm{\Delta V}_{L^\infty(\Rd)} + \frac{1}{\nu}\norm{p}_{L^2(\Rd)}\norm{W}_{L^2(\Rd)} + \frac{1}{\nu}\norm{p}_{L^2(\Rd)}^2\right)\\
        &\leq C\norm{\varphi}_{H^{\frak s}(\Rd)}.
    \end{align*}
    Here the constants are independent of $\eta$, because the previous estimates give uniform control of $p$ in $L^\infty(0,T;L^1\cap L^2)$, while Brinkman's law gives the corresponding control of $W$ and $\nabla W$, and $\nabla V, \Delta V$ are bounded by Assumption \ref{assu:existence}.

    It remains to estimate the growth terms. On the set $\{n>0\}$, we have $n^{k-1}n\i=p\frac{n\i}{n}$ and $0\leq n\i/n\leq1$, whereas the integrand vanishes on $\{n=0\}$. Therefore, using $\abs*{G\i(p)}\leq \Gamma(1+p^s)$ and, once more, $H^{\frak s}(\Rd)\hookrightarrow L^\infty(\Rd)$,
    \begin{align*}
        \abs*{k\int_\Rd n^{k-1}n\i G\i(p)\varphi \dx x}
        &= \abs*{k\int_{\{n>0\}} p\frac{n\i}{n}G\i(p)\varphi \dx x }\\
        &\leq k \Gamma \int_\Rd p(1+p^s)\abs{\varphi}\dx x\\
        &\leq k \Gamma \left(\norm{p}_{L^1(\Rd)} +  \norm{p}_{L^{s+1}(\Rd)}^{s+1}\right)\norm{\varphi}_{L^\infty(\Rd)}\\
        &\leq C_k\norm{\varphi}_{H^{\frak s}(\Rd)}.
    \end{align*}
   Thus, we have $\partial_t p$ in $L^\infty(0,T;H^{-\frak s}(\Rd))$, uniformly in $\eta$ but not in $k$.
\end{proof}

\subsection{The weights}
We define the weights $\upsilon_\eta$ as solutions of the transport equation 
\begin{align}
\label{eq:weight}
\left\{
\begin{array}{rl}
    \ds \partialt {\upsilon_\eta} - \nabla \upsilon_\eta \cdot \nabla \mathcal{V}_\eta \remspace &= -\lambda B_\eta \upsilon_\eta,\\[\alnspace]
    \upsilon_\eta(\cdot, 0) \remspace &= 1,
\end{array}
\right.
\end{align}
where $B_\eta = \mathfs{M}|D^2 \mathcal{V}_\eta| + |\Delta\mathcal{V}_\eta|$, and where $\lambda\geq0$ is a sufficiently large constant to be fixed below.
Here, $\mathfs{M}$ denotes the maximal operator
$$
    \mathfs{M} f(x) = \sup_{0 < \rho \leq 2}  \frac1{|\{|z| \leq \rho\}|}\int_{|z|\leq \rho} f(x+z)\,\dx{z}.
$$
We recall that $\mathfs{M}$ is bounded on $L^q$ for every $q>1$. 

The following proposition establishes existence of the weights and two key properties that will be used below. Specifically, the second bound will be instrumental in removing the weights from the compactness quantity for the densities.

\begin{proposition}\label{prop:w}
There exists a unique solution to Eq.~\eqref{eq:weight}. Moreover, we have
\begin{itemize}
    \setlength\itemsep{0.7em}
    \item[(i)] $0\leq \upsilon_\eta(t,x) \leq 1$, for almost every $(x, t)\in \Rd \times (0,T)$.
    \item[(ii)] There exists $C>0$ such that for almost every $t$ there holds
    \begin{equation*}
    \label{boundlogw}
        \intO{ n_\eta^{(i)} |\log \upsilon_\eta| } \leq C \lambda.
    \end{equation*}
\end{itemize}
\end{proposition}

\begin{proof}
$(i)$ Since $D^2 W_\eta\in L^2(\Rd\times (0,T))$, we have $B_\eta\in L^2(\Rd \times (0,T))$. Since also $\nabla (W_\eta+V)\in L^\infty(0,T;H^1(\R^d))$, by Assumption \ref{assu:integrable-data}, the DiPerna-Lions theory of renormalised solutions~\cite{DL} provides the existence and uniqueness of a nonnegative solution to Eq.~\eqref{eq:weight}. Moreover, since $B_\eta$ is nonnegative, the property $\upsilon_\eta\leq 1$ is propagated from the initial condition.

$(ii)$ Since $\upsilon_\eta\leq 1$, we have $|\log \upsilon_\eta|=-\log \upsilon_\eta$. By renormalisation, Eq. \eqref{eq:weight} implies
$$
    \fpartial t |\log \upsilon_\eta| - \nabla (W_\eta+V) \cdot \nabla |\log \upsilon_\eta| = \lambda B_\eta.
$$
Using also the density equation associated to the regularised initial data, we get
$$
    \fpartial t \prt*{n_\eta\i |\log \upsilon_\eta|} - \nabla\cdot \prt*{n\i_\eta |\log \upsilon_\eta|  \nabla (W_\eta+V)} = n\i_\eta |\log \upsilon_\eta| G^{(i)}(p_\eta) + \lambda n\i_\eta B_\eta,
$$
cf. \cite[Lemma 6.8]{BreschJabin2}. Then, integrating in space, we deduce
$$
    \ddt \intO{ n\i_\eta |\log \upsilon_\eta|} \leq C_G \intO{ n\i_\eta |\log \upsilon_\eta|} + \lambda \intO{ n\i_\eta B_\eta }.
$$
Using Gronwall's lemma, we obtain
$$
    \intO{n\i_\eta |\log \upsilon_\eta|(t,x)} \leq \lambda e^{C_G T}\int_0^T\!\!\int_{\Rd} n\i_\eta B_\eta\dx{x} \dx{t}.
$$
We conclude using the Cauchy-Schwarz inequality, as $B_\eta$ and $n\i_\eta$ are uniformly bounded in $L^2(\R^d\times (0,T))$.
\end{proof}

\subsection{Compactness of the densities}
We will now use the weights constructed above to propagate compactness of the individual density sequences. To this end, we will require global compactness of the velocity potential $W_\eta$, which we establish first.
\begin{proposition}\label{prop:CompactW_eta}
    The family $\{W_\eta\}_{\eta>0}$ is compact in $L^2(0,T;L^2(\Rd))$. We call its limit $W$.
\end{proposition}
\begin{proof}
    From Proposition~\ref{prop:TimeDerivatives1} and Corollary~\ref{cor:LebesguePressure} we have that $\partial_tp_\eta \in L^2(0,T;H^{-\frak{s}}(\Rd))$ and that $p_\eta \in L^2(0,T;L^2(\Rd))$, both bounds being uniform in $\eta$. Since $W_\eta = K_\nu\star p_\eta$ (and since $\nabla K_\nu$ is integrable), we deduce the corresponding bounds for $\partial_t W_\eta$ and $\nabla W_\eta$. By the Aubin-Lions Lemma, the sequence $W_\eta$ is locally compact in $L^2(0,T;L^2(\Rd))$.

    To obtain global compactness, we argue that the sequence is tight. First, we show that tightness is propagated for the densities, since the initial sequence converges globally.  Let $\chi_R \in C_c^\infty([0,\infty))$ be the usual cut-off function with $|\nabla \chi_R| \lesssim R^{-1}$ and $|\Delta \chi_R| \lesssim R^{-2}$
    Testing the density equation by $1-\chi_R$ we obtain
    \begin{align*},
        &\ddt \int_\Rd n_\eta\i(1-\chi_R)\dx x = \int_\Rd \nabla\chi_R\cdot n_\eta\i(\nabla W_\eta+\nabla V)\dx x + \int_\Rd n\i_\eta G\i(p_\eta)(1-\chi_R)\dx x\\
        &\leq \norm{n\i_\eta}_{L^2(\Rd)}(\norm{\nabla W_\eta}_{L^2(\Rd)}+\norm{\nabla V}_{L^2(\Rd)})\frac{C}{R} + C_G\int_\Rd n\i_\eta(1-\chi_R)\dx x.
    \end{align*}
    It follows that
    \begin{align*}
        \int_\Rd n\i_\eta(1-\chi_R) \dx x \leq e^{C_G T}\int_\Rd n^{(i),\mathrm{in}}_\eta(1-\chi_R)\dx x + \frac{C}{R},
    \end{align*}
    where the constant is independent of $\eta$, and consequently
    \begin{equation}\label{eq:tightness1}
        \lim_{R\to\infty}\;\sup_{\eta>0}\;  \sup_{0 \leq t \leq T} \;\int_\Rd n\i_\eta(1-\chi_R) \dx x =0.
    \end{equation}
    Now, we test the Brinkman equation with $(1-\chi_R)$, obtaining
    \begin{equation}\label{eq:tightness2}
    \begin{aligned}
        \int_\Rd W_\eta(1-\chi_R)\dx x &= \int_\Rd p_\eta(1-\chi_R)\dx x - \nu \int_\Rd W_\eta\Delta\chi_R\dx x\\
        & \leq \int_\Rd p_\eta(1-\chi_R)\dx x + \frac{C\nu}{R^2},
    \end{aligned}
    \end{equation}
    and we observe that
    \begin{align*}
        \int_\Rd p_\eta(1-\chi_R)\dx x &= 
        \int_\Rd n^k_\eta(1-\chi_R)\dx x\\
        &\leq \prt*{\int_\Rd n^{2k-1}_\eta(1-\chi_R)\dx x}^{1/2}\prt*{\int_\Rd n_\eta(1-\chi_R)\dx x}^{1/2}\\
        &= \prt*{\int_\Rd p^{2-\frac{1}{k}}_\eta(1-\chi_R)\dx x}^{1/2}\prt*{\int_\Rd n_\eta(1-\chi_R)\dx x}^{1/2}\\
        &\leq C\prt*{\int_\Rd n_\eta(1-\chi_R)\dx x}^{1/2},
    \end{align*}
    where the constant is independent of $\eta$, since $2-1/k \in (1,2)$ and $p_\eta$ is uniformly bounded in $L^\infty(0,T; L^1\cap L^2(\Rd))$, cf. Corollary \ref{cor:LebesguePressure}.
    It now follows from~\eqref{eq:tightness1} and~\eqref{eq:tightness2} that
    \begin{equation*}
        \lim_{R\to\infty}\;\sup_{\eta>0}\; \sup_{0 \leq t \leq T} \;\int_\Rd W_\eta(1-\chi_R) \dx x =0,
    \end{equation*}
    and we readily conclude the proof.
\end{proof}
\subsubsection{Propagation of compactness -- with weights}
Throughout, when invoking the compactness method by Belgacem, Bresch, and Jabin, cf. Appendix \ref{app:compcrit}, for brevity we shall use the superscript notation 
$$
	\mathcal K_h^{x,y} := \mathcal K_h(x-y), \quad n_\eta^{(i),x} :=  n_\eta^{(i)}(x), \quad n_\eta^{x} :=  n_\eta(x), \qquad p_\eta^x := p_\eta(x), 
$$
and, analogously, for all other functions appearing in the nonlocal compactness argument.

\begin{proposition}\label{prop:CompactnessEtaWeight}
    Let $\calK_h$ be the kernel defined in Appendix~\ref{app:compcrit}. The following estimate holds for $h<1$ small enough, for $i=1,2$ and all $t\in[0,T]$ 
    \begin{equation}
    \label{eq:WeightEstimate}
    \begin{aligned}
    &\int_{\R^{2d}} \calK_h^{x,y} | n_\eta^{(i),x}- n_\eta^{(i),y}|(\upsilon_\eta^{x}+\upsilon_\eta^{y}) \dx x\dx y\\
    &\leq C\int_{\R^{2d}} \calK_h^{x,y} |n_\eta^{(1),\mathrm{in},x} - n_\eta^{(1),\mathrm{in},y}| \dx x\dx y\\
    &\quad+C\int_{\R^{2d}} \calK_h^{x,y} |n_\eta^{(2),\mathrm{in},x}-n_\eta^{(2),\mathrm{in},y}| \dx x\dx y\\
    &\quad + C|\log h|^{1/2}\prt*{\int_0^T\iintO{\calK_h^{x,y}\abs*{W_\eta^x - W_\eta^y}^2}\dx t}^{1/2}\\
    &\quad + C|\log h|^{1/2}\prt*{\int_0^T\iintO{\calK_h^{x,y}\abs*{\Delta V^x-\Delta V^y}^2}\dx t}^{1/2}\\
    &\quad +C|\log h|^{1/2},
    \end{aligned}
    \end{equation}
    where the constants depend only on  $L^\infty(0,T;L^1\cap L^2(\Rd))$-norms of the densities and the pressure and the data. In particular, they are independent of $\eta$.
\end{proposition}
\begin{proof}
     
    For $i =1,2$, we define the quantities
 \[
    Q_{n^{(i)}_\eta}(t) := \int_{\R^{2d}}\calK_h^{x,y} \abs{ n_\eta^{(i),x}-  n_\eta^{(i),y}} (\upsilon_\eta^{x}+ \upsilon_\eta^{y}) \dx x \dx y,
 \]
which we will propagate in time. Then, to close the ensuing Gronwall estimate, we will also need to define
 \[
    Q_{n_\eta}(t) := \int_{\R^{2d}} \calK_h^{x,y} \abs{ n_\eta^{x} -  n_\eta^{y}}(\upsilon_\eta^x+\upsilon_\eta^y) \dx x \dx y.
 \]
Throughout, we suppress the time dependence for brevity and we use the notation from before,
 \begin{equation*}
     \mathcal{V}_\eta^x := W_\eta^x + V^x.
 \end{equation*}
Then, for  two solutions $n_\eta^{(i),x}, n_\eta^{(i),y}$, for $x,y \in \R^d$, associated to the regularised initial data, we find
 \begin{align*}
        \fpartial t &\prt*{n_\eta^{(i),x} -  n_\eta^{(i),y}}
     + \nabla_x \cdot \prt*{-\nabla \mathcal{V}_\eta^x\prt{n_\eta^{(i),x} - n_\eta^{(i),y}}}\\[3pt]
     &+ \nabla_y \cdot \prt*{-\nabla \mathcal{V}_\eta^y\prt{n_\eta^{(i),x} - n_\eta^{(i),y}}} \\[3pt]
     =\, &-\frac 1 2 \prt*{\Delta \mathcal{V}_\eta^x + \Delta \mathcal{V}_\eta^y} \prt{n_\eta^{(i),x} - n_\eta^{(i),y}}\\[3pt]
    &+ \frac 1 2 \prt*{\Delta \mathcal{V}_\eta^x - \Delta \mathcal{V}_\eta^y} \prt*{n_\eta^{(i),x} + n_\eta^{(i),y}}
     + n_\eta^{(i),x} G^{(i)} (p_\eta^x) - n_\eta^{(i),y} G^{(i)} (p_\eta^y).
 \end{align*}
Upon multiplying with $\sigma_i := \mathrm{sign} \prt{n_\eta^{(i),x} - n_\eta^{(i),y}}$, we get
 \begin{align*}
     \begin{split}
         \fpartial t \abs*{n^{(i),x}_\eta - n_\eta^{(i),y}}
         = 
         &- \frac 1 2 \prt*{\Delta \mathcal{V}_\eta^x + \Delta \mathcal{V}_\eta^y} \abs*{n^{(i),x}_\eta - n_\eta^{(i),y}} \\[3pt]
         &+ \frac 1 2 \prt*{\Delta \mathcal{V}_\eta^x - \Delta \mathcal{V}_\eta^y} \prt{ n^{(i),x}_\eta + n_\eta^{(i),y}} \sigma_i \\[3pt]
         & + \nabla_x \cdot \prt*{\nabla \mathcal{V}_\eta^x  \abs*{n^{(i),x}_\eta - n_\eta^{(i),y}}} \\[3pt]
         & + \nabla_y \cdot \prt*{\nabla \mathcal{V}_\eta^y  \abs*{n^{(i),x}_\eta - n_\eta^{(i),y}}} \\[3pt]
         & + \prt*{n^{(i),x}_\eta G^{(i)}(p_\eta^x) - n^{(i),y}_\eta G^{(i)}(p_\eta^y)} \sigma_i.    \end{split}
 \end{align*}

To compute $\ddt Q_{n^{(i)}_\eta}$, we integrate the above equality in time and use the symmetry of the kernel. This yields
 \begin{align*}
    &Q_{n^{(i)}_\eta}(t) - \;Q_{n^{(i)}_\eta}(0)\\[3pt]
    &= \int_0^t\int_{\R^{2d}}\calK_h^{x,y} \prt*{\Delta \mathcal{V}_\eta^x - \Delta \mathcal{V}_\eta^y}  \prt*{n_\eta^{(i),x} + n_\eta^{(i),y} + (n_\eta^{(i),x} - n_\eta^{(i),y})}\sigma_i \upsilon_\eta^{x} \dx x\dx y\dx\tau \\[6pt]
    &\qquad - \int_0^t \int_{\R^{2d}}\nabla \calK_h^{x,y}\cdot \prt*{\nabla \mathcal{V}_\eta^x - \nabla \mathcal{V}_\eta^y } \abs*{n^{(i),x}_\eta - n_\eta^{(i),y}}(\upsilon_\eta^{x}+\upsilon_\eta^{y}) \dx x\dx y \dx \tau \\[6pt]
    &\qquad + \int_0^t\int_{\R^{2d}}\calK_h^{x,y} \prt*{n_\eta^{(i),x} G^{(i)}(p_\eta^x) - n_\eta^{(i),y} G^{(i)}(p_\eta^y) } \sigma_i (\upsilon_\eta^{x}+\upsilon_\eta^{y}) \dx x\dx y \dx\tau\\[6pt]
    &\qquad + 2\int_0^t\int_{\R^{2d}} \calK_h^{x,y} | n_\eta^{(i),x}- n_\eta^{(i),y}| \prt*{\fpartial t \upsilon_\eta^{x}-\nabla \mathcal{V}_\eta^x\cdot\nabla \upsilon_\eta^{x}-\Delta \mathcal{V}_\eta^x\upsilon_\eta^{x}}\dx x \dx y \dx \tau
    \\[3pt]
    &\equiv \mathcal A_1^{(i)} + \mathcal A_2^{(i)} + \mathcal A_3^{(i)} + \mathcal A_4^{(i)}.
 \end{align*}
We obtain a similar equality for the quantity $Q_{n_\eta}$ involving the total population density -- the corresponding terms on the right-hand side will be denoted by $\mathcal{A}^{(tot)}_\ell$, for $\ell=1,\dots,4$. In particular, note that
\begin{align*}
    \mathcal A^{(tot)}_3 = \sum_{i=1}^2\int_0^t \int_{\R^{2d}}\calK_h^{x,y} \Big(n_\eta^{(i),x} G^{(i)}(p_\eta^x) &- n_\eta^{(i),y} G^{(i)}(p_\eta^y)\Big) (\upsilon_\eta^{x}+\upsilon_\eta^{y}) \sigma \dx x\dx y \dx\tau,
\end{align*}
where $\sigma = \sign( n_\eta^{x}- n_\eta^{y})$. The other corresponding terms are the same, up to removing the $(i)$ superscript. We will now treat each term individually, making use of some cancellations between the individual species and the total population.\\

\underline{The highest order term}:
Notice that
\begin{align*}
    a_1\i &\coloneqq \prt{\Delta \mathcal{V}_\eta^x - \Delta \mathcal{V}_\eta^y}  \prt{n_\eta^{(i),x} + n_\eta^{(i),y}+(n_\eta^{(i),x} - n_\eta^{(i),y})} \sigma_i \upsilon_\eta^{x}\\[\alnspace]
    &= 2\prt*{\Delta  \mathcal{V}_\eta^x - \Delta  \mathcal{V}_\eta^y} \sigma_i  n_\eta^{(i),x}\upsilon_\eta^{x} \\[\alnspace]
    &= 2\brk*{\nu^{-1}\prt*{W_\eta^x - p_\eta^x - W_\eta^y + p_\eta^y} + \prt*{\Delta V^{x} - \Delta V^{y}}} \sigma_i  n_\eta^{(i),x}\upsilon_\eta^{x} \\[\alnspace]
    &\leq 2 \nu^{-1} \abs*{W_\eta^x-W_\eta^y} n_\eta^{(i),x} + 2 \nu^{-1} \abs*{p_\eta^x - p_\eta^y}   n_\eta^{(i),x}\upsilon_\eta^{x}\\[\alnspace]
    &\quad+ 2\abs*{\Delta V^{x} - \Delta V^{y}} n_\eta^{(i),x},
\end{align*}
using the fact that $|\sigma_i|\leq 1$ and $0\leq \upsilon_\eta\leq 1$. The term coming from the joint population reads
\begin{align*}
    a_1^{(tot)} &\coloneqq 2\prt*{\Delta  \mathcal{V}_\eta^x - \Delta  \mathcal{V}_\eta^y} \sigma  n_\eta^{x} \upsilon_\eta^{x}
    \\[\alnspace]
    &= 2\brk*{\nu^{-1} \prt*{W_\eta^x - p_\eta^x - W_\eta^y + p_\eta^y} + \prt*{\Delta V^{x} - \Delta V^{y}}} \sigma  n_\eta^{x} \upsilon_\eta^{x}\\[\alnspace]
    &\leq 2 \nu^{-1} \abs*{W_\eta^x-W_\eta^y} n_\eta^{x} - 2\nu^{-1} \abs*{p_\eta^x - p_\eta^y}   n_\eta^{x}\upsilon_\eta^{x}\\[\alnspace]
    &\quad + 2\abs*{\Delta V^{x} - \Delta V^{y}} n_\eta^{x},
\end{align*}
where we used the fact that  
$$
    \sigma = \sign( n_\eta^{x} -  n_\eta^{y}) = \sign(p_\eta^x-p_\eta^y).
$$
Adding all three terms up yields
\begin{align*}
    a_1\1 + a_1\2 + a_1^{(tot)} \leq 4 \nu^{-1} \abs*{W_\eta^x-W_\eta^y} n_\eta^{x} + 4\abs*{\Delta V^{x} - \Delta V^{y}} n_\eta^{x},
\end{align*}
due to the cancellation $n_\eta\1 + n_\eta\2 -n_\eta = 0$. Thus, using H\"older's inequality, we obtain the following bound
\begin{equation}
\begin{aligned}
\label{eq:A1-terms}
    &\mathcal{A}_1\1 + \mathcal{A}_1\2 + \mathcal{A}_1^{(tot)}\\[3pt]
    &\quad\leq 4 \nu^{-1} \norm{n_\eta}_{L^2(0,T;L^2(\R^d))} \norm{\calK_h}_{L^1(\R^d)}^{1/2}\prt*{\int_0^t\iintO{\calK_h^{x,y}\abs*{W_\eta^x-W_\eta^y}^2}\dx \tau}^{1/2}\\[3pt]
    &\qquad+ 4\norm{n_\eta}_{L^2(0,T;L^2(\R^d))} \norm{\calK_h}_{L^1(\R^d)}^{1/2}\prt*{\int_0^t\iintO{\calK_h^{x,y}\abs*{\Delta V^{x}-\Delta V^{y}}^2}\dx \tau}^{1/2}.
\end{aligned}
\end{equation}

\underline{The reaction terms}: 
For the individual species we can write
\begin{align*}
    a_3\i &:= \prt*{ n_\eta^{(i),x} G\i(p_\eta^x) -  n_\eta^{(i),y} G\i (p_\eta^y)} \sigma_i \\[\alnspace]
    &= G\i (p_\eta^x) \abs*{ n_\eta^{(i),x}- n_\eta^{(i),y}}  + \prt*{G\i (p_\eta^x)-G\i (p_\eta^y)} n_\eta^{(i),y} \sigma_i\\[\alnspace]
    &\leq  G\i (p_\eta^x) \abs*{ n_\eta^{(i),x}- n_\eta^{(i),y}}  + \abs*{G\i (p_\eta^x)-G\i (p_\eta^y)} n_\eta^{(i),y},
\end{align*}
having used the fact that $\abs{\sigma_i} \leq 1$. 
Similarly, for the sum we obtain 
\begin{align*}
    a_3^{(tot)}&:=\prt*{n_\eta^{(1),x} G\1(p_\eta^x) - n_\eta^{(1),y} G\1 (p_\eta^y)}\sigma\\[\alnspace]
    &\qquad+ \prt*{n_\eta^{(2),x} G\2(p_\eta^x) - n_\eta^{(2),y} G\2 (p_\eta^y)} \sigma \\[\alnspace]
    &= G\1 (p_\eta^x) \prt*{n_\eta^{(1),x}-n_\eta^{(1),y}} \sigma   + \prt*{G\1 (p_\eta^x)-G\1 (p_\eta^y)} n_\eta^{(1),y}\sigma\\[\alnspace]
    &\qquad + G\2 (p_\eta^x) \prt*{n_\eta^{(2),x}-n_\eta^{(2),y}} \sigma  + \prt*{G\2 (p_\eta^x)-G\2 (p_\eta^y)}n_\eta^{(2),y} \sigma\\[\alnspace]
    &=G\1 (p_\eta^x) \prt*{n_\eta^{(1),x}-n_\eta^{(1),y}} \sigma  - \abs*{G\1 (p_\eta^x)-G\1 (p_\eta^y)} n_\eta^{(1),y}\\[\alnspace]
    &\qquad + G\2 (p_\eta^x) \prt*{n_\eta^{(2),x}-n_\eta^{(2),y}} \sigma   -\abs*{G\2 (p_\eta^x)-G\2 (p_\eta^y)}n_\eta^{(2),y},
\end{align*}
where, in the last equality, we use again the facts that  
$\sign( n_\eta^{x} -  n_\eta^{y}) = \sign(p_\eta^x-p_\eta^y)$,
and that the functions $G\i$ are decreasing.
Adding the three contributions yields
\begin{align*}
   a_3\1+a_3\2+a_3^{(tot)} \leq \sum_{i=1}^2\brk*{G\i (p_\eta^x) \abs*{ n_\eta^{(i),x}- n_\eta^{(i),y}} +  G\i (p_\eta^x) \prt*{ n_\eta^{(i),x}- n_\eta^{(i),y}}  \sigma}. 
\end{align*}

Recall that the properties of $G\i$ imply that for each $i=1,2$ there is a constant $p_L\i$ such that $G\i(p)<0$ for $p>p_L\i$.
We can therefore estimate the terms on the right-hand side of the last inequality by considering the following cases:
\begin{itemize}
    \item if $p_\eta^x \leq p_L\i$, then
    \begin{equation*}
        G\i (p_\eta^x) \abs*{ n_\eta^{(i),x}- n_\eta^{(i),y}} +  G\i (p_\eta^x) \prt*{ n_\eta^{(i),x}- n_\eta^{(i),y}}  \sigma \leq 2C_G\abs*{ n_\eta^{(i),x}- n_\eta^{(i),y}};
    \end{equation*}
    \item if $p_\eta^x > p\i_L$, then
    \begin{align*}
	   &G\i (p_\eta^x) \abs*{ n_\eta^{(i),x}- n_\eta^{(i),y}} +  G\i (p_\eta^x)     \prt*{ n_\eta^{(i),x}- n_\eta^{(i),y}}  \sigma\\
	   &\leq
	   G\i (p_\eta^x) \abs*{ n_\eta^{(i),x}- n_\eta^{(i),y}} +  |G\i (p_\eta^x)|        \abs*{ n_\eta^{(i),x}- n_\eta^{(i),y}}\\
	   &=0
    \end{align*}
\end{itemize}

Finally, the entire sum of $\mathcal{A}_3$-terms can be estimated as
\begin{equation}\label{eq:A3-terms}
    \mathcal{A}_3\1 + \mathcal{A}_3\2 + \mathcal{A}_3^{(tot)} \leq 2C_G\int_0^t Q_{n\1_\eta}(\tau)\dx \tau + 2C_G\int_0^t Q_{n\2_\eta}(\tau)\dx \tau.
\end{equation}

\underline{The commutator}:
First, we make use of the following inequality
$$
    |\nabla \mathcal{V}_\eta^x-\nabla \mathcal{V}_\eta^y|\leq C|x-y|\prt*{\mathsf{D}_{|x-y|}\nabla \mathcal{V}_\eta^x+\mathsf{D}_{|x-y|}\nabla \mathcal{V}_\eta^y},
$$
where 
$$
    \mathrm{D}_\rho\nabla \mathcal{V}_\eta^x=\frac{1}{\rho} \int_{|z|\leq \rho} \frac{\abs*{D^2 \mathcal{V}_\eta (x+z)}}{|z|^{d-1}}\dx{z},
$$
the proof of which can be found in~\cite[Lemma 3.1]{Jabin2010}.

Thus, using also the symmetry of $\calK_h$, we get
\begin{align*}
    \mathcal{A}_2 & :=\mathcal{A}_2\1 + \mathcal{A}_2\2 + \mathcal{A}_2^{(tot)}\\
    &\leq C
    \int_0^t\int_{\R^{2d}}|x-y|\abs*{\nabla\calK_h^{x,y}}\prt*{\mathsf{D}_{|x-y|}\nabla \mathcal{V}_\eta^x+
    \mathsf{D}_{|x-y|} \nabla \mathcal{V}_\eta^y}\times  \\
    &\hspace{3cm}\times\prt*{\sum_{i=1}^2\abs*{ n_\eta^{(i),x}- n_\eta^{(i),y}} + \abs*{ n_\eta^{x}- n_\eta^{y}}}\upsilon_\eta^{x}\dx{x}\dx{y}\dx\tau.
\end{align*}
Now, we split the $\R^{2d}$ domain into two regions: the annulus $\{1 < |x-y| < 2\}$ and its complement. On the latter we have the inequality $|z||\nabla\calK_h(z)| \leq C \calK_h(z)$, while on the annulus we have $|z||\nabla \calK_h(z)| \leq C (\calK_h(z) + |\log |\log h||)$ --- both estimates follow from the precise definition of the kernel and are proved in Lemma~\ref{lem:prop-gradK}. Therefore, the term $\mathcal{A}_2$ can be estimated as
\begin{align*}
    \mathcal{A}_2 &\leq C|\log |\log h|| + C\sum_{i=1}^2\int_0^t\int_{\R^{2d}} \calK_h^{x,y}\prt*{\mathsf{D}_{|x-y|}\nabla \mathcal{V}_\eta^x+
    \mathsf{D}_{|x-y|} \nabla \mathcal{V}_\eta^y}\\
    &\hspace{8cm}\abs*{ n_\eta^{(i),x}- n_\eta^{(i),y}}\upsilon_\eta^{x}\dx{x}\dx{y}\dx\tau,
\end{align*}
where the first constant depends on the $L^\infty(0,T;L^1(\Rd)\cap L^2(\Rd))$ norm of $n\i_\eta$ and the $L^1(0,T;L^2(\Rd))$ norm of $D^2\mathcal{V}_\eta$.
We will now focus on the last term of the previous inequality --- let us denote it by $\tilde{\mathcal{A}}_2$.
Using that
$$
    \mathsf{D}_{|x-y|}\nabla \mathcal{V}_\eta^y+\mathsf{D}_{|x-y|}\nabla \mathcal{V}_\eta^x = \mathsf{D}_{|x-y|}\nabla \mathcal{V}_\eta^y- \mathsf{D}_{|x-y|}\nabla \mathcal{V}_\eta^x + 2\mathsf{D}_{|x-y|}\nabla \mathcal{V}_\eta^x,
$$ 
and changing the variables $z=x-y$, we may apply the Cauchy-Schwarz inequality
and use the uniform $L^\infty(0,T; L^2(\Rd))$-bounds on $n\i_\eta$ and $n_\eta$ to deduce
\begin{align*}
    \tilde{\mathcal{A}}_2
    \leq&\ C(\norm{n\1_\eta}_{L^\infty(0,T;L^2(\Rd))} + \norm{n\2_\eta}_{L^\infty(0,T;L^2(\Rd))})\times\\[\alnspace]
    &\quad\times\int_0^t\int_{\R^{d}} \calK_h(z) \norm*{\mathsf{D}_{|z|}\nabla \mathcal{V}_\eta (\cdot) - \mathsf{D}_{|z|}\nabla \mathcal{V}_\eta (\cdot+z)}_{L^2(\Rd)}\dx{z}\dx\tau \\[\alnspace]
    & + C\sum_{i=1}^2 \int_0^t\iintO{\calK_h^{x,y} \mathsf{D}_{|x-y|} \nabla \mathcal{V}_\eta^x\abs*{ n_\eta^{(i),x}- n_\eta^{(i),y}} \upsilon_\eta^{x}}\dx\tau.
\end{align*}
We may bound $\mathsf{D}_{|x-y|}\nabla \mathcal{V}_\eta$ by the maximal operator $\mathsf{M}\abs*{D^2 \mathcal{V}_\eta}$, whence
\begin{align}
\begin{split}\label{estimA1}
    \tilde{\mathcal{A}}_2 \leq & 
    C\int_0^t\int_{\R^{d}} \calK_h^z \norm*{\mathsf{D}_{|z|}\nabla \mathcal{V}_\eta (\cdot) - \mathsf{D}_{|z|}\nabla \mathcal{V}_\eta (\cdot+z)}_{L^2(\Rd)}\dx{z}\dx\tau \\[\alnspace]
    & + C\sum_{i=1}^2 \int_0^t\iintO{\calK_{h}^{x,y} \mathsf{M}\abs*{D^2 \mathcal{V}_\eta^x}\abs*{ n_\eta^{(i),x}- n_\eta^{(i),y}} \upsilon_\eta^{x}}\dx\tau.
\end{split}
\end{align}
The last two terms on the right-hand side of inequality~\eqref{estimA1} will be controlled by the term $\mathcal{A}_4$ and the equation for the weight $\upsilon_\eta$.

\smallskip
From Eq.~\eqref{eq:weight}, we have
\begin{align*}
    \mathcal{A}_4 &:= \mathcal{A}_4\1 + \mathcal{A}_4\2 + \mathcal{A}_4^{(tot)}\\
    &= \int_0^t\iintO{\calK_h^{x,y} \sum_{i=1}^2\abs*{ n_\eta^{(i),x} - n_\eta^{(i),y}} (-2\lambda B_\eta(x)-2\Delta \mathcal{V}_\eta^x) \upsilon_\eta^{x}}\dx\tau \\[\alnspace]
    &\quad + \int_0^t\iintO{\calK_h^{x,y}\abs*{ n_\eta^{x}- n_\eta^{y}}\prt*{-2\lambda B_\eta(x)-2\Delta \mathcal{V}_\eta^x} \upsilon_\eta^{x}}\dx\tau.
\end{align*}    
    
Therefore, combining the latter equality with Eq.~\eqref{estimA1}, we deduce
\begin{align*}
    \tilde{\mathcal{A}}_2+\mathcal{A}_4 \leq & \ 
    C\int_0^t\int_{\R^{d}} \calK_h(z) \norm*{\mathsf{D}_{|z|}\nabla \mathcal{V}_\eta (\cdot) - \mathsf{D}_{|z|}\nabla \mathcal{V}_\eta (\cdot+z)}_{L^2(\Rd)}\dx{z}\dx\tau \\[\alnspace]
    &+ \int_0^t\iintO{ \calK_h^{x,y}\sum_{i=1}^2 \abs*{ n_\eta^{(i),x}- n_\eta^{(i),y}} \upsilon_\eta^{x}\times\\[\alnspace]
    &\qquad\quad
    \times\prt*{C\mathsf{M}|D^2 \mathcal{V}_\eta^x| + 2|\Delta \mathcal{V}_\eta^x| - 2\lambda B_\eta(x)}}\dx\tau.
\end{align*}
From the choice of $B_\eta$ in the weight equation Eq.~\eqref{eq:weight}, we can find $\lambda$
large enough, independently of $\eta$, such that 
\begin{align*}
    \mathcal{A}_2+\mathcal{A}_4 &\leq  
    C\int_0^t\int_{\R^{d}} \calK_h(z) \norm*{\mathsf{D}_{|z|}\nabla \mathcal{V}_\eta (\cdot) - \mathsf{D}_{|z|}\nabla \mathcal{V}_\eta (\cdot+z)}_{L^2}\dx{z}\dx\tau + C|\log |\log h||.
\end{align*}

We shall now employ \cite[Lemma 6.3]{BreschJabin} which provides a constant $C>0$ such that for any $u\in H^1(\R^d)$,
\begin{equation*}
    \int_{\R^{d}} \calK_h(z) \norm*{\mathsf{D}_{|z|}u(\cdot) - \mathsf{D}_{|z|}u(\cdot+z)}_{L^2(\R^d)}\dx{z}
    \leq C |\log h|^{1/2} \norm*{u}_{H^1(\R^d)}.
\end{equation*}
In our context, observing that $u=\nabla \mathcal{V}_\eta$ belongs to $H^1(\Rd)$, uniformly in $\eta$, by Assumption \ref{assu:integrable-data}, we have the following bound
\begin{equation*}
    \int_{\R^{d}} \calK_h(z) \norm*{\mathsf{D}_{|z|}\nabla \mathcal{V}_\eta(\cdot) - \mathsf{D}_{|z|}\nabla \mathcal{V}_\eta(\cdot+z)}_{L^2(\R^d)}\dx{z}
    \leq C |\log h|^{1/2} \norm*{\nabla \mathcal{V}_\eta}_{H^1(\R^d)}.
\end{equation*}
Thus, we finally have the estimate
\begin{align}\label{eq:A2A4-terms}
    \mathcal{A}_2+\mathcal{A}_4  \leq  C |\log h|^{1/2} \int_0^t\norm*{\mathcal{V}_\eta(\tau)}_{H^{2}(\R^d)}\dx \tau + C|\log |\log h||\leq C|\log h|^{1/2},
\end{align}
for $h$ small enough, where the constant depends only on the $L^\infty(0,T; L^1(\Rd)\cap L^2(\Rd))$-norm of the densities $n\i_\eta$ and the $L^1(0,T; H^{2}(\Rd))$-norm of the velocity and drift potentials $W_\eta$ and $V$. Thus, this constant is independent of $\eta$.\\

\underline{Conclusion}:
Collecting all the estimates~\eqref{eq:A1-terms}, \eqref{eq:A3-terms}, and \eqref{eq:A2A4-terms}, we deduce
\begin{align*}
    Q_{n\1_\eta}(t) &+ Q_{n\2_\eta}(t) + Q_{n_\eta}(t) \\[3pt]
    &\leq C\int_0^t Q_{n\1_\eta}(\tau) + Q_{n\2_\eta}(\tau)\dx\tau + (Q_{n\1_\eta}(0) + Q_{n\2_\eta}(0) + Q_{n_\eta}(0)) + C|\log h|^{1/2}\\[3pt]
    &\quad + C|\log h|^{1/2}\prt*{\int_0^T\iintO{\calK_h^{x,y}\abs*{W_\eta^x-W_\eta^y}^2}\dx\tau}^{1/2}\\[3pt]
    &\quad + C|\log h|^{1/2}\prt*{\int_0^T\iintO{\calK_h^{x,y}\abs*{\Delta V^{x}-\Delta V^{y}}^2}\dx\tau}^{1/2},
\end{align*}
where the constants are independent of $\eta$.
Let us stress at this point that all the constants above are independent of $k$ whenever the initial densities and pressure are uniformly bounded in $L^1(\Rd)\cap L^2(\Rd)$.
Applying Gronwall's lemma, we obtain~\eqref{eq:WeightEstimate}.
\end{proof}

\subsubsection{Propagation of compactness -- removing the weights}
\begin{proposition}
	\label{prop:CompactnessEta}
	\begin{enumerate}
    	\item The following estimate holds for any $t\in[0,T]$ and any $\zeta\in(0,1)$:
	    \begin{equation}\label{eq:CompactnessEstimateEta}
    	\begin{aligned}
        		&\int_{\R^{2d}} \overline{\calK}_h^{x,y}|n_\eta^{(i),x}  - n_\eta^{(i),y}| \dx x \dx y\\[3pt]
	        &\leq \frac{C}{\zeta}\sum_{j=1}^2\int_{\R^{2d}} \overline{\calK}_h^{x,y} |n_\eta^{(j),\mathrm{in},x} - n_\eta^{(j),\mathrm{in},y}| \dx x\dx y\\[3pt]
	        &\quad + \frac{C}{\zeta}\prt*{\int_0^T\iintO{\overline{\calK}_h^{x,y}\abs*{W_\eta^x-W_\eta^y}^2}\dx t}^{1/2}\\[3pt]
	        &\quad + \frac{C}{\zeta}\prt*{\int_0^T\iintO{\overline{\calK}_h^{x,y}\abs*{\Delta V^x-\Delta V^y}^2}\dx t}^{1/2}\\[3pt]
	        &\quad +\frac{C}{\zeta}|\log h|^{-1/2} + \frac{C}{|\log\zeta|},
    	\end{aligned}
    \end{equation}
    where all constants are independent of $\eta$. 
    \item The densities $n_\eta\i$ satisfy
    \begin{equation}\label{eq:CompactnessEta}
        \lim_{h\to 0}\;\limsup_{\eta\to0}\;\sup_{t\in[0,T]}\;\int_{\R^{2d}} \overline{\calK}_h^{x,y} \abs*{n_\eta^{(i),x} -  n_\eta^{(i),y}} \dx x \dx y = 0.
    \end{equation}
\end{enumerate}
    
\end{proposition}
\begin{proof}
    We will make use of the properties of the weights proved in Proposition~\ref{prop:w}. To this end, we fix a number $\zeta\in(0,1)$ and introduce the set $\Omega_\zeta:=\{x\;:\; \upsilon_\eta\leq \zeta\}$. Now, we decompose the domain of integration:
    \begin{align*}
        \int_{\R^{2d}} \calK_h^{x,y}| n_\eta^{(i),x} -  n_\eta^{(i),y}| \dx x \dx y &= \int_{x\in\Omega_\zeta^c\lor y\in\Omega_\zeta^c} \calK_h^{x,y}| n_\eta^{(i),x} -  n_\eta^{(i),y}| \dx x \dx y\\[3pt]
        &\;+ \int_{x\in\Omega_\zeta\land y\in\Omega_\zeta} \calK_h^{x,y}| n_\eta^{(i),x} -  n_\eta^{(i),y}| \dx x \dx y. 
    \end{align*}
    For the first term, we observe
    \begin{align*}
        \int_{x\in\Omega_\zeta^c\lor y\in\Omega_\zeta^c} &\calK_h^{x,y}| n_\eta^{(i),x} -  n_\eta^{(i),y}| \dx x \dx y\\[3pt]
        &\leq \frac{1}{\zeta}\int_{\R^{2d}} \calK_h^{x,y}| n_\eta^{(i),x} -  n_\eta^{(i),y}|(\upsilon_\eta^{x}+\upsilon_\eta^{y}) \dx x \dx y\\[3pt]
        &= \frac{1}{\zeta}Q_{n\i_\eta}(t).
    \end{align*}
    For the second term, we use the symmetry of $\calK_h$ to write
    \begin{align*}
        \int_{x\in\Omega_\zeta\land y\in\Omega_\zeta} \calK_h^{x,y}| n_\eta^{(i),x} -  n_\eta^{(i),y}| \dx x \dx y &\leq 2\int_{x\in\Omega_\zeta\land y\in\Omega_\zeta} \calK_h^{x,y} n_\eta^{(i),x}\dx y\dx x \\[3pt]
        &\leq C|\log h| \int_{x\in\Omega_\zeta} n_\eta^{(i),x}\dx x\\[3pt]
        &\leq C|\log h|\frac{1}{|\log\zeta|} \int_{\Rd} n_\eta^{(i),x}|\log \upsilon_\eta^{x}|\dx x\\[3pt]
        &\leq C|\log h|\frac{1}{|\log\zeta|}.
    \end{align*}
    It follows that, for any $\zeta \in (0,1)$,
    \begin{align*}
        \int_{\R^{2d}} \calK_h^{x,y}| n_\eta^{(i),x} -  n_\eta^{(i),y}| \dx x \dx y \leq \frac{1}{\zeta}Q_{n\i_\eta}(t) + C\frac{|\log h|}{|\log\zeta|}.
    \end{align*}
    Dividing by the norm of $\calK_h$ and using estimate~\eqref{eq:WeightEstimate}, we obtain the desired estimate~\eqref{eq:CompactnessEstimateEta}. 
    To deduce~\eqref{eq:CompactnessEta}, we can write
    \begin{equation*}
        \limsup_{\eta\to0}\;\sup_{t\in[0,T]}\;\int_{\R^{2d}} \overline{\calK}_h^{x,y}| n_\eta^{(i),x} -  n_\eta^{(i),y}| \dx x \dx y \leq \frac{\beta(h)}{\zeta(h)} + \frac{C}{|\log\zeta(h)|},
    \end{equation*}
    where $\beta(h)\to0$ as $h\to0$. The choice $\zeta(h) = \beta(h)^{1/2}$ guarantees that the right-hand side vanishes as $h\to0$.
\end{proof}
\begin{corollary}\label{cor:ConvergenceNEta}
    The densities $\{n\i_\eta\}_{\eta>0}$ are compact in $L^q(0,T;L^p(\Rd))$, for any $q\in[1,\infty)$ and $p\in[1,2)$. We call their limits $n\i$.
\end{corollary}
\begin{proof}
    The compactness criterion~\ref{lem:CompactnessCriterion} together with Proposition~\ref{prop:TimeDerivatives1} provide compactness of $n\i_\eta$ in $L^1(0,T;L^1_{\mathrm{loc}}(\Rd))$. Since the sequence of initial data is tight, so is the entire sequence at any later time -- as shown in the proof of Proposition~\ref{prop:CompactW_eta}. Thus, we deduce global compactness in $L^1(0,T;L^1(\Rd))$. Interpolating with the uniform bounds in $L^\infty(0,T;L^2(\Rd))$, we conclude the proof.
\end{proof}

\subsection{The limit $\eta\to0$}
We are now ready to state and prove the main result of this section.
\begin{proposition}\label{prop:ExistenceEtaTo0}
    For any $\eta>0$, let the initial conditions be given as in Subsection~\ref{subsec:DataApproximation} and let $(n\1_\eta, n\2_\eta, W_\eta)$ be an associated weak solution of System~\eqref{eq:main-system}. Then, up to extraction of a subsequence, the following convergences hold:
    \begin{alignat*}{2}
        W_\eta&\to W,\quad &&\text{in $L^2(0,T;L^2(\Rd))$},\\
        n\i_\eta &\to n\i,\quad &&\text{in $L^1(0,T;L^1(\Rd))$},\\
        p_\eta&\to p,\quad &&\text{a.e.\ in $\Rd \times (0,T)$},
    \end{alignat*}
    where the functions $(n\i, p, W)$ are a weak solution of the system
    \begin{align*}
    \partialt {n\i} &= \nabla \cdot(n\i\nabla W) + \nabla \cdot(n\i\nabla V) + n\i G\i(p),\\
    -\nu \Delta W + W &= p\\
    n\i(0,x) &= n^{(i),in}
    \end{align*}
    in the sense of Definition~\ref{def:weak-sol-integrable}.
\end{proposition}
\begin{proof}
    We have already established the strong convergences of $W_\eta$ and $n\i_\eta$. Passing to a subsequence we can assume these convergences hold also almost everywhere. Then, the sequence $p_\eta = n_\eta^k$ converges a.e.\ to $p:=n^k$.

    To pass to the limit $\eta\to 0$ in the reaction term, we observe that $n\i_\eta G\i(p_\eta) \to n\i G\i(p)$ a.e.\ and
    \begin{equation*}
        |n\i_\eta G\i(p_\eta)| \leq \Gamma(n\i_\eta + n\i_\eta p_\eta^s). 
    \end{equation*}
    Using the uniform bounds $n_\eta\i\in L^\infty(0,T;L^2(\Rd))$ and $p_\eta \in L^\infty(0,T;L^\frak q(\Rd))$, $\frak q>2s$, we deduce that the product $\{n\i_\eta G\i(p_\eta)\}_{\eta>0}$ is uniformly bounded, and hence weakly precompact, in some Lebesgue space $L^\frak{r}(\Rd \times (0,T))$ with $\frak{r}>1$. Using the a.e.\ convergence, we identify its weak limit as $n\i G\i(p)$. Passing to the limit in all the other terms is even more straightforward. Finally, the continuity in time of $n^{(i)}$ follows from the same argument as in the proof of Proposition \ref{lem:RemoveEpsilon}, using Lemma~\ref{lem:Arzela-ascoli}. However, this time we can only interpolate up to $L^\infty(0,T;L^2(\Rd))$, such that $n^{(i)} \in C([0,T];L^q(\Rd))$, $1\leq q < 2$. Concerning the continuity in time of the pressure, we note
    \begin{align*}
        \norm{p(t) - p(s)}_{L^1(\Rd)}   \leq C(k) \norm{p}_{L^\infty(0,T; L^{2+}(\Rd))} \norm{n(t) - n(s)}_{L^{2-}(\Rd)},
    \end{align*}
    so that $p\in C([0,T];L^q(\Rd))$, for any $1\leq q < \frak q$.
\end{proof}

\section{Stiff limit}
\label{sec:StiffLimit}
This section is dedicated to establishing the stiff limit, $k\to \infty$ for solutions in the sense of Definition \ref{def:weak-sol-integrable}, as the density-pressure relation becomes a graph. Before we pass to the limit in Subsection \ref{sec:limit-passage}, we have to derive further uniform-in-$k$ estimates. Henceforth, we shall assume 
\begin{equation}
    \label{eq:InitialSupport}
    \Big|\bigcup_{k=1}^\infty\mathrm{supp} (n^{\mathrm{in}}_k)\Big| < \infty,
\end{equation}
which is indispensable for the incompressible limit and the derivation of uniform bounds. 
Note that this assumption is implied by those imposed in \cite{PV2015}. However, Eq.~\eqref{eq:InitialSupport} is weaker in the sense that we do not require the quantity $Q_k^\mathrm{in}$, appearing in Lemma \ref{lemma:pQinL1}, to vanish as $k\to \infty$.

We assume also the following integrability uniformly in $k$ 
    \begin{align*}
        n_k^{(i), \mathrm{in}} &\in L^1(\Rd)\cap L^{\frak{p}}(\Rd),\quad \frak{p} = \max(2,(d/2)^+),\\
        p_k^{\mathrm{in}} &\in L^1(\Rd)\cap L^\frak{q}(\Rd),\quad \frak{q} > 2s\frak{p}',
    \end{align*}
    \begin{itemize}
    \item $\frak{p} = \max(2,(d/2)^+)$, for the densities,
    \item  $\frak{q} > 2s\frak{p}'$, for the pressures,
\end{itemize}
where $\frak{p}'$ is the H\"older conjugate of $\frak p$ (i.e., $\frak p' = \min(2,(\frac{d}{d-2})^{-})$ for $d\geq 2$ and $\frak p' = 2$ in one dimension). Notice that then we always have $p_k^{\mathrm{in}}\in L^2(\Rd)$ and $p_k^{\mathrm{in}}\in L^{2s}(\Rd)$. Furthermore, recall that this integrability is propagated in time by Lemma \ref{lem:propagation-of-LP-for-n} and Corollary \ref{cor:LebesguePressure}.
Finally, we assume that
\begin{equation*}
    n^{(i),\mathrm{in}}_k \to n^{(i),\mathrm{in}}_\infty, \quad\text{in $L^1(\Rd)$ as $k\to\infty$}.
\end{equation*}

For the sake of exposition we shall drop the subscript $k$, but we stress that all bounds derived below are uniform in $k$.

\subsection{Further a priori estimates}
In order to obtain further regularity, it is important to study the evolution of the relative population, $r = n^{(1)} / (n^{(1)} + n^{(2)})$, with the convention that $r=0$ on $n=0$), first introduced in \cite{BHIM2012, BGH1987a}. Indeed, a short computation shows that
\begin{align*}
    \partialt r = \nabla r \cdot \nabla (W + V) + r (1-r) (G^{(1)}(p) - G^{(2)}(p)).
\end{align*}
Before stating the main estimates of this subsection, let us introduce an auxiliary localisation.

\begin{lemma}[Localising test function]
\label{lem:VarPi}
For any $k>1$ there exists a unique renormalised solution $\varpi \in L^\infty(0,T;L^1\cap L^\infty(\Rd))$ of the equation
    \begin{align}
        \label{eq:VarPi}
        \partialt \varpi - \nabla (W + V) \cdot \nabla \varpi = 0,
    \end{align}
    with initial condition $\varpi(\cdot, 0) =\mathbbm{1}_{\mathrm{supp}(n^{\mathrm{in}})}$. Moreover, there holds
    \begin{align}
        \label{eq:VarPiSupport}
        \mathrm{supp}(\varpi(\cdot,t)) \supset \mathrm{supp}(n(\cdot,t)).
        \end{align}
\end{lemma}
\begin{proof}
Similarly as in Proposition~\ref{prop:w} the existence of a unique solution of equation~\eqref{eq:VarPi} follows from the DiPerna-Lions theory -- indeed, $\varpi$ is transported with the same velocity field as $\upsilon$. Moreover, we have the bound $0\leq \varpi \leq 1$, and 
\begin{equation*}
    \ddt \norm{\varpi(t)}_{L^1(\Rd)} \leq \frac12\norm{\Delta(W+V)}_{L^2(\Rd)} + \frac12\norm{\varpi(t)}_{L^1(\Rd)}.
\end{equation*}
It follows that $\varpi \in L^\infty(0,T;L^q(\Rd))$ for any $q\in[1,\infty]$, uniformly in $k$.

To establish~\eqref{eq:VarPiSupport}, we consider the equation
    \begin{align*}
        \fpartial t((1-\varpi)n) = \nabla \cdot((1-\varpi)n\nabla(W+V)) + (1-\varpi)\sum_{i=1}^2n\i G\i(p).        
    \end{align*}
Upon integrating in space, we have
    \begin{align*}
        \ddt \int_{\Rd} (1-\varpi)n \dx x &= \sum_{i=1}^2 \int_\Rd (1-\varpi)n\i G\i(p) \dx x\\
        &\leq \max(G\1(0), G\2(0)) \int_\Rd (1-\varpi) n\dx x,
    \end{align*}
whence, for any $t\in(0,T)$,
    \begin{align*}
        0\leq \int_\Rd (1-\varpi(t))n (t) \dx x \leq C_G \int_\Rd (1-\varpi(0))n^{\mathrm{in}}\dx x = 0,
    \end{align*}   
and~\eqref{eq:VarPiSupport} follows.   
\end{proof}

\begin{lemma}
    \label{lemma:pQinL1}
    There exists a constant $C>0$ independent of $k$ such that for all $k > (\alpha \nu)^{-1}$, there holds
    \begin{equation*}
        k\int_0^T\!\!\int_{\R^d}  p |W- p + \nu\Delta V + \nu r G\1(p) + \nu (1-r)G\2(p)| \dx{x}\dx{t} \leq C.
    \end{equation*}
\end{lemma}
\begin{proof}
    We begin by defining the quantity 
    \begin{equation*}
        Q \coloneqq W - p + \nu \Delta V + \nu r G\1(p) + \nu (1-r)G\2(p),
    \end{equation*}
    and introduce, for convenience, 
    $$
        \beta_1 = G^{(1)}(p) - G^{(2)}(p), 
        \qquad
        \beta_2 := 1- \nu r \partial_p G^{(1)}(p) - \nu (1 - r)\partial_p G^{(2)}(p).
    $$
    Next, we recall the equation for the pressure
    \begin{align}
        \label{eq:pressure-revisited}
        \partialt{p} = \nabla p \cdot \nabla (V + W) + \frac{k}{\nu} p Q,
    \end{align}
    and the concentration
    \begin{align}
        \label{eq:concentration-revisited}
        \partialt r = \nabla r \cdot \nabla (V + W) + r(1-r) \beta_1,
    \end{align}
    Next, observe that
        \begin{align}
        \label{eq:Q-equation}
        \partialt Q &= \partialt W + \nu \partialt {\Delta V} + \nu \beta_1  \partialt r   - \beta_2 \partialt p,
    \end{align}
    with
    \begin{align}\label{eq:coeff-grad-p}
        \nabla Q = \nabla W + \nu \nabla \Delta V + \nu \beta_1  \nabla r - \beta_2 \nabla p.
    \end{align}
    Substituting Eq. \eqref{eq:concentration-revisited} and Eq. \eqref{eq:pressure-revisited} into Eq. \eqref{eq:Q-equation} yields
    \begin{align*}
        \partialt{Q} 
        &=\partialt W + \nu \partialt {\Delta V} + \nu \beta_1 \left[\nabla r \cdot \nabla(V+W) + \beta_1  r(1-r)\right] - \beta_2 \left[\nabla p \cdot \nabla (V+W) + \frac{k}{\nu}pQ \right].
    \end{align*}
    Using Eq. \eqref{eq:coeff-grad-p}, we obtain
    \begin{align*}
        \partialt{Q} 
        &= - \beta_2  \frac{k}\nu p Q  + \partialt W + \nu \partialt {\Delta V}\\[0.5em]
        & \qquad + \nu r (1-r) |\beta_1|^2  + \nu \beta_1  \nabla r \cdot \nabla (V+W)\\[0.5em]
        & \qquad - \nabla (V+W) \cdot \left[ \nabla W + \nu \nabla \Delta V + \nu \beta_1  \nabla r- \nabla Q \right]\\[0.5em]
        &=- \beta_2 \frac{k}\nu p Q  + \partialt W + \nu \partialt {\Delta V}  + \nu r (1-r)  |\beta_1|^2 \\[0.5em]
        & \qquad - \nabla (V+W) \cdot \left[ \nabla W + \nu \nabla \Delta V - \nabla Q \right].
    \end{align*}
    Moving terms depending on $Q$ on the left-hand side, we derive the equation
    \begin{equation*}
        \begin{aligned}
        \partialt{\abs*{Q}} &- \nabla{\abs*{Q}}\cdot \prt*{\nabla{W}+\nabla V} + \beta_2 \frac{k}{\nu}p\abs*{Q}\\[3pt] 
        &\leq \abs*{\nabla{W}}\prt*{\abs*{\nabla{W}} + \abs*{\nabla{V}}} + \nu\abs*{\partial_t\Delta V} + \abs*{K \star \brk*{\nabla{p}\cdot\nabla{W}+ \nabla{p}\cdot\nabla{V} + \frac{k}{\nu}p \abs*{Q}}}  \\[3pt]
        &\qquad + \nu r \prt*{1-r}\abs*{\beta_1 }^2 + \nu\abs*{\prt*{\nabla W + \nabla V}\cdot \nabla\Delta V}.
        \end{aligned}
    \end{equation*}
    Since by Assumption \ref{assu:existence} $G\i_p \leq - \alpha$, we observe that
    $$
        \beta_2 \geq 1 + \nu \alpha.
    $$
Now we multiply by $\varpi(x,t)$, the solution of the transport equation from Lemma~\ref{lem:VarPi}, and integrate in space and time to get
    \begin{equation*}
        \begin{aligned}
         &(1+\nu \alpha ) \frac{k}{\nu} \int_0^T\!\!\int_{\R^d} \varpi p\abs*{Q} \dx{x}\dx{t} \\[3pt]
          &\leq \int_{\R^d}  \varpi(x,0)\abs*{Q(x,0)} - \varpi(x,T)\abs*{Q(x,T)} \dx{x} + \int_0^T\!\!\int_\Rd |Q|\prt*{\partial_t\varpi - \nabla\varpi\cdot\nabla(W+V)}\dx x\dx t\\[3pt]
          &\quad - \int_0^T\!\!\int_{\R^d}\varpi\abs*{Q}\prt*{ \Delta W + \Delta V} \dx x \dx t + \int_0^T\!\!\int_{\R^d}\abs*{\nabla{W}}(\abs*{\nabla V} + \abs*{\nabla W}) \dx{x}\dx{t}\\[3pt]
          &\quad + \nu\int_0^T\!\!\int_\Rd \varpi \abs*{\partial_t\Delta V} \dx x \dx t +\int_0^T\!\!\int_{\R^d} \varpi\abs*{K \star \brk*{\nabla{p}\cdot\nabla{W}+ \nabla{p}\cdot\nabla{V} + \frac{k}{\nu}p \abs*{Q}}}  \dx{x}\dx{t}  \\[3pt]
          &\quad + \nu  \int_0^T\!\!\int_{\R^d} \varpi r(1-r)  \abs*{\beta_1}^2\dx{x}\dx{t} + \nu\int_0^T\!\!\int_\Rd \varpi\abs*{\prt*{\nabla W + \nabla V}\cdot \nabla\Delta V} \dx x\dx t.
        \end{aligned}
    \end{equation*}
    Notice that we used the property $\varpi\leq 1$ to omit the weight from some positive terms.
    
    Recall that $n(x,t)>0$ implies $\varpi(x,t) = 1$, for any $t\in(0,T)$, so that
    \begin{align*}
        (1+\nu \alpha ) \frac{k}{\nu}\int_0^T\!\!\int_\Rd \varpi p |Q| \dx x \dx t
        &= (1+\nu \alpha ) \frac{k}{\nu}\int_0^T\!\! \int_{\Rd}p |Q| \dx x \dx t.
    \end{align*}
    Furthermore, 
    $$
        \frac{k}{\nu}\int_0^T\!\! \int_\Rd \varpi K \star (p |Q|) \dx x \dx t \leq \frac{k}{\nu}\int_0^T \int_\Rd p |Q| \dx x \dx t,
    $$
    and, using the Brinkman equation
    $$
        -\int_0^T\!\!\int_\Rd \varpi |Q|\Delta W \dx x \dx t \leq \frac{1}{\nu}\int_0^T\!\!\int_\Rd p|Q|\dx x \dx t.
    $$
    Therefore, the estimate becomes
    \begin{align*}
        (\alpha k - \nu^{-1})\int_0^T\int_\Rd p |Q|\dx x \dx t \leq C(W, V) (1 + \norm{\varpi Q}_{L^1(0,T;L^1(\Rd))}) + I_1 + I_2 + I_3,
    \end{align*}
    having used Assumptions \ref{assu:existence}, Assumptions \ref{assu:stiff}, as well as the uniform control of $\nabla W$ through the bounds on $p$ and the fact that $\varpi$ in $L^1(0,T;L^1(\Rd))$, cf.\ Lemma \ref{lem:VarPi}, uniformly in $k$. Here,
    \begin{align*}
        I_1 &:= \int_{\R^d}  \varpi(x,0)\abs*{Q(x,0)} \dx x, \\
        I_2 &:= \nu  \int_0^T\!\int_{\R^d} \varpi r(1-r)  \abs*{\beta_1}^2\dx{x}\dx{t}, \\
        I_3 &:= \int_0^T\!\int_{\R^d} \abs*{K \star \brk*{\nabla{p}\cdot\nabla{(W+ V)}}}  \dx{x}\dx{t}.
    \end{align*}
    We note that 
    \begin{align*}
        I_1 &= \int_{\{n^{\mathrm{in}}>0\}} |Q(x,0)| \dx x\\
        &\leq \int_{\mathcal{N}} |K\star p^{\mathrm{in}} - p^{\mathrm{in}} + \nu\Delta V(x,0) + \nu r^{\mathrm{in}}G\1(p^{\mathrm{in}}) + \nu(1-r^{\mathrm{in}})G\2(p^{\mathrm{in}})| \dx x \leq C,
    \end{align*}
    where $\mathcal{N}\subset \Rd$ is a set of finite measure containing $\cup_{k\geq 1}\mathrm{supp}(n_k^{\mathrm{in}})$, cf.\ \eqref{eq:InitialSupport}.
    For the second term we observe
    \begin{align*}
        I_2 \leq 8  \Gamma^2 \int_0^T\int_\Rd \varpi r(1-r)  (1 + p^{2s}) \dx x \dx t \leq 8\Gamma^2(\norm{\varpi}_{L^1(0,T;L^1(\Rd))} + \norm{p}_{L^{2s}(0,T;L^{2s}(\Rd))}),
    \end{align*}
    by the uniform $L^1$-bounds on $\varpi$ and $L^{2s}$-bounds on $p$, and Lemma \ref{lem:VarPi}.
    Finally,
    \begin{equation*}
        K \star \brk*{\nabla{p}\cdot\nabla{(W_k+V)}} = \sum_{i=1}^d\frac{\partial K}{\partial x_i} \star \brk*{p\frac{\partial (W + V)}{\partial x_i}} - K \star \brk*{p\Delta (W +V)}.
    \end{equation*}
    Using $K\in W^{1,1}$, $p\in L^\infty L^2$, and the regularity of the velocity potentials, we conclude
    \begin{equation*}
        (\alpha k- \nu^{-1})\int_0^T\int_{\R^d} p|Q| \dx{x}\dx{t} \leq C,
    \end{equation*}
    where $C>0$ is independent of $k$.
\end{proof}

\begin{lemma}
    \label{prop:TimeDerivative}
    We have the following time regularities:
    \begin{enumerate}[label=(\roman*)]
        \item $\partial_t n \in L^2(0,T;H^{-\frak s}(\Rd))$, $\frak s>d/2 + 1$;
        \item $\partial_t p \in L^1(0,T;H^{-\frak s}(\Rd))$, $\frak s > d/2 + 1$;
        \item $\partial_t{W} \in L^{1}(0,T; L^{q}(\Rd))$,\, for $1\leq q < d/(d-2)$, (or $\infty$ in 1D).
    \end{enumerate} 
\end{lemma}
\begin{proof}
    The proof of the three statements are very similar and follow the same line of reasoning as Proposition  \ref{prop:TimeDerivatives1}. Indeed, the regularity for $n$ follows exactly the same manner. To prove ($ii$), we have to use the equation for $p$. Now we can control all terms uniformly in $k$. The only term that has to be estimated anew is
    \begin{align*}
        \int_\Rd kpQ\phi \dx x \leq C\norm{\phi}_{H^{\frak s}(\Rd)} \norm{kpQ}_{L^1(\Rd)}.
    \end{align*}
    We can leverage the uniform bound from Lemma \ref{lemma:pQinL1} to obtain the desired $L^1(0,T;H^{-\frak s}(\Rd))$ bound.
    Finally, the time regularity of $W$ follows from the regularity of $K$ and $\partial_t p$.
\end{proof}

\begin{lemma}\label{lem:npQinL1}
    There holds
    \begin{align}
         k \int_0^T\int_\Rd np|Q|\dx x\dx t \leq C,
    \end{align}
    where $C>0$ is independent of $k$.
\end{lemma}
\begin{proof}
    Similarly as in the proof of Lemma \ref{lemma:pQinL1}, we compute
    \begin{align*}
        \partialt {(n|Q|)} 
        &= |Q| \nabla \cdot(n \nabla (W + V)) + |Q| (n^{(1)}G^{(1)} + n^{(2)}G^{(2)}) \\[3pt]
        &\quad - \frac{k}\nu \beta_2 np|Q| + n\abs*{\partialt W} + \nu n \abs*{\partialt {\Delta V}} + \nu r(1-r) n |\beta_1|^2 \\[3pt]
        &\quad + n\nabla|Q|\cdot \nabla (W + V) + n\abs*{\nabla (W + V)} \cdot \abs*{\nabla W + \nu \Delta \nabla V}.
    \end{align*}
    Integrating in space-time and using an integration by parts in space yields
    \begin{align*}
        &\frac{k}\nu (1 + \alpha \nu) \int_0^T\int_\Rd np|Q|\dx x\dx t\\ 
        &= \int_\Rd n^\mathrm{in}|Q^\mathrm{in}| \dx x  - \int_\Rd (n|Q|)(T)\dx x  + C_G \mathcal B_1 + \nu \mathcal B_2 + \nu \mathcal B_3 +  \mathcal B_4  + \mathcal B_5,
    \end{align*}
    where
    \begin{align*}
        \mathcal B_1 &:= \int_0^T\int_\Rd n|Q|\dx x \dx t,\\
        \mathcal B_2 &:= \int_0^T\int_\Rd  n \abs*{\partialt {\Delta V}} \dx x \dx t, \\
        \mathcal B_3 &:= \int_0^T\int_\Rd r(1-r) n |\beta_1|^2 \dx x \dx t,\\
        \mathcal B_4 &:= \int_0^T\int_\Rd n\abs*{\nabla (W + V)} \cdot \abs*{\nabla W + \nu \Delta \nabla V}\dx x \dx t,\\
        \mathcal B_5 &:= \int_0^T\int_\Rd n\abs*{\partialt W} \dx x \dx t.
    \end{align*}
    Using  Assumptions \ref{assu:existence}, \ref{assu:integrable-data}, and \ref{assu:stiff}, as well as uniform bounds of the $L^\frak p$-norm of $n$ and the $L^{\frak q}$-norm of $p$, these terms can be uniformly bounded in $k$, and we refer to Appendix \ref{app:further-details}.
\end{proof}

\subsection{Limit passage}
\label{sec:limit-passage}
We have gathered all necessary additional a priori estimates to pass to the limit $k\to \infty$. The main result of this section is the following theorem.
\begin{theorem}\label{thm:StiffLimit}
    There exist functions $n\i_\infty \in L^\infty(0,T;L^\frak{p}(\Rd)) \cap C([0,T];L^q(\Rd))$, $q\in[1,\frak p)$, and $p_\infty \in L^\infty(0,T;L^1(\Rd)\cap L^\frak{q}(\Rd))$ such that, up to subsequences,
    \begin{alignat*}{2}
        n\i_k &\to n\i_\infty,\quad &&\text{in } L^1(0,T;L^1(\Rd)),\\
        p_k &\to p_\infty,\quad &&\text{in } L^1(0,T;L^1(\Rd)).
    \end{alignat*}
    Moreover, 
    \begin{equation*}
        W_k\to W_\infty:=K_\nu\star p_\infty,\quad \text{in } L^2(0,T;L^2(\Rd)),
    \end{equation*}
    and $(n\i_\infty, W_\infty, p_\infty)$ satisfy (in the weak sense) the system:
    \begin{align*}
        \partialt {n\i_\infty} &= \nabla \cdot(n\i_\infty \nabla W_\infty) + \nabla \cdot(n\i_\infty\nabla V) + n\i_\infty G\i(p_\infty),\\
        -\nu\Delta W_\infty + W_\infty &= p_\infty,\\
        n\i_\infty(x, 0) &= n_\infty^{(i),\mathrm{in}},
    \end{align*}
    with the pointwise almost everywhere relation $p_\infty(1-n_\infty)=0$.
    Finally, the following complementarity relation holds pointwise almost everywhere
    \begin{equation}\label{eq:ComplementarityRelation}
    	p_\infty\left(W_\infty-p_\infty + \nu \Delta V + \nu n^{(1)}_\infty G\1(p_\infty) + \nu n^{(2)}_\infty G\2(p_\infty)\right) = 0.
	\end{equation}
\end{theorem}
The proof of the theorem is done in several steps. The most important one is to guarantee strong convergence of the sequence of pressures. To achieve it, we shall first propagate compactness of the densities, for which we leverage the Gronwall-type inequality proved in Section~\ref{sec:UnboundedData}. Indeed, we have the following proposition.

\begin{proposition}\label{prop:DensityCompactK}
    The sequence $(n\i_k)_k$ is compact in $L^1(0,T;L^1(\Rd))$. Its limit, $n\i_\infty$, belongs to $C([0,T];L^q(\Rd))$,  for $q\in[1,\frak p)$.
\end{proposition}

\begin{proof}
    First, we notice that the global strong compactness of $W_k$ in $L^2(0,T;L^2(\Rd))$ follows from the same argument as in the proof of Proposition~\ref{prop:CompactW_eta}, now exploiting the uniform-in-$k$ estimates of Lemma \ref{prop:TimeDerivative}.
   We now consider the approximating sequence $n\i_{k,\eta}$ constructed as in Section~\ref{sec:UnboundedData} and exploit the fact that we already know that it converges strongly in $L^1(0,T;L^1(\Rd))$ to $n\i_k$.
    Passing to the limit $\eta\to 0$ in~\eqref{eq:CompactnessEstimateEta}, we deduce, for any $\zeta\in(0,1)$,
    \begin{equation}\label{eq:CompactnessEstimateK}
    \begin{aligned}
        &\int_{\R^{2d}} \overline{\calK}_h^{x,y}|n_k^{(i),x} - n_k^{(i),y}| \dx x \dx y\\[3pt]
        &\leq \frac{C}{\zeta}\sum_{j=1}^2\int_{\R^{2d}} \overline{\calK}_h^{x,y} |n_k^{(j),\mathrm{in},x} - n_k^{(j),\mathrm{in},y}| \dx x\dx y\\[3pt]
        &\quad + \frac{C}{\zeta}\prt*{\int_0^T\iintO{\overline{\calK}_h^{x,y}\abs*{W_k^x - W_k^y}^2}\dx t}^{1/2}\\[3pt]
        &\quad + \frac{C}{\zeta}\prt*{\int_0^T\iintO{\overline{\calK}_h^{x,y}\abs*{\Delta V^x-\Delta V^y}^2}\dx t}^{1/2}\\[3pt]
        &\quad +\frac{C}{\zeta}|\log h|^{-1/2} + \frac{C}{|\log\zeta|},
    \end{aligned}
    \end{equation}
    where all the constants are independent of $k$, as they depend only on the data and the $L^\infty(0,T;L^2(\Rd))$-norms of the densities and the pressure. The same argument as in the proof of Proposition~\ref{prop:CompactnessEta} shows that
    \begin{equation*}
        \lim_{h\to 0}\;\limsup_{k\to\infty}\;\sup_{t\in[0,T]}\;\int_{\R^{2d}} \overline{\calK}_h^{x,y}|n^{(i),x}_{k} - n^{(i),y}_{k}| \dx x \dx y = 0.
    \end{equation*}
    Combining the above limit with the uniform control of $\partial_tn\i_k$ in $L^2(0,T;H^{-\frak{s}}(\Rd))$, we can use the compactness criterion, Lemma~\ref{lem:CompactnessCriterion}, to conclude compactness. The continuity of the limit follows from Lemma~\ref{lem:Arzela-ascoli}.
\end{proof}
We now state the following extra estimate that can be deduced from the proofs of Propositions~\ref{prop:CompactnessEtaWeight} and~\ref{prop:CompactnessEta}.
\begin{proposition}\label{prop:GradP}
    There holds
    \begin{align*}
        \lim_{h\to0}\; \limsup_{k\to\infty}\;\int_0^t\int_{\R^{2d}} \overline{\calK}_h^{x,y} n_{k}^x|p_{k}^{x}-p_{k}^{y}| \dx x\dx y\dx\tau = 0.
    \end{align*}
\end{proposition}
\begin{proof}
Recall the quantity 
\[
    Q_{n_{k,\eta}}(t) := \int_{\R^{2d}} {\calK_h^{x,y} \abs*{n_{k,\eta}^{x} - n_{k,\eta}^{y}}(\upsilon_{k,\eta}^{x}+\upsilon_{k,\eta}^{y})} \dx x \dx y,
\]
which satisfies
\begin{equation}\label{eq:GradP1}
    Q_{n_{k,\eta}}(t) = Q_{n_{k,\eta}}(0) + 
    \mathcal{A}^{(tot)}_1 + \mathcal{A}^{(tot)}_2 +
    \mathcal{A}^{(tot)}_3 + \mathcal{A}^{(tot)}_4, 
\end{equation}
where $\mathcal{A}^{(tot)}_\ell$ are defined in the proof of Proposition~\ref{prop:CompactnessEtaWeight}. From the estimates in that proof it follows that
\begin{equation*}
    \mathcal{A}^{(tot)}_2 + \mathcal{A}^{(tot)}_4 \leq C|\log h|^{1/2},
\end{equation*}
and 
\begin{align*}
    &\mathcal{A}^{(tot)}_1 + \int_0^t \int_{\R^{2d}} \calK_h^{x,y} n_{k,\eta}^{x}|p_{k,\eta}^{x}-p_{k,\eta}^{y}| \upsilon_{k,\eta}^{x} \dx x\dx y \dx \tau\\[3pt]
    &\leq \nu^{-1}\norm{n_{k,\eta}}_{L^2(0,T;L^2(\R^d))}\norm{\calK_h}_{L^1(\R^d)}^{1/2}\prt*{\int_0^t\iintO{\calK_h^{x,y}\abs*{W_{k,\eta}^{x}-W_{k,\eta}^{y}}^2}\dx \tau}^{1/2}\\[3pt]
    &\quad+\norm{n_{k,\eta}}_{L^2(0,T;L^2(\R^d))} \norm{\calK_h}_{L^1(\R^d)}^{1/2} \prt*{\int_0^t\iintO{\calK_h^{x,y}\abs*{\Delta V^{x}-\Delta V^{y}}^2}\dx \tau}^{1/2}.
\end{align*}
The remaining term contains the growth dynamics and is equal to
\begin{align*}
    \mathcal{A}^{(tot)}_3 = \sum_{i=1}^2  \int_0^t\int_{\R^{2d}} \calK_h^{x,y} a^{(i),(tot)}_3(x,y)(\upsilon_{k,\eta}^{x}+\upsilon_{k,\eta}^{y}) \dx x \dx y \dx \tau,
\end{align*}
where
\begin{align*}
    a^{(i),(tot)}_3(x,y) 
    &:= G\i (p_{k,\eta}^{x}) \prt*{n_{k,\eta}^{(i),x} - n_{k,\eta}^{(i),y}} \sigma -\abs*{G\i (p_{k,\eta}^{x})-G\i (p_{k,\eta}^{y})} n_{k,\eta}^{(i),y} \\[3pt]
    &\leq |G\i (p_{k,\eta}^{x})| \abs*{n_{k,\eta}^{(i),x}-n_{k,\eta}^{(i),y}}\\[3pt]
    &\leq \Gamma \prt*{1+(p_{k,\eta}^{x})^s}\abs*{n_{k,\eta}^{(i),x}-n_{k,\eta}^{(i),y}}.
\end{align*}
It follows that
\begin{align*}
    \mathcal{A}^{(tot)}_3 &\leq 2\sum_{i=1}^2\int_0^t\iintO{\calK_h^{x,y}\abs*{n_{k,\eta}^{(i),x}-n_{k,\eta}^{(i),y}}}\dx \tau\\[3pt]
    &\quad+\sum_{i=1}^2\prt*{\int_0^t\iintO{\calK_h^{x,y}
    (p_{k,\eta}^{x})^\frak{q}}\dx\tau}^{\frac{s}{\frak{q}}} \prt*{\int_0^t\iintO{\calK_h^{x,y}\abs*{n_{k,\eta}^{(i),x}-n_{k,\eta}^{(i),y}}^{\frac{\frak{q}}{\frak{q}-s}}}\dx\tau}^{\frac{\frak{q}-s}{\frak{q}}}\\[3pt]
    &\leq 2\sum_{i=1}^2\int_0^t\iintO{\calK_h^{x,y}\abs*{n_{k,\eta}^{(i),x}-n_{k,\eta}^{(i),y}}}\dx\tau\\[3pt]
    &\quad+2\sum_{i=1}^2\norm{\calK_h}_{L^1(\Rd)}^{s/\frak{q}}\norm{p_{k,\eta}}_{L^\frak{q}(0,T;L^\frak{q}(\Rd))}^s \prt*{\int_0^t\iintO{\calK_h^{x,y}\abs*{n_{k,\eta}^{(i),x}-n_{k,\eta}^{(i),y}}^{\frac{\frak{q}}{\frak{q}-s}}}\dx\tau}^{\frac{\frak{q}-s}{\frak{q}}}.
\end{align*}
Consequently, using the above estimates in~\eqref{eq:GradP1}, and dividing by $|\log h|$, we deduce
\begin{equation}
    \label{eq:GradP2}
    \begin{aligned}
    &\int_0^t\int_{\R^{2d}} \overline{\calK}_h^{x,y} n_{k,\eta}^{x}|p_{k,\eta}^{x}-p_{k,\eta}^{y}|\upsilon_{k,\eta}^{x} \dx x\dx y\dx \tau\\[3pt]
    & \leq C|\log h|^{-1/2} + C\prt*{\int_0^t\iintO{\overline{\calK}_h^{x,y}\abs*{W_{k,\eta}^{x}-W_{k,\eta}^{y}}^2}\dx\tau}^{1/2}\\[3pt]
    &\quad+ C\prt*{\int_0^t\iintO{\overline{\calK}_h^{x,y}\abs*{\Delta V^{x}-\Delta V^{y}}^2}\dx\tau}^{1/2}\\[3pt]
    &\quad+2\sum_{i=1}^2\int_0^t\iintO{\overline{\calK}_h^{x,y}\abs*{n_{k,\eta}^{(i),x}-n_{k,\eta}^{(i),y}}}\dx\tau\\[3pt]
    &\quad+C\prt*{\int_0^t\iintO{\overline{\calK}_h^{x,y}\abs*{n_{k,\eta}^{(i),x}-n_{k,\eta}^{(i),y}}^{\frac{\frak{q}}{\frak{q}-s}}}\dx\tau}^{\frac{\frak{q}-s}{\frak{q}}}\\[3pt]
    &\quad+\iintO{\overline{\calK}_h^{x,y}\abs*{n^{in,x}_{k,\eta} -n^{in, y}_{k,\eta}}},
\end{aligned}
\end{equation}
where all the constants are independent of $\eta$ and $k$. Moreover, we observe that the exponent $\frak{q} / (\frak{q}-s) \in (1,2)$. 

Performing the same domain decomposition as in the proof of Proposition~\ref{prop:CompactnessEta}, we infer the estimate
\begin{align*}
    &\int_0^t\int_{\R^{2d}} \overline{\calK}_h^{x,y} n_{k,\eta}^{x}|p_{k,\eta}^{x}-p_{k,\eta}^{y}| \dx x\dx y\dx \tau \\
    &\leq \frac{1}{\zeta}\int_0^t\int_{\R^{2d}} \overline{\calK}_h^{x,y} n_{k,\eta}^{x}|p_{k,\eta}^{x}-p_{k,\eta}^{y}|\upsilon_{k,\eta}^{x} \dx x\dx y\dx \tau +  \frac{C}{|\log\zeta|},
\end{align*}
where for the first term on the right-hand side we can use the bound~\eqref{eq:GradP2}.
Passing to the limit $\eta\to 0$, we can remove the subscript $\eta$ in each term, replacing appropriate quantities with their limits indexed only by $k$. Then, using Proposition~\ref{prop:DensityCompactK}, we can take $\limsup$ as $k\to\infty$ to obtain
\begin{align*}
   \limsup_{k\to\infty}\;\int_0^t\int_{\R^{2d}} \overline{\calK}_h^{x,y} n_k^x |p_{k}^{x} - p_{k}^{y}|\dx x\dx y\dx \tau \leq \frac{\omega(h)}{\zeta} + \frac{C}{|\log\zeta|},
\end{align*}
where $\omega(h)\to 0$ as $h\to 0$. We now choose $\zeta=\omega(h)^{1/2}$ and take the limit $h\to0$ to conclude the proof.
\end{proof}

\begin{proposition}[Pressure compactness]\label{prop:CompactPressureFinal}
        The sequence $(p_k)_k$ is compact in $L^1(0,T;L^1(\Rd))$.
\end{proposition}
\begin{proof}
    The idea of the proof is to use Proposition~\ref{prop:GradP} to obtain compactness of a nonlinear function of the pressure, $p_k^{(k+1)/k} = n_k p_k$, and deduce compactness of the pressure as a consequence. 
        
\medskip
\underline{The compactness quantity.}
We have
\begin{align*}
    &\int_0^t\int_{\R^{2d}} \overline{\calK}_h^{x,y} |n_k^x p_{k}^{x}-n_k^y p_{k}^{y}|\dx x\dx y\dx \tau\\[3pt]
    &\leq
    \int_0^t\int_{\R^{2d}} \overline{\calK}_h^{x,y} n_k^x|p_{k}^{x}-p_{k}^{y}|\dx x\dx y\dx \tau + \int_0^t\int_{\R^{2d}} \overline{\calK}_h^{x,y} p_k^y |n_k^x - n_k^y|\dx x\dx y\dx \tau.
\end{align*}
For the last term we observe the following bound:
\begin{align*}
    &\int_0^t\int_{\R^{2d}} \overline{\calK}_h^{x,y} p_k^{y}|n_{k}^x-n_{k}^y|\dx x\dx y\dx \tau\\
    &\leq \norm{p_k}_{L^\infty(0,T;L^\frak{q}(\Rd))}
    \prt*{\int_0^t\int_{\R^{2d}} \overline{\calK}_h^{x,y}|n_{k}^x-n_{k}^y|^{\frak{q}/(\frak{q}-1)}\dx x\dx y\dx \tau}^{(\frak{q}-1)/\frak{q}}.
\end{align*}
Since $p_k$ is bounded uniformly in $L^\infty(0,T; L^\frak{q}(\Rd))$ and $n_k$ is compact in $L^q(\Rd \times (0,T))$, for any $q\in[1,2)$, we deduce that
\begin{align*}
    \lim_{h\to0}\;\limsup_{k\to\infty}\;\int_0^t\int_{\R^{2d}} \overline{\calK}_h^{x,y} |n_k^xp_{k}^{x}-n_k^yp_{k}^{y}|\dx x\dx y\dx \tau = 0.
\end{align*}

\medskip
\underline{The time derivative.} The product $n_kp_k$ satisfies
\begin{align*}
    \fpartial t(n_kp_k) &= \nabla \cdot(n_kp_k \nabla(W_k+V)) + \frac{k}{\nu}n_kp_kQ_k + p_k \prt*{n_k\1 G\1(p_k)+n\2 G\2(p_k)}\\
    &= \nabla \cdot(n_kp_k \nabla(W_k+V)) + \frac{k+1}{\nu}n_kp_kQ_k - n_kp_k(\Delta V + \Delta W_k).
\end{align*}
Let $\phi \in H^\frak{s}(\Rd), \frak{s} > d/2+1$. Then, using Lemma~\ref{lem:npQinL1},
\begin{align*}
    \abs*{\int_\Rd (k+1)n_kp_kQ_k \phi \dx x} \leq \norm{\phi}_{ L^\infty(\Rd)}\norm{(k+1)p_kn_kQ_k}_{L^1(\Rd)}
    \leq C_0(t)\norm{\phi}_{H^\frak{s}(\Rd)},
\end{align*}
where $C_0 \in L^1(0,T)$.
Next, we consider the term
\begin{align*}
    \abs*{\int_\Rd n_k p_k \Delta W_k \phi \dx x} &\leq \frac{1}{\nu}\norm{\phi}_{L^\infty(\Rd)}(\norm{n_k p_k W_k}_{L^1(\Rd)} + \norm{n_k p_k^2}_{L^1(\Rd)})\\
    &\leq C\norm{\phi}_{H^\frak{s}(\Rd)} (C+\norm{n_k}_{L^\frak{p}(\Rd)}\norm{p_k}_{L^{2\frak{p}'}(\Rd)}),
\end{align*}
and note that $2\frak{p}' = \min{(4, 2d/(d-2))}\leq \frak{q}$.
Similarly, for the transport term we write
\begin{align*}
    \abs*{\int_\Rd  n_kp_k\nabla(W_k+V)\cdot \nabla\phi \dx x}
    &\leq \norm{\nabla\phi}_{L^\infty(\Rd)}\norm{n_k p_k \nabla(W_k+V)}_{L^1(\Rd)}\\[3pt]
    &\leq C\norm{\phi}_{H^\frak{s}(\Rd)}(\norm{n_kp_k\nabla W_k}_{L^1(\Rd)} + \norm{n_kp_k\nabla V}_{L^1(\Rd)})\\[3pt]
    &\leq C\norm{\phi}_{H^\frak{s}(\Rd)}(\norm{n_k}_{L^{\frak{p}}(\Rd)}\norm{p_k}_{L^{2\frak{p}'}(\Rd)}\norm{\nabla W_k}_{L^{2\frak{p}'}(\Rd)} \\[3pt]
    &\qquad\qquad\qquad\qquad\qquad + \norm{n_kp_k}_{L^1(\Rd)}\norm{\nabla V}_{L^\infty(\Rd)})\\[3pt]
    &\leq C\norm{\varphi}_{H^\frak{s}(\Rd)}.
\end{align*}

Combining the above estimates, we obtain a uniform bound for the time derivative $\partial_t(n_kp_k)$ in $L^1(0,T;H^{-\frak{s}}(\Rd))$. 

\medskip
\underline{Strong convergence of pressure.}
It follows from Lemma~\ref{lem:CompactnessCriterion} that the sequence $n_kp_k = p_k^{(k+1)/k}$ is compact in $L^1(0,T;L_{\mathrm{loc}}^1(\Rd))$.

On the other hand, this product converges weakly in $L^1$ to $n_\infty p_\infty$. Indeed, this follows from Proposition~\ref{prop:DensityCompactK} and the uniform bound $n_k\in L^\infty(0,T;L^\frak{p}(\Rd))$ with $\frak{p}\geq 2$ in conjunction with the uniform bound $p_k\in L^\infty(0,T;L^\frak{q}(\R^d))$ with $\frak{q}>2$.
Then, using the elementary inequality
\begin{equation*}
    p\abs*{p^{1/k}-1} \leq \frac{1}{k}(1+p^2),\qquad p\geq 0,\, k\geq 1,
\end{equation*}
we deduce
\begin{equation*}
    \norm{n_kp_k - p_k}_{L^1(0,T;L^1(K))} \leq \frac{1}{k}(|K| + \sup_{k}\,\norm{p_k}_{L^2(0,T;L^2(\Rd))}^2)\to 0,
\end{equation*}
for every compact set $K\subset\R^d$. This implies that $n_\infty p_\infty = p_\infty$ a.e.\ and that $p_k\to p_\infty$ strongly in $L^1(0,T;L^1_{\mathrm{loc}}(\Rd))$.
Finally, arguing as earlier in the proof of Proposition~\ref{prop:CompactW_eta}, we can deduce global convergence by transferring tightness of the sequence $n_k$ to tightness of the sequence $p_k$.
\end{proof}

\subsection{Proof of Theorem~\ref{thm:StiffLimit}}
Propositions~\ref{prop:DensityCompactK} and~\ref{prop:CompactPressureFinal}, together with the a priori bounds, imply that, after extracting subsequences, the following convergences hold:
\begin{alignat*}{2}
    W_k &\to W_\infty, \quad &&\text{in $L^2(0,T; L^2(\Rd))$ and a.e.},\\
    n_k &\to n_\infty, \quad &&\text{in $L^{\infty-}(0,T;L^{\frak{p}-}(\Rd))$ and a.e.},\\
    p_k &\to p_\infty,\quad &&\text{in $L^{\infty-} (0,T;L^{\frak q-}(\Rd))$ and a.e.},
\end{alignat*}
where $W_\infty = K_\nu\star p_\infty$.
Moreover, there is a function $\frak h \in L^1(\Rd)\cap L^{s+1}(\Rd)$ such that $p_k \leq \frak h$ (since $s+1 <\frak q$).
Passing to the limit in the Brinkman equation we readily deduce $-\nu \Delta W_\infty + W_\infty = p_\infty$ a.e.\ in $\Rd \times (0,T)$.
Passing to the limit in the equations for the densities $n_k\i$ is also straightforward; for instance, the growth term can be treated in the same way as in the proof of Proposition~\ref{prop:ExistenceEtaTo0}. 
The relation $p_\infty(1-n_\infty) = 0$ a.e.\ was already proved in Proposition~\ref{prop:CompactPressureFinal}.

Finally, we prove the complementarity condition. From the uniform bound
\begin{equation*}
    \norm{k p_k Q_k}_{L^1(0,T; L^1(\Rd))} \leq C,
\end{equation*}
we deduce that $p_k Q_k \to 0$ in $L^1(0,T;L^1(\Rd))$.
On the other hand, we argue that
\begin{equation*}
    p_k Q_k = p_k(W_k-p_k) + \nu p_k\Delta V + \nu r_k p_k G\1(p_k) + \nu(1-r_k)p_k G\2(p_k),
\end{equation*}
converges in the sense of distributions to
\begin{equation*}
    p_\infty(W_\infty-p_\infty) + \nu p_\infty\Delta V + \nu r_\infty p_\infty G\1(p_\infty) + \nu (1-r_\infty)p_\infty G\2(p_\infty),
\end{equation*} 
where $r_\infty := \frac{n\1_\infty}{n_\infty}\mathbbm{1}_{\{n_\infty>0\}}$. The only nontrivial terms to consider are the reaction terms. 
From the a.e. convergence of the $n_k^{(i)}$ we clearly have $r_k \to r_\infty$ a.e. Then, using the trivial bound $0 \leq r_k \leq 1$, we deduce that this convergence holds in any $L_{\mathrm{loc}}^q(\Rd \times (0,T))$ with $ 1\leq q <\infty$.
We now claim that 
\begin{equation}\label{eq:StrongReactionFinal}
    p_k G\1(p_k)\to p_\infty G\1(p_\infty), \quad \text{strongly in $L^1(0,T;L^1(\Rd))$}.
\end{equation}
Indeed, convergence a.e.\ follows from the convergence of $p_k$, and we have
\begin{equation*}
    |p_k G\1(p_k)| \leq \Gamma(p_k + p_k^{s+1}) \leq \Gamma(\frak h + \frak h^{s+1}).
\end{equation*}
Since $\frak h + \frak h^{s+1}$ is integrable, we deduce~\eqref{eq:StrongReactionFinal} by dominated convergence.
Thus, we have shown that $r_kp_kG\1(p_k)\to r_\infty p_\infty G\1(p_\infty)$ strongly $L^1(0,T;L_{\mathrm{loc}}^1(\Rd))$.

We therefore deduce that 
\begin{equation*}
    p_\infty(W_\infty-p_\infty) + \nu p_\infty\Delta V + \nu r_\infty p_\infty G\1(p_\infty) + \nu (1-r_\infty)p_\infty G\2(p_\infty) = 0
\end{equation*}
or, equivalently, using $n_\infty p_\infty = p_\infty$,
\begin{equation*}
    p_\infty\left((W_\infty-p_\infty) + \nu \Delta V + \nu n^{(1)}_\infty G\1(p_\infty) + \nu n^{(2)}_\infty G\2(p_\infty)\right) = 0,
\end{equation*}
a.e.\ in $\Rd\times(0,T)$, as claimed in~\eqref{eq:ComplementarityRelation}. The proof of Theorem~\ref{thm:StiffLimit} is complete. \qed

\subsection{Smoothing effect of the stiff limit}
We conclude by proving the following regularising effect for the limit pressure and the limit density.
\begin{theorem}
    \label{thm:regularising-effect}
    The limit total density $n_\infty$ and the limit pressure $p_\infty$ satisfy
    $$
        n_\infty, p_\infty \in L^\infty(0,T;L^\infty(\Rd)),
    $$ 
    with $0\leq p_\infty \leq p_L$ and $0 \leq n_\infty \leq 1$.
\end{theorem}
\begin{proof}
    Let $\vartheta>1$. Pointwise almost everywhere in time there holds
    \begin{align*}
        \int_\Rd |p_\infty|^\vartheta \dx x
        &= \int_{\set{p_\infty < p_L}} p_\infty^\vartheta \dx x + \int_{\set{p_\infty \geq p_L}} p_\infty^\vartheta \dx x \\[3pt]
        &\leq |p_L|^{\vartheta-1}\norm{p_\infty}_{L^\infty(0,T;L^1(\Rd))}\\[3pt]
        &\quad+ \int_{\set{p_\infty \geq p_L}} \bigg(p_\infty - \nu \Delta V - \nu r_\infty G^{(1)}(p_\infty) - \nu (1 - r_\infty)G^{(2)}(p_\infty)\bigg)^\vartheta \dx x\\[3pt]
        &\leq C + \int_{\set{p_\infty \geq p_L}} |W_\infty|^\vartheta \dx x\\[3pt]
        &\leq C + \norm{W_\infty}_{L^\vartheta(\Rd)}^\vartheta,
    \end{align*}    
    having used the complementarity relation. Using Young's inequality for convolutions, we find
    \begin{align*}
        \norm{W_\infty}_{L^\vartheta(\Rd)}\leq \norm{K_\nu}_{L^{\frac{d}{d-2} - }(\Rd)} \norm{p_\infty}_{L^{q+}(\Rd)},
    \end{align*}
    with $q = \vartheta d/(2\vartheta + d) < \vartheta$, and a simple modification in $d=1,2$. Thus, we can bootstrap integrability until we reach $p_\infty \in L^\infty(0,T; L^{d/2+}(\Rd))$, whence we conclude $W_\infty \in L^\infty(0,T; L^\infty(\Rd))$, which implies that also $p_\infty\in L^\infty(0,T;L^\infty(\Rd))$. Indeed, from the above estimate, we get
    \begin{align*}
        \norm{p_\infty}_{L^\vartheta(\Rd)} \leq |p_L|^{(\vartheta-1)/\vartheta} \norm{p_\infty}^{1/\vartheta}_{L^\infty(0,T;L^1(\Rd))} + \norm{W_\infty}_{L^\infty(0,T;L^\infty(\Rd))}^{(\vartheta-1)/\vartheta} \norm{W_\infty}_{L^\infty(0,T;L^1(\Rd))}^{1/\vartheta},
    \end{align*}
    which, upon passing to $\vartheta \to \infty$, yields the result.

    \medskip
    Finally, to show that $p_\infty \leq p_L$, we argue by contradiction. Indeed, let us assume that $\norm{p_\infty}_{L^\infty(0,T;L^\infty(\Rd))} > p_L + \epsilon$ for some $\epsilon>0$. For convenience, we introduce the set $\frak O_\epsilon = \set{p_\infty > p_L + \epsilon}$. On $\frak O_\epsilon$ we have $p_\infty < W_\infty$ from the complementarity relation. Therefore,
    \begin{align*}
        \norm{p_\infty}_{L^\infty(0,T;L^\infty(\Rd))} = \sup_{\frak O_\epsilon} p_\infty \leq \sup_{\frak O_\epsilon} W_\infty \leq \norm{W_\infty}_{L^\infty(0,T;L^\infty(\Rd))} \leq \norm{p_\infty}_{L^\infty(0,T;L^\infty(\Rd))},
    \end{align*}
    whence we infer
    \begin{align}
        \label{eq:same-norms}
         \sup_{\frak O_\epsilon} p_\infty = \sup_{\frak O_\epsilon} W_\infty.
    \end{align}
    On the set $\frak O_\epsilon$ we have from the complementarity relation
    \begin{align*}
        W_\infty &= p_\infty - \nu\Delta V - \nu r_\infty  G^{(1)}(p_\infty) - \nu(1 - r_\infty) G^{(2)}(p_\infty)\\[3pt]
        &\geq p_\infty  -\nu \norm{\Delta V}_\infty - \nu r_\infty  G^{(1)}(p_L + \epsilon) - \nu (1 - r_\infty) G^{(2)}(p_L + \epsilon)\\[3pt]
        &=p_\infty   -\nu\norm{\Delta V}_\infty - \nu r_\infty  G^{(1)}(p_L) - \nu (1 - r_\infty) G^{(2)}(p_L)\\[3pt]
        &\qquad\qquad-  \nu r_\infty  \partial_p G^{(1)}(\tilde p_L)\epsilon - \nu (1- r_\infty)  \partial_p G^{(2)}(\tilde p_L')\epsilon\\[3pt]
        &\geq  p_\infty  -   \nu r_\infty  \partial_p G^{(1)}(\tilde p_L)\epsilon -  \nu(1- r_\infty)  \partial_p G^{(2)}(\tilde p_L')\epsilon\\[3pt]
        &\geq p_\infty  + \alpha \epsilon \nu,
    \end{align*}
    having used the definition of $p_L$, cf. Eq. \eqref{eq:def-pL}, and a Taylor expansion in the third line. Passing to the supremum over $\frak O_\epsilon$, we get
      \begin{align*}
        \norm{W_\infty}_{L^\infty(\frak O_\epsilon )} 
        &\geq  \norm{p_\infty}_{L^\infty(\frak O_\epsilon )} + \alpha \epsilon \nu >   \norm{p_\infty}_{L^\infty(\frak O_\epsilon )},
    \end{align*}
    which is absurd since $\norm{W_\infty}_{L^\infty(\frak O_\epsilon )}  = \norm{p_\infty}_{L^\infty(\frak O_\epsilon )} $, by Eq. \eqref{eq:same-norms}. Therefore, there cannot exist any $\epsilon>0$ such that $\norm{p_\infty}_{L^\infty(0,T;L^\infty(\Rd))} > p_L + \epsilon$.
    
    \medskip
    Finally, the bound on the limit total density follows from the elementary inequality $n^k - k(n-1)_+ \geq 0$, for any $n\geq 0$. Substituting $n=n_k$ and integrating over space, we get
    $$
        0 \leq \int_\Rd (n_k-1)_+ \dx x \leq \frac1k \int_\Rd n_k^k \dx x \leq \frac{1}{k}\norm{p_k}_{L^\infty(0,T; L^1(\Rd))} \to 0,
    $$
    as $k\to \infty$. From the pointwise convergence $n_k\to n_\infty$, we infer $0 \leq (n_\infty-1)_+ \leq 0$, i.e.,
    $$
        0 \leq n_\infty \leq 1,
    $$
    which concludes the proof. 
\end{proof}

\section*{Acknowledgements}
This project was partly funded by the Federal Ministry of Education and Research (BMBF) and the Free State of Saxony as part of the Excellence Strategy of the German Federal and State Governments.

T.D.\ acknowledges the support of the National Science Centre, Poland, project no. 2023/51/D/ST1/02316. 

M.S.\ warmly acknowledges funding from the Thematic Research Programme "Mathematical Biology - Analysis and Applications" within the Excellence Initiative - Research University at the University of Warsaw.

\appendix
\section{Compactness criterion}
\label{app:compcrit}

For $h\in(0,1)$ we define the family of functions $\mathcal{L}_h$ by
\begin{equation}
    \label{eq:pre-kernel-L}
    \mathcal{L}_h(z) = \frac{\phi(z)}{(|z|^2+h^2)^{d/2}},
\end{equation}
where $\phi \in C_c^\infty(\Rd)$ is a nonnegative, radially decreasing function with $\mathrm{supp}(\phi) \subset B_2(0)$ and $\phi \equiv 1$ on $B_1(0)$.
Now we define the kernel $\calK_h$ by
\begin{equation}
    \label{eq:kernel-K}
    \calK_h(z) := \int_h^1 \frac{\mathcal{L}_s(z)}{\norm{\mathcal{L}_s}_{L^1(\Rd)}}\frac{\dx s}{s}.
\end{equation}
Notice that 
\begin{equation*}
    \norm{\calK_h}_{L^1(\Rd)} = |\log h|.
\end{equation*}
We introduce the notation
\begin{equation*}
    \overline{\calK}_h := \frac{\calK_h}{\norm{\calK_h}_{L^1(\Rd)}}.
\end{equation*}

We observe the following two properties of the kernel $\mathcal{K}_h$.
\begin{lemma}[Approximation of the identity]
    \label{lem:Kh_dirac}
    The family of kernels, $(\overline {\mathcal K}_h)_h$, defined in Eq. \eqref{eq:kernel-K} is a Dirac sequence.
\end{lemma}
\begin{proof}
    In order to prove the claim, we shall employ ~\cite[Definition 1.2.15, Theorem 1.2.21]{GrafakosClassical}. Since the kernel is normalised, we only have to show that
    \begin{align*}
        \int_{B_R(0)^c} \overline \calK_h(z) \dx z \to 0,
    \end{align*}
    for all $R>0$, as $h\to 0$. We only consider $R\in(0,2)$ as otherwise the integrand is identically zero.
    We can write
    \begin{align}\label{eq:KDirac}
        \int_{B_R(0)^c}\overline \calK_h(z) \dx z &= \frac{1}{|\log h|} \int_h^1 \frac{1}{s\norm{\mathcal L_s}_{L^1(\Rd)}} \left[ \int_{B_R(0)^c} \mathcal L_s(z) \dx z\right]\dx s,
    \end{align}
    where $\mathcal L_s$ is as in Eq. \eqref{eq:pre-kernel-L}, and satisfies 
    \begin{equation}\label{eq:LNorm}
        \norm{\mathcal L_s}_{L^1(\Rd)} \geq c_d (|\log s|+1).
    \end{equation}
    Furthermore, 
    \begin{align*}
		\int_{B_R(0)^c} \mathcal L_s(z) \dx z 
		&\leq \int_{B_R(0)^c\cap B_2(0)} \frac1{(|z|^2 + s^2)^{d/2}} \dx z\\[3pt]
		&= \omega_d\int_R^2 \frac{r^{d-1}}{(r^2+s^2)^{d/2}}\dx r\\[3pt]
		&\leq C(R).
	\end{align*}
    Now, let us fix some $h_0\in(h,1)$ and use the above observations in Eq.~\eqref{eq:KDirac} to estimate
    \begin{align*}
        \int_{B_R(0)^c}\overline \calK_h(z) \dx z &\leq \frac{C(R)}{|\log h|}\prt*{\int_h^{h_0}\frac{1}{s|\log s|} \dx s + \int_{h_0}^1 \frac{1}{s} \dx s}\\[3pt]
        &\leq \frac{C(R,h_0)}{|\log h|} \log|\log h|,
    \end{align*}
    which converges to $0$ as $h\to 0$ for every $R$.
\end{proof}

\begin{lemma}\label{lem:prop-gradK}
There exists a constant $C>0$ such that
    \begin{equation}
\label{eq:prop-gradK}
    |z||\nabla \mathcal{K}_h(z)| \leq \left\{
    \begin{array}{cc}
        d\mathcal{K}_h(z), & z\in\{1\leq|z|\leq2\}^c,\\[0.5em]
        d\mathcal{K}_h(z) + C(|\log|\log h|| + 1), & z\in\{1\leq|z|\leq2\}.
    \end{array}
    \right.
\end{equation}
\end{lemma}
\begin{proof}
First, by direct calculation, we observe that
\begin{equation*}
\label{eq:L_grad}
    |z||\nabla \mathcal{L}_h(z)| \leq \left\{
    \begin{array}{cc}
        d\mathcal{L}_h(z), & z\in\{1\leq|z|\leq2\}^c \\[0.5em]
        d\mathcal{L}_h(z) + \norm{\phi'}_{L^\infty}\frac{1}{(|z|^2+h^2)^{d/2}}, & z\in\{1\leq|z|\leq2\}.
    \end{array}
    \right.
\end{equation*}
Since
\begin{equation*}
    |z||\nabla \mathcal{K}_h(z)| \leq \int_h^1 \frac{|z||\nabla \mathcal{L}_s(z)|}{\norm{\mathcal{L}_s}_{L^1(\Rd)}} \frac{\dx s}{s},
\end{equation*}
the result follows easily --- for the error term in the annulus $\{1\leq |z|\leq 2\}$ we choose $h_0\in(h,1)$ and write
\begin{equation*}
    \int_h^1 \frac{1}{\norm{\mathcal{L}_s}_{L^1(\Rd)}(|z|^2+s^2)^{d/2}} \frac{\dx s}{s}
    \leq \int_h^1 \frac{1}{s\norm{\mathcal{L}_s}_{L^1(\Rd)}} \dx s
    \leq C\int_h^{h_0} \frac{1}{s|\log s|}\dx s + C\int_{h_0}^1\frac{1}{s}\dx s,
\end{equation*}
having used bound~\eqref{eq:LNorm} again.
\end{proof}

For brevity of notation, throughout we shall use the superscript notation $\overline \calK_h^{x,y} := \overline \calK_h(x-y)$.
We can now recall the following compactness criterion which we use significantly in Section~\ref{sec:UnboundedData} and Section~\ref{sec:StiffLimit}. It was first proved in~\cite{BelgacemJabin}.

\begin{lemma}
    \label{lem:CompactnessCriterion}
    Let $(u_n)_n$ be a sequence of functions uniformly bounded in $L^p(\Rd \times (0,T))$ with $1\leq p < \infty$. If $(\partial_t u_n)_n$ is uniformly bounded in $L^r(0,T;W^{-1,r}(\Rd))$, $r\geq 1$, and if
    \begin{equation*}
    \lim_{h\to0}\;\limsup_{n\to\infty}\int_0^T\int_{\R^{2d}} \overline{\calK}_h(x-y) |u_n(x,t)-u_n(y,t)|^p\dx x\dx y\dx t =0,
    \end{equation*}
    then $(u_n)$ is compact in $L^p_{\mathrm{loc}}(\Rd \times (0,T))$. If $(u_n)$ is compact in $L^p(\Rd \times (0,T))$, then the above limit holds.
\end{lemma}
Below we prove an extension of the result of \cite{BelgacemJabin}, in the style of the Aubin-Lions lemma with compact injection into Banach-space valued functions that are continuous in time, which will be useful later.

\begin{lemma}
    \label{lem:Arzela-ascoli}
    Let $(u_n)_n$ be a sequence of functions uniformly bounded in $L^\infty(0,T;L^p(\Rd))$ with $1\leq p < \infty$, such that  
    \begin{itemize}
        \item \textbf{time regularity:} $\displaystyle \limsup_{n\to \infty} \ \norm{\partial_t u_n}_{L^r(0,T;W^{-1,r}(\Rd))} <\infty$, and $r>1$,
    \end{itemize}
    and 
    \begin{itemize}
        \item \textbf{space regularity:} $\displaystyle\limsup_{n \to \infty} \sup_{0 \leq t \leq T} \int_{\R^{2d}} \overline{\calK}_h(x-y) |u_n(x,t ) - u_n(y,t)|^p \dx x \dx y \to 0,$
    \end{itemize}
    as $h\to 0$. 
    
    Then, $(u_n)_n$ is precompact in $L^\infty(0,T;L_{\mathrm{loc}}^p(\Rd))$, and any limit, $u$, is continuous in time.
\end{lemma} 

\begin{proof}
    We have to show that for any compact set $\frak S \subset \Rd$ and any $\epsilon>0$, there exists $N_\eps \in \N$ such that 
    \begin{align*}
        \norm{u_n - u_m}_{L^\infty(0,T;L^p(\frak S))} < \epsilon,
    \end{align*}
    for all $m, n \geq N_\epsilon$. 
    To this end, let us fix an arbitrary $\epsilon>0$ and a compact set $\frak S \subset \Rd$. 

    \smallskip
    \textbf{Step 1 -- Almost time continuity.}\\
    We show that the weak time regularity in conjunction with the nonlocal control of oscillations implies almost-continuity in norm in the sense that $\norm{u_n(t) - u_n(s)}_{L^p(\frak S)}$ is small, whenever $n\gg1$ and $|t-s|\ll 1$. Indeed, we observe that 
    \begin{align}
        \label{eq:cty-time-translates}
        \begin{split}
        &\norm{u_n(t)-u_n(s)}_{L^p(\frak S)} \\
        &\leq \norm{\overline{\calK}_h\star u_n(t)-u_n(t)}_{L^p(\frak S)} + \norm{\overline{\calK}_h\star u_n(s)-u_n(s)}_{L^p(\frak S)}\\
        &\quad + \norm{\overline{\calK}_h\star u_n(t)-\overline{\calK}_h\star u_n(s)}_{L^p(\frak S)},
        \end{split}
    \end{align}
    for any $n \in \mathbb N$ and $0 \leq s < t \leq T$, by triangulation. Next, exploiting the spatial regularity, let us agree that there exist some $h_\epsilon>0$ and $N_\epsilon \in \N$ such that 
    \begin{align*}
        \sup_{h \leq h_\epsilon}  \sup_{n \geq N_\epsilon} \sup_{0 \leq t \leq T} \int_{\R^{2d}} \overline{\calK}_h^{x,y} |u_n(x) - u_n(y)|^p \dx x \dx y< (\epsilon / 12)^p,
    \end{align*}
    such that the first two terms in Eq. \eqref{eq:cty-time-translates}, can be estimated by
    \begin{align*}
        \norm{\overline{\calK}_h\star u_n(t)-u_n(t)}_{L^p(\frak S)} &\leq \prt*{\int_{\R^{2d}} \overline{\calK}_h^{x,y}|u_n(x,t)-u_n(y,t)|^p\dx x\dx y}^{1/p}< \epsilon/12,
    \end{align*}
    having used Jensen's inequality and the fact that $\overline {\mathcal K}_h$ integrates to unity. In summary, for $n\geq N_\epsilon$ and $h\leq h_\epsilon$ we have 
    \begin{align*}
        &\norm{u_n(t)-u_n(s)}_{L^p(\Rd)} \leq  \frac\epsilon6 + \norm{\overline{\calK}_h\star u_n(t)-\overline{\calK}_h\star u_n(s)}_{L^p(\Rd)}.
    \end{align*}
    Concerning the remaining term, we have
    \begin{align*}
        &\norm{\overline{\calK}_h\star u_n(t)-\overline{\calK}_h\star u_n(s)}_{L^p(\frak S)} \\
        &\quad \leq \int_s^t\prt*{\int_{\frak S}|\overline{\calK}_h\star \partial_\tau u_n(\tau)|^p \dx x}^{1/p} \dx \tau\\
        &\quad =: \int_s^t I(\tau) \dx \tau,
    \end{align*}
    where
    \begin{align}
        \label{eq:estimating-I}
        I(\tau) = \norm*{\overline{\calK}_h\star \frac{\partial u_n}{\partial \tau} (\tau)}_{L^p(\frak S)}.
    \end{align}
    Recall that $\partial_\tau u_n(\tau) \in W^{-1,r}(\Rd)$ implies the existence of functions $\frak u_\ell \in L^{r}(\Rd)$, $\ell = 0, 1, \ldots, d$, such that
    \begin{align}
        \label{eq:rep-neg-sob}
        \frac{\partial u_n}{\partial \tau} (\tau) = \frak u_0 + \sum_{\ell = 1}^d \frac{\partial \frak u_\ell}{\partial x_\ell},
    \end{align}
    with 
    \begin{align}
        \label{eq:neg-norm-identity}
        \norm*{\frac{\partial u_n}{\partial \tau}}_{W^{-1,r}(\Rd)} = \max_{0 \leq \ell \leq d} \norm{\frak u_\ell}_{L^{r}(\Rd)},
    \end{align}
    see, for instance, \cite[Prop. 9.20]{Brezis}.  Thus, substituting Eq. \eqref{eq:rep-neg-sob} into Eq. \eqref{eq:estimating-I}, we get
    \begin{align}
        \begin{split}
        I(\tau) 
        &\leq \norm*{\overline \calK_h \star \frak u_0}_{L^p(\frak S)}  +\sum_{\ell = 1}^d\norm*{\frac{\partial \overline \calK_h}{\partial x_\ell} \star \frak u_\ell}_{L^p(\frak S)}\\
        &\leq |\frak S|^{1/p} \norm{\overline \calK_h}_{L^{r'}(\Rd)} \norm{\frak u_0}_{L^{r}(\Rd)} +  d |\frak S|^{1/p} \norm{\nabla \overline{\calK}_h}_{L^{r'}(\Rd)} \max_{1\leq \ell\leq d} \norm{\frak u_\ell}_{L^{r}(\Rd)}\\
        &\leq \iota(|\frak S|, h) \norm{\partial_\tau u_n}_{W^{-1,r}(\Rd)},
        \end{split}
    \end{align}
    with $1/r + 1/{r'} = 1$, and having used Eq. \eqref{eq:neg-norm-identity}. We note that  $\iota(|\frak S|, h) \to \infty$, as $h\to 0$ or as $|\frak S| \to \infty$.

    We conclude
    \begin{align}
        \label{eq:almost-equi-continuity}
        \norm{\overline{\calK}_h\star u_n(t)-\overline{\calK}_h\star u_n(s)}_{L^p(\frak S)} &\leq \iota(|\frak S|, h) \norm{\partial_\tau u_n}_{L^r(0, T; W^{-1,r}(\Rd))} |t-s|^{1/{r'}}.
    \end{align}
    By the time regularity assumption there exists $\tilde N_\eps\in \N$ such that for $n> \tilde N_\epsilon$ 
    \begin{align*}
        \norm{\overline{\calK}_{h}\star u_n(t)-\overline{\calK}_h\star u_n(s)}_{L^p(\frak S)} &\leq C(h, |\frak S|, \epsilon) |t-s|^{1/{r'}}.
    \end{align*}
    Thus, for $n > N_\epsilon' := \max(N_\eps, \tilde N_\eps)$ and $h<h_\eps$, it follows that 
    \begin{align*}
        \norm{u_n(t) - u_n(s)}_{L^p(\frak S)} &\leq  \frac\epsilon6 + C(h, |\frak S|, \eps) |t-s|^{1/{r'}}.
    \end{align*}
    Then, we can choose $\delta_\epsilon  > 0 $ such that 
    \begin{align}
        \label{eq:small-time-translations}
        \norm{u_n(t) - u_n(s)}_{L^p(\frak S)} &\leq  \frac\epsilon3,
    \end{align}
    whenever $|t-s|<\delta_\epsilon $, which concludes the first step.

    \medskip
    \textbf{Step 2 -- Cauchy property in $L^\infty(0,T;L^p(\frak S))$.}\\
    Next, let us choose a countable dense subset of $[0,T]$ denoted by $(t_i)_{i\in \N}$. By the oscillation control, we have that $\set{u_n(t_i)}_{n}$ is relatively compact in $L^p(\frak S)$ for any given $i\in \N$. Applying a diagonal argument, we can assume that the same sequence is relatively compact for each $i\in \N$, i.e., let us agree that
    \begin{align*}
        \lim_{n\to \infty} u_n(t_i) = u(t_i),
    \end{align*}
    for each $i\in \N$. We select a finite subset of these time instances, denoted by $\set{\hat t_1, \ldots, \hat t_M} \subset \set{t_i}_{i\in \N}$, such that 
    \begin{align}
        \min_{i=1,\ldots, M} |t - \hat t_i| \leq \delta_\epsilon,
    \end{align}
    for all $t\in[0,T]$, where $\delta_\epsilon$ was chosen in Eq. \eqref{eq:small-time-translations}. However, then for $n,m\geq N_\epsilon'$, we have
    \begin{align*}
        &\norm{u_n(t) - u_m(t)}_{L^p(\frak S)} \\
        &\quad \leq \norm{u_n(t) - u_n(\hat t_i)}_{L^p(\frak S)}  + \norm{u_m(\hat t_i) - u_m(t)}_{L^p(\frak S)}\\
        &\qquad + \norm{u_n(\hat t_i) - u_m(\hat t_i)}_{L^p(\frak S)}\\
        &\quad \leq \frac{2\epsilon}3 + \norm{u_n(\hat t_i) - u_m(\hat t_i)}_{L^p(\frak S)},
    \end{align*}
    where the first two terms are small by Eq. \eqref{eq:small-time-translations}. Finally, upon possibly increasing $N_\epsilon'$ further, and exploiting the compactness of the sequence $\set{u_n(\hat t_i)}_n$ we have for any $n, m\geq N_\epsilon'$, that 
    \begin{align*}
        &\norm{u_n(t) - u_m(t)}_{L^p(\frak S)} < \epsilon,
    \end{align*}
   for almost every $t\in [0,T]$,  which shows that $(u_n|_{\frak S})_n$ is relatively compact in $L^\infty(0,T;L^p(\frak S))$. \\
    
    \textbf{Step 3 -- Compactness in $L^\infty(0,T; L_\mathrm{loc}^p(\Rd))$.}\\
    To conclude, we observe that $\frak S \subset \Rd$ was arbitrary and, if $(\frak S_j)_{j\in \N}$ is a compact exhaustion of $\Rd$, we can extract a diagonal subsequence such that the same subsequence converges on all $\frak S_j$.

    \medskip
    \textbf{Step 4 -- Time-continuity of the limit.}\\
    Finally, without relabelling, let us assume that $u_n$ is a subsequence which converges to a function $u$ in $L^\infty(0,T;L^p_{\mathrm{loc}}(\Rd))$. 
    To show that the limit is continuous in time, we consider a time-continuous auxiliary sequence and apply the Arzela-Ascoli theorem to infer the time continuity of its limit. We subsequently identify the limits.
    To this end, we first extend the limit above, $u$, by a mirror extension
    \begin{align*}
        \bar u(t) := \left\{
        \begin{array}{ll}
            u(-t), & t \in (-T, 0),\\
            u(t), & t \in [0,T],\\
            u(2T-t), & t\in (T, 2T),
        \end{array}
        \right.
    \end{align*}
    and consider the family of functions
    $$
        v^\alpha := \phi_\alpha \star_t u,
    $$
    obtained from $u$ by mollifying in time, where $0<\alpha<T$. By definition, the restriction of $v^\alpha$ to $[0,T]$ satisfies $v^\alpha\in C([0,T]; L_{\mathrm{loc}}^p(\Rd))$, and 
    \begin{align*}
        \max_{0\leq t\leq T}\; \norm{v^\alpha}_{L^p(\frak S)} \leq C,
    \end{align*}
    i.e., the family $(v^\alpha)_{\alpha>0}$ is uniformly bounded, and it remains to show that the sequence is also equi-continuous. Before we address this property, we need to establish a control of spatial oscillations uniformly in the parameter $\alpha$.
    Indeed, we observe that
    \begin{align*}
        &\int_{\R^{2d}} \overline{\calK}_h^{x,y}|v^\alpha(x,t) - v^{\alpha}(y,t)|^p\dx x \dx y\\
        &= \int_{\R^{2d}} \overline{\calK}_h^{x,y}\abs*{\int_\R \phi_\alpha(t-\tau) [u(x,\tau) - u(y,\tau)]\dx \tau}^p\dx x \dx y\\
        &\leq \int_{\R^{2d}} \overline{\calK}_h^{x,y}\int_\R \phi_\alpha(t-\tau) \abs*{u(x,\tau) - u(y,\tau)}^p\dx \tau \dx x \dx y,
    \end{align*}
    by Jensen's inequality applied to the time integral. Thus, we have 
    \begin{align*}
        & \int_{\R^{2d}} \overline{\calK}_h^{x,y}|v^\alpha(x,t), - v^{\alpha}(y,t)|^p\dx x \dx y\\
        &= \int_\R \phi_\alpha(t-\tau) \brk*{\int_{\R^{2d}} \overline{\calK}_h^{x,y} \abs*{u(x,\tau) - u(y,\tau)}^p \dx x \dx y} \dx \tau\\
        &\leq \sup_{0\leq \tau \leq T} \int_{\R^{2d}} \overline{\calK}_h^{x,y} \abs*{u(x,\tau) - u(y,\tau)}^p \dx x \dx y\\
        &\leq \limsup_{n\to \infty} \sup_{0\leq \tau \leq T} \int_{\R^{2d}} \overline{\calK}_h^{x,y} \abs*{u_n(x,\tau) - u_n(y,\tau)}^p \dx x \dx y,
    \end{align*}
    along the subsequence as above, by lower semi-continuity. Thus, the spatial regularity of $v^\alpha$ is controlled uniformly in $\alpha$, i.e.,
    \begin{align}
        \label{eq:sup_alpha_oscill}
        \sup_{\alpha > 0} \sup_{0 \leq t \leq T}\int_{\R^{2d}} \overline{\calK}_h^{x,y}|v^\alpha(x,t) - v^{\alpha}(y,t)|^p\dx x \dx y \to 0,
    \end{align}
    as $h \to 0$. 
    We have now garnered all information necessary to establish 
    equi-continuity of the mollified family. To this end, let $\epsilon>0$, and repeat Eq. \eqref{eq:cty-time-translates} for $v^\alpha$ to get
    \begin{align*}
        &\norm{v^\alpha(t) - v^\alpha(s)}_{L^p(\frak S)} \\
        &\leq \norm{\overline \calK_h \star_x v^\alpha(t) - v^\alpha(t)}_{L^p(\frak S)} + \norm{\overline \calK_h \star_x v^\alpha(s) - v^\alpha(s)}_{L^p(\frak S)} \\
        &\qquad + \norm{\overline \calK_h \star_x v^\alpha(t) - \overline \calK_h \star_x v^\alpha(s)}_{L^p(\frak S)}.
    \end{align*}
    We use Eq. \eqref{eq:sup_alpha_oscill}, to observe that 
    \begin{align*}
        \sup_{0\leq t \leq T} \norm{\overline \calK_h \star_x v^\alpha(t) - v^\alpha(t)}_{L^p(\frak S)} < \epsilon/3,
    \end{align*}
    for some $h>0$ sufficiently small. 
    Next, we repeat Eq. \eqref{eq:almost-equi-continuity}, to get 
    \begin{align*}
        \norm{\overline{\calK}_h\star v^\alpha(t)-\overline{\calK}_h\star v^\alpha(s)}_{L^p(\frak S)} &\leq \iota(|\frak S|, h) \norm{\partial_\tau u}_{L^r(0, T; W^{-1,r}(\Rd))} |t-s|^{1/{r'}} < \epsilon/3,
    \end{align*}
    for $|t-s|\leq \delta_\epsilon$, where we used the fact that $\norm{\partial_\tau v^\alpha}_{L^r(0,T; W^{-1,r}(\Rd))} \leq C \norm{\partial_\tau u}_{L^r(0,T; W^{-1,r}(\Rd))}$. Combining all pieces of the jigsaw, for given $\epsilon>0$ we can always find $\delta_\epsilon$ such that
    \begin{align*}
        \norm{v^\alpha(t)-v^\alpha(s)}_{L^p(\frak S)} < \epsilon,
    \end{align*}
    as soon as $|t-s|<\delta_\eps$. Applying the Arzel\'{a}-Ascoli theorem yields a continuous limit $v\in C([0,T];L^p(\frak S))$, which agrees with $u$ almost everywhere. Thus $u$ has a continuous representative.
\end{proof}
\begin{remark}
    \label{rem:global-time-continuity}
    Let us remark that strong convergence in $L^\infty(0,T;L^p(\Rd))$ can be attained if tightness of the sequence $(u_n)_n$ is assumed additionally.
    Furthermore, if the regularity assumptions in the statement of Lemma \ref{lem:Arzela-ascoli} hold for all $n\in \N$ (replacing $\limsup_n$ by $\sup_n$), each member of the sequence is time-continuous and the convergence holds in $C([0,T]; L^p_{\mathrm{loc}}(\Rd))$.
\end{remark}

\section{Existence of bounded solutions}
\label{sec:existence}
This section is dedicated to establishing the existence of weak solutions to System \eqref{eq:main-system} in the following sense.
\begin{definition}[Bounded weak solutions]
    \label{def:weak-sol-Brinkman}
    Given nonnegative initial data, 
    $$
        (n^{(1), \mathrm{in}}, n^{(2), \mathrm{in}}) \in (L^1(\Rd)\cap L^\infty(\Rd))^2,
    $$ we call a pair of nonnegative functions $(n^{(1)}, n^{(2)})$ with
    \begin{align}
        n^{(1)}, n^{(2)} \in L^\infty(0,T;L^1(\Rd)\cap L^\infty(\Rd)) \cap C([0,T];L^q(\Rd)),
    \end{align}
    for any $1\leq q < \infty$, and 
    \begin{align}
        \partialt {n^{(1)}}, \partialt {n^{(2)}} \in L^2(0,T; H^{-1}(\R^d)),
    \end{align}
    a weak solution to System \eqref{eq:main-system} if there holds
    \begin{align}
        \int_0^T\int_\Rd \varphi \partialt{n}^{(i)}  + n^{(i)} \nabla  \varphi \cdot (\nabla V + \nabla W)  \dx x \dx t = \int_0^T \int_\Rd  \varphi n^{(i)} G^{(i)}(p) \dx x \dx t,
    \end{align}
    for $i=1,2$, and any test function $\varphi \in L^2(0,T ; H^1(\Rd))$, and $n\i(0) = n^{(i),\mathrm{in}}$, where $p = (n^{(1)} + n^{(2)})^k$.
\end{definition}
Then, the rest of this appendix is dedicated to proving the existence of bounded weak solutions.
\begin{theorem}
    \label{thm:bounded-solutions}
    Let $n^{(1),in}, n^{(2),in} \in L^1(\Rd)\cap L^\infty(\Rd)$ be nonnegative. Then, there exist functions
    \begin{align*}
        n\1, n\2 \in L^\infty(0,T;L^1(\Rd)\cap L^\infty(\Rd)) \cap C([0,T];L^q(\Rd)) \cap H^1(0,T;H^{-1}(\Rd)),
    \end{align*} 
    for any $1\leq q < \infty$, such that $(n\1, n\2)$ is a weak solution of the system
    \begin{align*}
        \partialt {n\i}  &= \nabla \cdot(n\i\nabla W) + \nabla \cdot(n\i\nabla V) + n\i G\i(p),\\
        -\nu \Delta W + W &= p\\
        n\i(\cdot, 0) &= n^{(i),in},
    \end{align*}
    in the sense of Definition~\ref{def:weak-sol-Brinkman}.
\end{theorem}

The proof of Theorem \ref{thm:bounded-solutions} is based on several regularisation procedures. In Section \ref{sec:fixedpoint}, we begin by proving existence of solutions for the linearised system
\begin{align*}
	\partialt {n^{(i)}} - \epsilon \Delta n^{(i)} = \nabla \cdot (m^{(i)} \nabla U) +\Phi_\kappa(m^{(i)})  G^{(i)}\left((\Phi_\kappa(m^{(1)} + m^{(2)}))^k\right),\quad i=1,2,
\end{align*}
where $\epsilon>0$ is a small artificial viscosity, $U$ is an arbitrary, yet fixed, external potential, $\Phi_\kappa$ is a cut-off at $\kappa>0$, and $m^{(1)}, m^{(2)}$ are given functions. In particular, this requires us to bound $G \circ \Phi_k$, and, in allusion to Eq. \eqref{eq:bound-G}, we define
\begin{align}
    \label{eq:bound-G-kappa}
    C_{G,\kappa} = \max_{i=1,2} \max_{0 \leq s \leq \kappa^k } |G^{(i)}(s)|.
\end{align}
Next, we show that there is a fixed point allowing us to replace $m^{(i)}$ by $n^{(i)}$, for $i=1,2$. Subsequently, in Section \ref{sec:brinkman}, we replace $U$ by a given external potential, $V$, and the potential $W^\delta$, coming from a mollified Brinkman's law, cf. Eq. \eqref{eq:brinkman}, and subsequently, in Section \ref{sec:remove-mollification-brinkman}, we remove the mollification, i.e., we let $\delta \to 0$. Next, in Section \ref{sec:Linfty-propag} we prove a propagation of $L^\infty$-control which allows us to remove the cut-offs. In the final step, Section \ref{sec:removal-of-eps}, we let the artificial viscosity, $\epsilon>0$, vanish --- thus proving Theorem~\ref{thm:bounded-solutions} for bounded solutions.

\subsection{Existence of solutions to a linear, regularised system}
\label{sec:fixedpoint}
Let $n^{(i), \mathrm{in}} \in H^1(\Rd)$ be some initial data. We begin by introducing the truncation 
\begin{align}
	\Phi_\kappa(s) :=
	\left\{
	\begin{array}{ll}
		0, & s \leq 0,\\
		s, & 0 \leq s \leq \kappa,\\
		\kappa, & \kappa \leq s,
	\end{array}
	\right.
\end{align}
where $\kappa >0$ is fixed and will be chosen later. As the strategy is based on a fixed-point argument, let us introduce the space $\mathcal X:= C([0, T]; H^1(\Rd))$.
Then, for
$$
    (m^{(1)}, m^{(2)}) \in \mathcal X \times \mathcal X,
$$
we consider the system of linear parabolic equations
\begin{align}
    \label{eq:linear-regular-approx-sys}
	\partialt {n^{(i)}} - \epsilon \Delta n^{(i)} = \nabla \cdot (m^{(i)} \nabla U) + R^{(i)}(m^{(1)}, m^{(2)}),
\end{align}
where
\begin{align*}
    R^{(i)}(m^{(1)}, m^{(2)}) := \Phi_\kappa(m^{(i)})  G^{(i)} \left( (\Phi_\kappa(m^{(1)} + m^{(2)}))^k\right),
\end{align*}
and $U$ is a given function satisfying
\begin{align}
    \label{eq:bounds-on-U}
    \norm{\nabla U}_{L^\infty(0,T;L^\infty(\Rd))}, \norm{\Delta U}_{L^\infty(0,T;L^\infty(\Rd))} \leq C_U,
\end{align}
for some $C_U>0$. Upon setting 
$$
    \Phi^{(i)} := \Phi_\kappa(m^{(i)}), \  i=1,2,\quad \text{and} \quad  \underline \Phi := \Phi_\kappa(m^{(1)} + m^{(2)}),
$$ 
we observe that
    \begin{align*}
        &\norm{\nabla \cdot (m^{(i)} \nabla U) + R^{(i)}(m^{(1)}, m^{(2)})}_{L^\infty(0,T;L^2(\Rd))} \\[0.5em]
	   &\quad = \norm{\nabla m^{(i)} \cdot \nabla U + m^{(i)} \Delta U + \Phi^{(i)} G^{(i)}(\underline \Phi^k)}_{L^\infty(0,T;L^2(\Rd))}\\[0.5em]
	   &\quad \leq \norm{m^{(i)}}_{\mathcal X} \norm{\nabla U}_{L^\infty(0,T;L^\infty(\Rd))} + \norm{m^{(i)}}_{\mathcal X} \norm{\Delta U}_{L^\infty(0,T;L^\infty(\Rd))} + C_{G, \kappa} \norm{m^{(i)}}_{\mathcal X}\\[0.5em]
        &\quad \leq C(C_U,C_{G,\kappa}) \norm{m^{(i)}}_{\mathcal X},
    \end{align*}
    for $i=1,2$, where $C_{G, \kappa}$ is given in Eq. \eqref{eq:bound-G}. Therefore, as the right-hand side is bounded in $L^\infty(0,T;L^2(\Rd))$, linear semigroup theory gives the existence of a unique solution $(n^{(1)}, n^{(2)})$ to System \eqref{eq:linear-regular-approx-sys} with
    $$
        n^{(1)}, n^{(2)} \in L^\infty(0,T;H^1(\Rd)) \cap L^2(0,T;H^2(\Rd)) \cap H^1(0,T;L^2(\Rd)) \subset \mathcal X,
    $$
    and $n^{(i)}(0) = n^{(i),\mathrm{in}}$, see~\cite[Theorem 1.3.11]{zheng1995nonlinear}.

\subsection{Existence of solutions to a regularised nonlinear system}
Based on this existence result, we are now able to prove existence of solutions to the nonlinear, yet still regularised, system. 
\begin{proposition}[Existence of solutions]
\label{prop:existence-reg-nonlin-system}
    For given initial data $n^{(1), \mathrm{in}}, n^{(2),\mathrm{in}} \in  H^1(\Rd)$, there exist functions 
    $$
        n^{(1)}, n^{(2)} \in L^\infty(0,T;H^1(\Rd)) \cap L^2(0,T; H^2(\Rd)) \cap H^1(0,T;L^2(\Rd)),
    $$
    solving 
    \begin{align*}
	   \partialt {n^{(i)}} - \epsilon \Delta n^{(i)} = \nabla \cdot (n^{(i)} \nabla U) + \Phi_\kappa(n^{(i)})  G^{(i)}\left((\Phi_\kappa(n^{(1)} + n^{(2)}))^k\right),
    \end{align*}
    pointwise a.e., and $n^{(i)}(0) = n^{(i), \mathrm{in}}$. Moreover, if $n^{(i), \mathrm{in}} \geq 0$ for $i=1,2$, then $n^{(i)} \geq 0$, almost everywhere.
\end{proposition}
\begin{proof}
    We define the operator
    \begin{align*}
        \Psi: \mathcal X \times \mathcal X &\to \mathcal X \times \mathcal X,\\
        m = (m^{(1)}, m^{(2)}) & \mapsto n := (n^{(1)}, n^{(2)}),
    \end{align*}
    where $n$ is the solution we constructed above. In order to show that $\Psi$ admits a fixed point, let $m, \tilde m \in \mathcal X \times \mathcal X$, and $n = \Psi(m), \tilde n = \Psi(\tilde m)$ be the associated solutions. We set 
    $$
        \hat n := n - \tilde n, \quad \hat m:= m - \tilde m, \quad \text{and} \quad \hat \Phi^{(i)} := \Phi(m^{(i)}) - \Phi(\tilde m^{(i)}),
    $$
    and observe that the differences satisfy
    \begin{align}
       \label{eq:diffs-lin-approx}
       \begin{split}
            \partialt {\hat n^{(i)}} - \epsilon \Delta \hat n^{(i)} 
            &= \nabla \hat m^{(i)} \cdot \nabla U + \hat m^{(i)} \Delta U + \hat \Phi^{(i)} G^{(i)}(\underline \Phi^k) + \tilde \Phi^{(i)} \left(G^{(i)}(\underline \Phi^k) - G^{(i)}(\tilde {\underline \Phi}^k)\right),
        \end{split}
    \end{align}
    equipped with zero initial data $\hat n^{\mathrm{in}} = 0$. 
    
    \underline{$L^2$-Contraction.} Multiplying Eq. \eqref{eq:diffs-lin-approx} by $\hat n^{(i)}$, for $i=1,2$, and integrating in space gives
    \begin{align}
        \label{eq:L2-contraction-drift}
    	\frac12 \ddt \norm{\hat n^{(i)}}_{L^2(\Rd)}^2 + \epsilon \norm{\nabla \hat n^{(i)}}_{L^2(\Rd)}^2 \leq \frac12 \norm{\hat n^{(i)}}_{L^2(\Rd)}^2 + N^{(i)},
    \end{align}
    where
    \begin{align*}
        N^{(i)} := \frac12 \norm{\nabla \hat m^{(i)} \cdot \nabla U + \hat m^{(i)} \Delta U + \hat \Phi^{(i)} G^{(i)}(\underline \Phi^k)  + \tilde \Phi^{(i)} (G^{(i)}(\underline \Phi^k) - G^{(i)}(\tilde {\underline\Phi}^k))}_{L^2(\Rd)}^2.
    \end{align*}
    This term can be estimated
    \begin{align}
        \label{eq:rhs-estimate}
        \begin{split}
    	N^{(i)} 
    	&\leq C(C_U, C_{G,\kappa}, \kappa, k) \left(\norm{\hat m^{(1)}}_{H^1(\Rd)}^2 + \norm{\hat m^{(2)}}_{H^1(\Rd)}^2 \right),
        \end{split}
    \end{align}
    Thus, using Eq. \eqref{eq:rhs-estimate} in Eq. \eqref{eq:L2-contraction-drift}, and adding up the resulting inequality for $i=1,2$, we obtain
    \begin{align}
    	\label{eq:ddt-L2-diff}
        \begin{split}
    	\ddt &\left(\norm{\hat n^{(1)}}_{L^2(\Rd)}^2 + \norm{\hat n^{(2)}}_{L^2(\Rd)}^2 \right) \\
        &\leq \norm{\hat n^{(1)}}_{L^2(\Rd)}^2 + \norm{\hat n^{(2)}}_{L^2(\Rd)}^2 + C \left(\norm{\hat m^{(1)}}_{H^1(\Rd)}^2  + \norm{\hat m^{(2)}}_{H^1(\Rd)}^2 \right).
        \end{split}
    \end{align}

    \underline{$\dot H^1$-Contraction.}
    Next, we multiply Eq. \eqref{eq:diffs-lin-approx} by $-\Delta \hat n^{(i)}$ to obtain
    \begin{align*}
    	\frac12 \ddt &\norm{\nabla \hat n^{(i)}}_{L^2(\Rd)}^2 + \epsilon \norm{\Delta \hat n^{(i)}}_{L^2(\Rd)}^2\\[0.5em]
    	&\leq \frac{\epsilon}{2}\norm{\Delta \hat n^{(i)}}_{L^2(\Rd)}^2 + C(C_U, C_{G, \kappa}, \kappa, k, \epsilon) \left(\norm{\hat m^{(1)}}_{H^1(\Rd)}^2 + \norm{\hat m^{(2)}}_{H^1(\Rd)}^2 \right).
    \end{align*}
    Hence, 
    \begin{align}
    	\label{eq:ddt-H1hom-diff}
    	\ddt \left(\norm{\nabla \hat n^{(1)}}_{L^2(\Rd)}^2  + \norm{\nabla \hat n^{(2)}}_{L^2(\Rd)}^2 \right) \leq C \left( \norm{\hat m^{(1)}}_{H^1(\Rd)}^2 + \norm{\hat m^{(2)}}_{H^1(\Rd)}^2 \right).
    \end{align}
    
    \underline{$\mathcal X^2$-Contraction.}
    Adding up Eq. \eqref{eq:ddt-H1hom-diff} and Eq. \eqref{eq:ddt-L2-diff}, we obtain
    \begin{align*}
    	\ddt \norm{\hat n}_{H^1(\Rd)^2}^2 \leq \norm{\hat n}_{H^1(\Rd)^2}^2 + C \norm{\hat m}_{H^1(\Rd)^2}^2.
    \end{align*}
    Applying Gronwall's inequality and using the fact that the initial condition is identically equal to zero, we get
    \begin{align}
        \label{eq:gronwall-banach-fp}
    	\sup_{0 \leq t \leq T}\norm{\hat n}_{H^1(\Rd)^2}^2 \leq C \norm{\hat m}_{L^\infty(0,T; H^1(\Rd))^2}^2 T e^T.
    \end{align}
    Thus, upon choosing $T^\ast>0$ sufficiently small, we get
    \begin{align*}
        \norm{\hat n}_{L^\infty(0,T^\ast; H^1(\Rd))^2} \leq \frac12 \norm{\hat m}_{L^\infty(0,T^\ast; H^1(\Rd))^2}, 
    \end{align*}
    which shows the desired contraction. Applying the Banach fixed-point theorem yields the local existence of a solution. As the constant in Eq. \eqref{eq:gronwall-banach-fp} only depends on $C_U, C_{G, \kappa}, \kappa, k,$ and $\epsilon$, but not on the solution itself, we may repeat this argument iteratively until we have solutions on $(0,T)$.

    Finally, if $n^{(i), \mathrm{in}} \geq 0$, for $i=1,2$, we multiply the equations 
    \begin{align*}
	   \partialt {n^{(i)}} - \epsilon \Delta n^{(i)} = \nabla \cdot (n^{(i)} \nabla U) + \Phi_\kappa(n^{(i)})  G^{(i)}\left((\Phi_\kappa(n^{(1)} + n^{(2)}))^k\right), \qquad i=1,2,
    \end{align*}
    by $-(n^{(i)})_-$, integrate, and use Gronwall's inequality to find that $n^{(i)} \geq 0$.
\end{proof}

\subsection{Existence for the system coupled through a regularised Brinkman law}
\label{sec:brinkman}
Next, we replace $U$ by a velocity potential consisting of an external part, $V$ satisfying Assumptions \ref{assu:existence}, and a coupling to a regularised Brinkman's law. Before we proceed, let us
note that the class of initial data can be relaxed to $L^2(\Rd)$ with associated solutions in 
$$
    n^{(i)} \in C([0,T];L^2(\Rd)) \cap L^2(0,T; H^1(\Rd)),\; \text{and } \ \partialt{n^{(i)}} \in L^2(0,T;H^{-1}(\Rd)), i=1,2,
$$
This follows from the fact that $H^1$ is dense in $L^2$ such that we can find an approximating sequence $n^{(i), \mathrm {in}}_\ell \to n^{(i), \mathrm {in}}$, strongly in $L^2(\Rd)$. The corresponding solutions $n^{(i)}_\ell$ converge strongly in $L^2(0,T;L^2(\Rd))$ by the above energy estimates in conjunction with the Aubin-Lions lemma.

\begin{proposition}
    \label{prop:existence-regularised-truncated-Brinkman}
    Given some initial data $n^{(1),\mathrm{in}}, n^{(2),\mathrm{in}} \in L^2(\Rd)$, there exist functions 
    $$
        n^{(i)} \in C([0,T];L^2(\Rd)) \cap L^2(0,T; H^1(\Rd)),\; \text{and } \ \partialt{n^{(i)}} \in L^2(0,T;H^{-1}(\Rd)), i=1,2,
    $$
    such that
     \begin{align}
        \label{eq:soln-reg-brinkman-system}
        \begin{split}
	   \int_0^T\int_\Rd &\phi \partialt {n^{(i)}} \dx x \dx t + \epsilon \int_0^T\int_\Rd \nabla \phi \cdot \nabla n^{(i)} \dx x \dx t  \\
       &= \int_0^T\int_\Rd \phi \left( \nabla \cdot (n^{(i)} \nabla V) + \nabla \cdot (n^{(i)} \nabla K_\nu \star \omega_\delta \star p_\kappa) + \Phi^{(i)} G^{(i)}(p_\kappa)\right)\dx x \dx t,
       \end{split}
    \end{align}
    for any $\phi \in L^2(0,T; H^1(\Rd))$, and $n^{(i)}(0) = n^{(i), \mathrm{in}}$, where $\Phi^{(i)} := \Phi_\kappa(n^{(i)})$, $i=1,2$, and $p_\kappa := (\Phi_\kappa(n^{(1)} + n^{(2)}))^k$. Moreover, if $n^{(i), \mathrm{in}}\geq 0$, for $i=1,2$,  then $n^{(i)} \geq 0$, almost everywhere.
\end{proposition}
\begin{proof}
    To prove this proposition we employ another fixed-point argument and begin by defining the space
    $$
        \Upsilon := \left\{ f \in C([0, T];L^2(\Rd)) \ |\ \norm{f}_{C([0, T];L^2(\Rd))} \leq \norm{f(0)}_{L^2(\Rd)} e^{C_\Upsilon T} \right\},
    $$
    where
    \begin{align*}
        C_{\Upsilon} := \norm{\Delta V}_{L^\infty(0,T;L^\infty(\Rd))} + 2\kappa^k/\nu + 2 C_{G, \kappa}.
    \end{align*}
    Next, for fixed $m\in \Upsilon \times \Upsilon$, let us set $U = V + W^\delta[m]$, where
    \begin{align*}
        W^\delta[m] := K_\nu \star \omega_\delta \star p_m, \qquad \text{with} \quad   p_m := \Phi_\kappa(m^{(1)} + m^{(2)})^k.
    \end{align*}
    We observe that 
    $$
        \norm{\Delta U}_{L^\infty(0,T;L^\infty(\Rd))} \leq \norm{\Delta V}_{L^\infty(0,T;L^\infty(\Rd))} + C(\kappa, k, \nu),
    $$
    as well as 
    $$
        \norm{\nabla U}_{L^\infty(0,T;L^\infty(\Rd))} \leq \norm{\nabla V}_{L^\infty(0,T;L^\infty(\Rd))} + C(\kappa, k, \nu),
    $$
    for some constant $C(\kappa, k, \nu)>0$ independent of $\delta$. Therefore, the previous existence result for fixed velocity potential, Proposition \ref{prop:existence-reg-nonlin-system} 
    , provides a solution $(n^{(1)}, n^{(2)})$ with 
    $$
        n^{(1)}, n^{(2)} \in L^\infty(0,T;L^2(\Rd)) \cap L^2(0,T; H^1(\Rd)) \cap H^1(0,T;H^{-1}(\Rd)) \subset C([0,T];L^2(\Rd)),
    $$
    satisfying 
    \begin{align*}
    	\partialt {n^{(i)}} - \epsilon \Delta n^{(i)} = \nabla \cdot (n^{(i)} \nabla V) + \nabla \cdot (n^{(i)} \nabla W^\delta[m])  + R^{(i)}(n^{(1)}, n^{(2)}),
    \end{align*}
    equipped with initial condition $(n^{(1), \mathrm{in}}, n^{(2), \mathrm{in}})$.
    A quick computation shows that
    \begin{align}
        \label{eq:L2-propagation-fp}
        \begin{split}
            \frac12 \ddt \int_{\Rd} |n^{(i)}|^2 \dx x + \epsilon \int_{\Rd} |\nabla n^{(i)}|^2 \dx x 
            &\leq \frac12 \int_{\Rd} |n^{(i)}|^2 (\Delta V + \Delta W^\delta[m] + 2G^{(i)}(0)) \dx x\\
            &\leq \frac{C_\Upsilon}2 \int |n^{(i)}|^2 \dx x,
        \end{split}
    \end{align}
    with 
    $$
        C_\Upsilon = \norm{\Delta V}_{L^\infty(0,T;L^\infty(\Rd))} + 2\kappa^k/\nu + 2 C_{G, \kappa},
    $$
    as above. An application of Gronwall's inequality shows that the fixed-point operator
    \begin{align*}
        F : \Upsilon \times \Upsilon &\to \Upsilon \times \Upsilon,\\
        m &\mapsto n,
    \end{align*}
    is well-defined in that it maps back into the original space with the correct bound.
    
    To show that it is a contraction, consider $m, \tilde m \in \Upsilon$, and let $n, \tilde n$ be the associated solutions. As before, we write an equation for the difference
    \begin{align*}
        \partialt {\hat n^{(i)}} - \epsilon \Delta \hat n^{(i)} 
        &= \nabla \cdot(\hat n^{(i)} \nabla V) + \nabla\cdot (\hat n^{(i)} \nabla W^\delta[m] + n^{(i)} \nabla K_\nu\star \omega_\delta\star (p_m  - p_{\tilde m})) \\
        &\qquad + \Phi^{(i)} G^{(i)}(p) - \tilde \Phi^{(i)} G^{(i)}(\tilde p).
    \end{align*}
    We multiply this equation by $\hat n^{(i)}$ and integrate to find
    \begin{align*}
        &\frac12 \ddt \norm{\hat n^{(i)}}_{L^2(\Rd)}^2 + \epsilon \norm{\nabla \hat n^{(i)}}_{L^2(\Rd)}^2\\[6pt]
        & \leq \frac12 C(C_V, \kappa, k)\left(\norm{\hat n^{(1)}}_{L^2(\Rd)}^2 + \norm{\hat n^{(2)}}_{L^2(\Rd)}^2 + \norm{\hat m^{(1)}}_{L^2(\Rd)}^2 + \norm{\hat m^{(2)}}_{L^2(\Rd)}^2\right) \\[6pt]
        &\quad + \frac1{2\epsilon}\norm{n^{(i)} \nabla K_\nu \star \omega_\delta \star (p_m -  p_{\tilde m})}_{L^2(\Rd)}^2  + \frac\epsilon2 \norm{\nabla \hat n^{(i)}}_{L^2(\Rd)}^2.
    \end{align*}
    Since
    \begin{align*}
        \frac1{2\epsilon}\norm{n^{(i)} &\nabla K_\nu \star \omega_\delta \star (p_m -  p_{\tilde m})}_{L^2(\Rd)}^2\\[6pt] 
        &\leq C(\epsilon, \nu, \delta, k, \kappa) \norm{n^{(i), \mathrm{in}}}_{L^2(\Rd)}^2\exp(2C_\Upsilon T)\left(\norm{\hat m^{(1)}}_{L^2(\Rd)}^2 + \norm{\hat m^{(2)}}_{L^2(\Rd)}^2\right),
    \end{align*}
    we have
    \begin{align*}
           \frac12 \ddt \norm{\hat n^{(i)}}_{L^2(\Rd)}^2 
           &\leq \frac12 C(C_U, p_L, \epsilon, k, \kappa, \nu, \delta) \norm{n^{(i),\mathrm{in}}}_{L^2(\Rd)^2}^2 \exp(2 C_\Upsilon T) \times\\
           &\quad \times\left(\norm{\hat n^{(1)}}_{L^2(\Rd)}^2 + \norm{\hat n^{(2)}}_{L^2(\Rd)}^2 + \norm{\hat m^{(1)}}_{L^2(\Rd)}^2 + \norm{\hat m^{(2)}}_{L^2(\Rd)}^2\right).
    \end{align*}
    Upon summing over $i=1,2$, we obtain, for some constant $C^\ast>0$,
    \begin{align*}
        \ddt \left(
        \norm{\hat n}_{L^2(\Rd)^2}^2 
        \right)
        \leq C^\ast \norm{n^{\mathrm{in}}}_{L^2(\Rd)^2}^2 \exp(2C_\Upsilon T) \left(\norm{\hat n}_{L^2(\Rd)^2}^2 + \norm{\hat m}_{L^2(\Rd)^2}^2 \right).
    \end{align*}
    An application of Gronwall's lemma yields
    \begin{align}
        \label{eq:gronwall-contraction-reg-Brinkman}
        \norm{\hat n}_{L^2(\Rd)^2}^2 \leq C^\ast \norm{n^{\mathrm{in}}}_{L^2(\Rd)^2}^2 e^{2C_\Upsilon T} T \, e^{C^\ast  \norm{n^{\mathrm{in}}}_{L^2(\Rd)^2}^2 e^{2C_\Upsilon T} T}  \norm{\hat m}_{\Upsilon^2}^2 .
    \end{align}
    In particular, by choosing $T_1 > 0$ such that
    $$
        C^\ast \norm{n^{\mathrm{in}}}_{L^2(\Rd)^2}^2 e^{2C_\Upsilon T_1} T_1 = \frac12,
    $$
    we observe that Eq. \eqref{eq:gronwall-contraction-reg-Brinkman} becomes
    \begin{align*}
        \norm{\hat n}_{C([0,T_1];L^2(\Rd)^2)}^2 \leq \frac{\sqrt{e}}{2} \norm{\hat m}_{C([0,T_1];L^2(\Rd)^2)}^2.
    \end{align*}
    This yields a contraction and an application of Banach's fixed-point theorem yields a fixed point on $[0,T_1]$. Repeating the procedure above with initial data $n(T_1)$ on the space $C([T_1,T_1 + T_2];L^2(\Rd))$ with analogue bound, yields 
     \begin{align*}
        \norm{\hat n(t)}_{L^2(\Rd)^2}^2 &\leq C^\ast \norm{n(T_1)}_{L^2(\Rd)^2}^2 e^{2C_\Upsilon T_2} T_2 e^{C^\ast  \norm{n(T_1)}_{L^2(\Rd)^2}^2 e^{2C_\Upsilon T_2} T_2} \norm{\hat m}_{\Upsilon^2}^2\\
        &\leq C^\ast \norm{n(0)}_{L^2(\Rd)^2}^2 e^{2C_\Upsilon T_1} e^{2C_\Upsilon T_2} e^{C^\ast  \norm{n(0)}_{L^2(\Rd)^2}^2 e^{2C_\Upsilon T_1}  e^{2C_\Upsilon T_2} T_2} T_2 \norm{\hat m}_{\Upsilon^2}^2,
    \end{align*}
    for $T_1 \leq t \leq T_1 +  T_2$. Upon choosing $T_2$  such that
    \begin{align*}
        T_2 := T_1 e^{-2 C_\Upsilon T_2},
    \end{align*}
    we get
    \begin{align*}
        \norm{\hat n}_{C([T_1, T_1 + T_2];L^2(\Rd)^2)}^2 &\leq  \frac{\sqrt e}{2} \norm{\hat m}_{C([T_1, T_1 + T_2]; L^2(\Rd))}^2,
    \end{align*}
    which by the Banach fixed point theorem yields a solution on $[T_1, T_1 + T_2]$. We iterate this procedure, and it remains to show that we can continue this procedure indefinitely to obtain a global solution on $[0,T]$.
    By construction
    $$
        T_{k} = T_{k-1} e^{-2C_\Upsilon T_{k}} = \cdots = T_1 \exp\left(-2C_{\Upsilon} \left[S_k - T_1\right] \right),
    $$
    where $S_k = \sum_{j=1}^k T_j$. Summing from $k=1$ to $K$, we get
    \begin{align*}
        S_K 
        &= T_1 e^{2C_\Upsilon T_1} \sum_{k=1}^K e^{-2 C_\Upsilon S_k} \\
        &\geq T_1 e^{2C_\Upsilon T_1} K e^{-2 C_\Upsilon S_K}.
    \end{align*}
    Arguing by contradiction, let us assume that the solution is not global, i.e., that there is a constant $C>0$ such that 
    $0 \leq S_K \leq C$ for all $K\in \N$ but then
    \begin{align*}
        C \geq S_K \geq T_1 e^{2C_\Upsilon T_1} K e^{-2 C_\Upsilon C} \to \infty,
    \end{align*}
    as $K\to \infty$ which is absurd. Thus, we can repeat the fixed-point argument until we have $S_K \geq T$, the time horizon in the statement of the problem. The nonnegativity follows from the construction via Proposition \ref{prop:existence-reg-nonlin-system}, if nonnegative initial data is assumed.
\end{proof}

\subsection{Removal of mollification in Brinkman's law}
\label{sec:remove-mollification-brinkman}
This section is dedicated to removing the mollification kernel in the drift term associated to the Brinkman equation. To this end, we use the following property of mollifications.

\begin{proposition}\label{prop:MollificationProperty}
    Let $f\in L^1(\Rd)$ and $(g_\delta)_\delta \subset L^\infty(\Rd)$ be uniformly bounded and such that $g_\delta\to0$ pointwise almost everywhere. Then, $f\star g_\delta\to 0$, in $L^2_{\mathrm{loc}}(\Rd)$.
\end{proposition}
\begin{proof}
\begin{align*}
    \int_S |f \star g_\delta|^2 \dx x 
    &= \int_S \abs*{\int_\Rd f(x-y) g_\delta (y) \dx y}^2 \dx x \\
    &\leq 2 \int_S \abs*{\int_{B_R} f(x-y) g_\delta(y) \dx y}^2 + \abs*{\int_{B_R^c} f(x-y) g_\delta (y) \dx y}^2 \dx x\\
    &\leq 2 \int_S \abs*{\int_{B_R} f(x-y) g_\delta(y) \dx y}^2 \dx x + 2 \norm{g_\delta}_{L^\infty}\abs*{\int_{B_R^c} f(x-y) \dx y}^2 \dx x.
\end{align*}
Since $f \in L^1(\Rd)$ and $S$ is bounded the integrals over the complements can be made smaller than $\epsilon/2$ for $R$ large enough such that
\begin{align}
    \int_S |f \star g_\delta|^2 \dx x 
    &\leq \int_S \abs*{\int_{B_R} f(x-y) g_\delta(y) \dx y}^2 \dx x + \frac\epsilon2.
\end{align}
Finally, using the pointwise convergence of $g_\delta \to 0$, the dominated convergence theorem yields the result.
\end{proof}

We are now ready to prove the removal of the mollification in the Brinkman's law.
\begin{proposition}[Removal of mollification]
    For each $\delta >0$, and any given initial data $n^{(1), \mathrm{in}}, n^{(2), \mathrm{in}} \in L^2(\Rd)$, let $(n_\delta^{(1)}, n_\delta^{(2)})$
    be a solution to Eq. \eqref{eq:soln-reg-brinkman-system}. Then, there exist functions
    $$
        n^{(1)}, n^{(2)} \in C([0,T];L^2(\Rd)) \cap L^2(0,T;H^1(\Rd)),
    $$ 
    with
    $$
        \partialt {n^{(1)}}, \partialt {n^{(2)}} \in L^2(0,T;H^{-1}(\Rd)),
    $$
    such that, up to a subsequence, $n_\delta^{(i)} \to n^{(i)}$
    strongly in $L^2(0,T; L_{\mathrm{loc}}^2(\Rd))$
    as $\delta \to 0$. The limit $(n^{(1)}, n^{(2)})$ solves
    \begin{align}
    \label{eq:unmollified-brinkman-with-cutoff}
    \begin{split}
        \int_0^T\int_\Rd &\phi \partialt {n^{(i)}} \dx x \dx t + \int_0^T\int_\Rd \epsilon \nabla \phi \cdot \nabla n^{(i)} \dx x \dx t \\
        &= \int_0^T\int_\Rd \phi \left(\nabla \cdot (n^{(i)} \nabla V) + \nabla \cdot (n^{(i)} \nabla W) + R^{(i)}(n^{(1)}, n^{(2)}) \right) \dx x \dx t,
    \end{split}
    \end{align}    
    for all $\phi \in L^2(0,T;H^1(\Rd))$, and $n^{(i)}(0) = n^{(i), \mathrm{in}}$, as well as
    $$
        - \nu \Delta W + W = (\Phi_\kappa(n^{(1)} + n^{(2)}))^k\quad\text{a.e.}
    $$
    Moreover, if $n^{(i),\mathrm{in}}\geq 0$, then $n\i\geq0$ a.e. in $(0,T)\times\Rd$.
\end{proposition}
\begin{proof}
    Since the estimate in Eq. \eqref{eq:L2-propagation-fp} is independent of $\delta$, we readily observe that
    \begin{align*}
    \norm{n_\delta^{(i)}}_{L^\infty(0,T;L^2(\Rd))} + \norm{ \nabla n_\delta^{(i)}}_{L^2(0,T; L^2(\Rd))} \leq C,
    \end{align*}
    as well as
    \begin{align*}
        \norm{\partial_t n_\delta^{(i)}}_{L^2(0,T;H^{-1}(\Rd))} \leq C,
    \end{align*}
    where $C>0$ is independent of $\delta$.
    With this regularity we can invoke the Aubin-Lions lemma to extract subsequences in $\delta$, $n_\delta^{(i)} \to n^{(i)}$,
    converging strongly in $L^2(0,T;L_{\mathrm{loc}}^2(\Rd))$, and pointwise a.e., for $i=1,2$. It remains to identify the limit. We set 
    $\underline {\Phi} := \Phi(n^{(1)} + n^{(2)})$ and $\underline {\Phi}_\delta := \Phi(n_\delta^{(1)} + n_\delta^{(2)})$ and, for arbitrary $\varphi \in C_c^\infty(\Rd \times (0,T))$, we compute
    \begin{align*}
        &\abs*{\int_0^T\int_\Rd\nabla \varphi \cdot \nabla K_\nu \star \omega_\delta \star \underline{\Phi}_\delta^k \dx x \dx t 
        -
        \int_0^T\int_\Rd\nabla \varphi \cdot \nabla K_\nu  \star \underline{\Phi}^k \dx x \dx t
        }\\[3pt]
        & = \abs*{\int_0^T\int_\Rd\nabla\varphi \star \omega_\delta \cdot \nabla K_\nu\star (\underline{\Phi}_\delta^k-\underline{\Phi}^k) \dx x \dx t} + \abs*{\int_0^T\int_\Rd\nabla\varphi \cdot \nabla K_\nu\star (\omega_\delta \star\underline{\Phi}^k-\underline{\Phi}^k) \dx x \dx t}.
    \end{align*}
    Using Proposition~\ref{prop:MollificationProperty}, we see that both terms vanish as $\delta\to 0$.
    Next, we observe
    \begin{align*}
        &\abs*{\int_0^T\int_\Rd \varphi \Phi(n_\delta^{(i)}) G^{(i)}(\underline{\Phi}_\delta^k)\dx x \dx t 
        -
        \int_0^T\int_\Rd \varphi \Phi(n^{(i)}) G^{(i)}(\underline{\Phi}^k) \dx x \dx t
        }\\[3pt]
        &\leq \norm{\varphi}_{L^\infty(0,T;L^\infty(\Rd))} 
        \Big(C_{G, \kappa} 
        \norm{\Phi(n_\delta^{(i)}) - \Phi(n^{(i)})}_{L^2(0,T;L^2(\mathrm{supp}\,\varphi))}\\
        &\hspace{4cm}
        + 
        \kappa \norm{G^{(i)}(\underline{\Phi}_\delta^k) - G^{(i)}(\underline{\Phi}^k)}_{L^2(0,T;L^2(\mathrm{supp}\,\varphi))}
        \Big)\\[3pt]
        &\leq C
        \left(
        \norm{n_\delta^{(1)} - n^{(1)}}_{L^2(0,T;L^2(\mathrm{supp}\,\varphi))}
        + 
        \norm{n_\delta^{(2)} - n^{(2)}}_{L^2(0,T;L^2(\mathrm{supp}\,\varphi))}
        \right)\\
        & \to 0,
    \end{align*}
    as $\delta \to 0$, by the strong $L_{\mathrm{loc}}^2$-convergence. The linear terms are straightforward. Finally, to recover the weak form it is standard to approximate $L^2(0,T;H^1(\Rd))$ test functions by $C_c^\infty$ functions.
\end{proof}

\subsection{Removal of the cut-offs}
\label{sec:Linfty-propag}
\begin{lemma}\label{lem:Linfty-propag}
    Let $n^{(1), \mathrm{in}}, n^{(2), \mathrm{in}} \in  L^1(\Rd) \cap L^\infty(\Rd)$ be nonnegative. Then, the associated solution to System \eqref{eq:unmollified-brinkman-with-cutoff} satisfies 
    $$
        \norm{n^{(1)} + n^{(2)}}_{L^\infty(0,T;L^\infty(\Rd))} \leq \max\left(n_L, \norm{n^{(1),\mathrm{in}} + n^{(2),\mathrm{in}}}_{L^\infty(\Rd)}\right).
    $$
    In particular, the individual species are also uniformly bounded in $L^\infty(0,T;L^\infty(\Rd))$.
\end{lemma}

\begin{proof}
    For this proof, let us set the cut-off parameter,  introduced in Eq. \eqref{eq:def-pL}, to be
    $$
        \kappa := \max(n_L, \norm{n^{(1),\mathrm{in}} + n^{(2),\mathrm{in}}}_{L^\infty(\Rd)}),
    $$
    where we recall that $n_L = p_L^{1/k}$ as in Assumption~\ref{assu:existence}.
    Let us write $\Phi\i = \Phi_{\kappa}(n\i)$, $\underline{\Phi} = \Phi_\kappa(p)$, and $n = n\1+n\2$. Then, we consider
    \begin{align*}
        \ddt \int_\Rd |\kappa - n|_- \dx x 
        &= - \int_\Rd \sign_-(\kappa - n) \partialt n \dx x\\
        &= - \int_\Rd \sign_-(\kappa - n) \big(\epsilon \Delta n + \nabla \cdot (n \nabla V) + \nabla \cdot (n \nabla W)\\
        &\hspace{6cm}+ \Phi^{(1)} G^{(1)} + \Phi^{(2)} G^{(2)}\big) \dx x \\
        &= I_\epsilon + I_W + I_V + I_G.
    \end{align*}
    We treat each of these terms individually. First, we observe that
    \begin{align*}
        I_\epsilon 
        &= \epsilon \int_\Rd \sign_-(\kappa - n) \Delta (\kappa - n) \dx x \leq \epsilon \int_\Rd \Delta (|\kappa - n|_-) \dx x \leq 0,
    \end{align*}
    by Kato's inequality. Next, we observe that
    \begin{align*}
        I_W 
        &= - \int_\Rd \sign_-(\kappa - n) (\nabla n \cdot \nabla W + n \Delta W) \dx x\\[3pt]
        &= \int_\Rd \nabla \cdot (|\kappa - n|_- \nabla W)\dx x - \int_\Rd \sign_-(\kappa - n) \kappa \Delta W\dx x\\[3pt]
        &= \int_\Rd |\kappa - n|_- \Delta W\dx x - \int_\Rd \sign_-(\kappa - n) n \Delta W \dx x \\[3pt]
        &\leq \norm{\Delta W}_{L^\infty(0,T;L^\infty(\Rd))} \int_\Rd |\kappa - n|_- \dx x - \frac1\nu \int_\Rd \sign_-(\kappa - n) n (W - \underline \Phi^k) \dx x\\[3pt]
        &\leq \norm{\Delta W}_{L^\infty(0,T;L^\infty(\Rd))} \int_\Rd |\kappa - n|_- \dx x - \frac1\nu \int_\Rd \sign_-(\kappa - n) n (\kappa^k - \underline \Phi^k) \dx x\\[3pt]
        &= C(\kappa, k, \nu) \int_\Rd |\kappa - n|_- \dx x,
    \end{align*}
    as the last term in the penultimate line is identically equal to zero. Next, we consider the term associated to the external potential, i.e.,
    \begin{align*}
        I_V 
        &= \int_\Rd \sign_-(\kappa - n) \nabla \cdot((\kappa - n)\nabla V)\dx x - \int_\Rd \sign_-(\kappa - n) \kappa \Delta V\dx x\\
        &\leq \int_\Rd - \sign_-(\kappa - n) \norm{\Delta V}_{L^\infty(0,T;L^\infty(\Rd))} \kappa \dx x.
    \end{align*}
    Finally, let us address the growth term. We see that
    \begin{align*}
        I_G 
        &= -\int_\Rd \sign_-(\kappa - n) (\Phi^{(1)} G^{(1)}(\underline \Phi^k) + \Phi^{(2)} G^{(2)}(\underline \Phi^k))\dx x\\
        &= \int_\Rd -\sign_-(\kappa - n) (\Phi^{(1)} G^{(1)}(\kappa^k) + \Phi^{(2)} G^{(2)}(\kappa^k))\dx x,
    \end{align*}
    since the integrand vanishes for $n < \kappa$ and $\underline \Phi = \kappa$, whenever $n > \kappa$. By definition of $p_L$, Eq. \eqref{eq:def-pL}, we have
    \begin{align*}
        I_G \leq \int_\Rd - \sign_-(\kappa - n) (-\norm{\Delta V}_{L^\infty(0,T;L^\infty(\Rd))}) (\Phi^{(1)} + \Phi^{(2)})\dx x.
    \end{align*}
    Thus, 
    \begin{align*}
        I_V + I_G 
        &\leq \int_\Rd - \sign_-(\kappa - n) \norm{\Delta V}_{L^\infty(0,T;L^\infty(\Rd))}(\kappa - \Phi(n^{(1)}) - \Phi(n^{(2)}))\dx x \leq 0,
    \end{align*}
    since $\kappa - \underline \Phi \leq 0$ on the set where $n\geq \kappa$. In conclusion, we have
    \begin{align*}
        \ddt \int_\Rd |\kappa - n|_- \dx x \leq C(\kappa,k,\nu) \int_\Rd |\kappa - n|_- \dx x,
    \end{align*}
    which, upon using Gronwall, yields
    \begin{align*}
        0 \leq \int_\Rd |\kappa - n|_-(t) \dx x \leq  \int_\Rd |\kappa - n|_-(0) \dx x\  e^{C(\kappa, k, \nu) t} = 0,
    \end{align*}
    since the initial data are bounded by $\kappa$. Since $0 \leq  n^{(1)} + n^{(2)} \leq \kappa$, we can remove the truncation, since then $\Phi_\kappa(n) = n$ and $\Phi_\kappa(n^{(i)}) = n^{(i)}$, for $i=1,2$.
\end{proof}
To summarise, we have proven the following existence result.
\begin{proposition}
    For each $\epsilon>0$ and any given nonnegative initial data $n^{(1), \mathrm{in}}, n^{(2), \mathrm{in}} \in L^1(\Rd) \cap L^\infty(\Rd)$, there exist nonnegative functions
    $$
        n_\epsilon^{(1)}, n_\epsilon^{(2)} \in C([0,T];L^2(\Rd)) \cap L^2(0,T;H^1(\Rd)),
    $$ 
    with
    $$
        \partialt {n_\epsilon^{(1)}}, \partialt {n_\epsilon^{(2)}} \in L^2(0,T;H^{-1}(\Rd)),
    $$
    and $0 \leq n_\epsilon^{(1)}, n_\epsilon^{(2)} \leq C$, for some constant $C$ independent of $\epsilon$, such that
        \begin{align}
        \label{eq:viscous-brinkman}
        \begin{split}
    	\int_0^T\int_\Rd &\varphi\partialt {n_\epsilon^{(i)}}\dx x \dx t + \epsilon \int_0^T\int_\Rd \nabla n_\epsilon^{(i)}\cdot \nabla \varphi \dx x\dx t\\
        &= \int_0^T\int_\Rd \varphi\brk*{\nabla \cdot (n_\epsilon^{(i)} \nabla V) + \nabla \cdot (n_\epsilon^{(i)} \nabla W_\epsilon) + n^{(i)}_\epsilon G^{(i)}(p_\epsilon)}\dx x \dx t,
        \end{split}
    \end{align}
    for any $\phi \in L^2(0,T;H^1(\Rd))$ and $n^{(i)}(0) = n^{(i), \mathrm{in}}$, as well as
    $$
        - \nu \Delta W_\epsilon + W_\epsilon = (n_\epsilon^{(1)} + n_\epsilon^{(2)})^k,
    $$
    almost everywhere in $\Rd\times(0,T)$.
\end{proposition}

\subsection{Removal of the parabolic term}
\label{sec:removal-of-eps}
To complete the general $L^1\cap L^\infty$-existence theory, we have to remove the artificial viscosity terms. To this end, we employ a nonlocal compactness criterion, first proved in~\cite{BelgacemJabin} and recalled in Appendix \ref{app:compcrit}. Moreover, we prove an extension in the style of the Aubin-Lions lemma where gradient control is replaced by mere oscillation control. This yields compactness in $L^\infty(0,T;L^p(\Rd))$ functions whose limits are continuous in time, allowing us to show time-continuity of our limit functions.

\begin{proposition}
\label{lem:RemoveEpsilon}
    Let $n^{(1),in}, n^{(2),in} \in L^1(\Rd)\cap L^\infty(\Rd)$ be nonnegative. For $\epsilon>0$ let $(n\1_\eps, n\2_\eps)$ be a global weak solution to System~\eqref{eq:viscous-brinkman} and set $p_\epsilon = n_\eps^k:=(n\1_\eps+n\2_\eps)^k$.
    Then, there exist functions
    \begin{align*}
        n\1, n\2 \in L^\infty(0,T;L^1(\Rd)\cap L^\infty(\Rd)) \cap C([0,T];L^q(\Rd)) \cap H^1(0,T;H^{-1}(\Rd)),
    \end{align*}
    for any $1 \leq q < \infty$, such that up to a subsequence
    \begin{equation*}
        n\i_\eps \to n\i,\quad \text{in\; $L^1(0,T;L^1(\Rd))$},\; i=1,2,
    \end{equation*}
    as well as 
    \begin{equation*}
        p_\eps \to p = (n\1 + n\2)^k,
    \end{equation*}
    almost everywhere,
    and such that $(n\1, n\2)$ is a weak solution of the system
    \begin{align*}
    \partialt {n\i}  &= \nabla \cdot(n\i\nabla W) + \nabla \cdot(n\i\nabla V) + n\i G\i(p),\\
    -\nu \Delta W + W &= p,\\
    n\i(\cdot, 0) &= n^{(i),in},
    \end{align*}
    in the sense of Definition~\ref{def:weak-sol-Brinkman}.
\end{proposition}
\begin{proof}
    Similarly as in the previous steps, we derive the energy estimate
    \begin{align*}
        \norm{n\i_\eps}_{L^\infty(0,T;L^2(\Rd))} + \norm{\sqrt{\eps}\nabla n\i_\eps}_{L^2(0,T;L^2(\Rd))} \leq C,
    \end{align*}
    for a constant $C$ independent of $\eps$, whence we deduce that $\partial_tn\i_\eps$ is uniformly bounded in $L^2(0,T;H^{-1}(\Rd))$.
    With these bounds, we can already extract weakly convergent subsequences and see that the viscosity term vanishes in the limit. We will now discuss the strong convergence of the velocity potential $W_\eps$. 
    
    First, since $p_\eps \in L^\infty(0,T; L^\infty(\Rd))$, we have $W_\eps \in L^\infty(0,T;W^{2,q}(\Rd))$ for any $q\in(1,\infty)$, and $\Delta W_\eps \in L^\infty(0,T;L^\infty(\Rd))$, uniformly in $\eps$.
    Adding up the density equations for $i=1$ and $i=2$, and multiplying the resulting identity by $k n_\eps^{k-1}$ we obtain
    \begin{align*}
        \partialt {p_\eps} = \eps\Delta p_\eps - \eps k(k-1)n_\eps^{k-2}|\nabla n_\eps|^2 &+ \nabla p_\eps\cdot(\nabla W_\eps + \nabla V) + k p_\eps(\Delta W_\eps + \Delta V)\\
        &+ k n_\eps\1 G\1(p_\eps)n_\eps^{k-1} + k n_\eps\2 G\2(p_\eps)n_\eps^{k-1}.
    \end{align*}

    We note that, upon integration in time and space and using the a priori bounds, the term $\eps n_\eps^{k-2}|\nabla n_\eps|^2$ is uniformly bounded in $L^1(0,T;L^1(\Rd))$. It follows that $\partial_t p_\eps$ belongs to $L^2(0,T;H^{-\frak{s}}(\Rd))$ for some $\frak{s}>d/2$. Consequently, since $W_\eps = K_\nu \star p_\eps$, we have the same uniform bounds for $\partial_t W_\eps$ and $\partial_t\nabla W_\eps$. 
    
    By the Aubin-Lions lemma, we can extract a subsequence such that $W_\eps$ and $\nabla W_\eps$ converge in $L^2(0,T;L^2_{\mathrm{loc}}(\Rd))$ to limit quantities $W = K_\nu \star p$ and $\nabla W = \nabla K_\nu \star p$, where $p$ is the weak-$*$ limit of $p_\eps$.
    With these convergences, we are ready to pass to the limit in all terms in Eq.~\eqref{eq:viscous-brinkman}, except for the reaction term. 
    
    For now, let us assume that the densities $n\i_\eps$ converge strongly in $L^1(0,T;L^1(\Rd))$, which we show last. Extracting a subsequence, we also have convergence almost everywhere. Then, by continuity of $G\i$ and of the map $n\mapsto n^k$, the term $n\i_\eps G\i(p_\eps)$ converges a.e.\ to $n\i G\i(p)$. Since this term is also uniformly bounded in ($L^1\cap L^\infty)(\Rd \times (0,T))$, we can finally pass to the limit in the weak formulation.

    \medskip 
    It remains to prove the global $L^1$-strong convergence $n\i_\eps \to n\i$. To this end, we use the compactness criterion in Lemma~\ref{lem:CompactnessCriterion}, following~\cite{BelgacemJabin}.  The strategy is similar to the one we employ in Section~\ref{sec:UnboundedData} (but slightly easier as here no auxiliary weight is required since the velocity has bounded divergence). The only point worthwhile highlighting is that the local compactness from \cite{BelgacemJabin} can be upgraded to global compactness using uniform tightness which is obtained by propagating tightness of the initial data similar to the proof of Proposition~\ref{prop:CompactW_eta}. In the scope of the removal of the parabolic term this procedure is standard, and we refer the reader to Eq.~\eqref{eq:tightness1} where we give the full details when removing the $L^\infty$-assumption.
    
    \smallskip
    Let us point out that, in particular, we can establish the following limit
    \begin{equation}
        \label{eq:comp-to-zero}
        \lim_{h\to0}\; \limsup_{\eps\to0} \int_{\R^{2d}} \overline{\calK}_h^{x,y}| n^{(i),x}_\eps - n^{(i),y}_\eps| \dx x\dx y =0,
    \end{equation}
    uniformly for every time $t\in [0,T]$. Consequently, we can apply Lemma~\ref{lem:Arzela-ascoli} and Remark \ref{rem:global-time-continuity} using the tightness of the sequence to deduce that the limit densities $n\i$ are continuous in time with values in $L^1(\Rd)$. An interpolation with the $L^\infty(0,T;L^\infty(\Rd))$ control yields $n\i \in C([0,T];L^q(\Rd))$ for any $q\in[1,\infty)$.
    \end{proof}

\section{Further details of proof of Lemma \ref{lem:npQinL1}}
\label{app:further-details}
We begin by estimating $\mathcal B_1$. Since 
    \begin{align*}
        n|Q| \leq n|W| + n|p| + 2\Gamma \nu n + 2 n\Gamma |p|^s,
    \end{align*}
    we have 
    \begin{align*}
        \mathcal B_1 &= \int_0^T\int_\Rd n|Q|\dx x \dx t \\
        &\leq \norm{n}_{L^2(0,T;L^2(\Rd))} (\norm{W}_{L^2(0,T;L^2(\Rd))} + \norm{p}_{L^2(0,T;L^2(\Rd))} +  2 \nu\Gamma \norm{p}^s_{L^{2s}(0,T;L^{2s}(\Rd))}) \\
        &\quad + 2\Gamma \nu \norm{n}_{L^1(0,T;L^1(\Rd))}\\
        &\leq C.
    \end{align*}
    Next, it is easy to see that 
    \begin{align*}
        \mathcal B_2 
        &\leq \norm{n}_{L^\infty(0,T;L^\frak p(\Rd))} \norm{\partial_t \Delta V}_{L^1(0,T;L^{\frak p'}(\Rd))}
        \leq C.
    \end{align*}
    Concerning $\mathcal B_3$, we observe, by Assumption \ref{assu:integrable-data}, that  
    \begin{align*}
        \mathcal B_3 
        &\leq C \int_0^T\int_\Rd n (1 + p^s + p^{2s}) \dx x \dx t\\
        &\leq C \norm{n}_{L^\infty(0,T;L^\frak p (\Rd))}(\norm{p}_{L^1(0,T;L^{\frak p'}(\Rd))} + \norm{p}_{L^{1}(0,T;L^{2s\frak{p}'}(\Rd))})\\
        &\qquad +C\norm{n}_{L^1(0,T;L^1(\Rd))}\\
        &\leq C.
    \end{align*}
    Next, we observe that
    \begin{align*}
        \mathcal B_4 
        &\leq \int_0^T \int_\Rd n\prt*{|\nabla W|^2 + \nu |\nabla W| |\Delta \nabla V| + |\nabla V| |\nabla W| +  \nu |\nabla V| |\Delta \nabla V|} \dx x \dx t
        \leq C,
    \end{align*}
    where $C$ depends on the regularity imposed in  Assumptions \ref{assu:existence} and \ref{assu:stiff}, as well as the $L^\frak p$-norm of $n$ and the $L^{\frak q}$-norm of $p$. 
    
    It remains to show that $\mathcal B_5$ can be bounded uniformly in $k$. To this end, we write
    \begin{align*}
        \mathcal B_5 
        &= \int_0^T\int_\Rd n \abs*{K \star \prt*{\nabla p \cdot \nabla (W + V) + \frac k\nu pQ}} \dx x \dx t\\
        &\leq  \mathcal B_{5,1} + \mathcal B_{5,2} + \mathcal B_{5,3},
    \end{align*}
    where
    \begin{align*}
        \mathcal B_{5,1} &:= \frac k\nu \int_0^T\int_\Rd n \abs*{K \star \prt*{ pQ}} \dx x \dx t,\\
        \mathcal B_{5,2} &:= \int_0^T\int_\Rd n \abs*{K \star \prt*{\nabla p \cdot \nabla W}} \dx x \dx t, \\
        \mathcal B_{5,3} &:= \int_0^T\int_\Rd n \abs*{K \star \prt*{\nabla p \cdot \nabla V}} \dx x \dx t.
    \end{align*}
    First, we observe that
    \begin{align*}
        \mathcal B_{5,1} 
        &\leq \frac k\nu \int_0^T\int_\Rd (K\star n) p|Q| \dx x \dx t\\
        &\leq \left(\frac k\nu \int_0^T\int_\Rd p|Q| \dx x \dx t\right) \norm{K\star n}_{L^\infty(0,T;L^\infty(\Rd))}.
    \end{align*}
    By Lemma \ref{lemma:pQinL1}, the term in the parenthesis is bounded and it is enough to verify that $K\star n$ is essentially bounded. 
   
    Indeed,
    $$
        \norm{K\star n}_{L^\infty(0,T;L^\infty(\Rd))} \leq \norm{K}_{L^{\frac{d}{d-2} - }(\Rd)} \norm{n}_{L^{d/2 +}(\Rd)} \leq C,
    $$
    as $\frak p \geq d/2$. 
    Let us proceed by estimating the second term.    
    \begin{align*}
        \mathcal B_{5,2} \leq \int_0^T\int_\Rd n \abs*{\frac{\partial K}{\partial x_i} \star \prt*{p\frac{\partial W}{ \partial x_i}}} \dx x \dx t + \int_0^T\int_\Rd n \abs*{K \star \prt*{p\Delta W}} \dx x \dx t.
    \end{align*}
    The first term on the right-hand side is controlled by
    \begin{align*}
        \int_0^T\int_\Rd & n \abs*{\frac{\partial K}{\partial x_i} \star \prt*{p\frac{\partial W}{ \partial x_i}}} \dx x \dx t\\
        &\leq \norm{|\nabla K| \star n}_{L^1(0,T;L^{\frak p}(\Rd))} \norm{ p \nabla W}_{L^\infty(0,T;L^{\frak p'}(\Rd))}\\
        &\leq \norm{\nabla K}_{L^1(\Rd)} \norm{n}_{L^1(0,T;L^{\frak p}(\Rd))} \norm{p}_{L^\infty(0,T;L^{2\frak p'}(\Rd))} \norm{\nabla W}_{L^\infty(0,T;L^{2 \frak p'}(\Rd))}\\
        &\leq C,
    \end{align*}
    having used the $L^{\frak q}$-regularity of $p$ and $L^\frak p$-regularity of $n$. The second term follows in the same vein, and $\mathcal B_{5,3}$ is even simpler due to the regularity of $V$, by Assumption \ref{assu:existence}.

\bibliography{doku}
\bibliographystyle{plain}

\end{document}